\documentclass[a4paper, reqno]{amsart}
\usepackage[utf8x]{inputenc} 
\usepackage{amsmath}
\usepackage{mathtools}
\usepackage{amssymb}
\usepackage{subcaption}
\usepackage{hyperref} 
\usepackage{bookmark}
\usepackage{amsthm}
\usepackage{float}
\usepackage[absolute,overlay]{textpos}
\usepackage{titletoc}
\usepackage{tikz,xstring}
\usepackage{tikz-cd}
\usepackage[left=2cm, right=2cm]{geometry}
\usepackage{chngcntr}

\newtheorem{mythm}{Theorem}[section]
\newtheorem*{mythm*}{Theorem}
\newtheorem{myprop}[mythm]{Proposition}
\newtheorem*{myprop*}{Proposition}
\newtheorem{mylemma}[mythm]{Lemma}
\newtheorem{mycor}[mythm]{Corollary}
\newtheorem*{mycor*}{Corollary}

\theoremstyle{definition}
\newtheorem{mydef}[mythm]{Definition}
\newtheorem{myrem}[mythm]{Remark}
\newtheorem{myex}[mythm]{Example}
\newtheorem{myconjecture}[mythm]{Conjecture}

\counterwithin*{equation}{mythm}

\newcommand{\C}{\mathbb{C}}
\newcommand{\Z}{\mathbb{Z}}

\newcommand{\K}{\mathbb{K}}

\newcommand{\id}{\mathrm{Id}}
\newcommand{\R}{\mathbb{R}}

\newcommand{\tot}{0}

\newcommand{\totimes}{
\raisebox{-0.2pt}{\includegraphics[scale = 0.6]{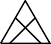}}
}

\newcommand{\dpath}[5]{
\left(\vcenter{\hbox{
\begin{tikzpicture}[scale = 0.9]
	\foreach \z in {1,...,#1} {
		\node at (0.4, \z) {$\z$};
	}
	\renewcommand{\tot}{0}
	\foreach \x [count=\i, evaluate=\i as \j using \i-1] in {#2}{
		\coordinate (A\i) at (\i/2,\x);
		\xdef\tot{\j}
	}
	\foreach \j [evaluate=\j as \k using \j+1] in {1,...,\tot}{
		\draw (A\j) -- node (C\j) {} (A\k);
	}
	\foreach \x in {#3}{
		\filldraw[fill=white] (C\x) circle (0.1);
	}
	\foreach \y in {#4}{
			\filldraw[color=red, fill=white] (C\y) circle (0.1);
	}
	\foreach \a in {#5}{
			\filldraw[fill=black] (C\a) circle (0.1);
	}
\end{tikzpicture}
}
}
\right)
}

\DeclareMathOperator{\coker}{\text{coker}}
\DeclareMathOperator{\im}{\text{Im}}
\DeclareMathOperator{\Hom}{\text{Hom}}

\DeclareMathOperator{\sgn}{\mathrm{sgn}}

\title[Preprojective Algebras of $d$-Representation Finite Species with Relations]{Preprojective Algebras of $d$-Representation Finite Species with Relations}
\author{Christoffer S\"oderberg}

\begin{document}
\begin{abstract}
	In this article we study the properties of preprojective algebras of representation finite species. To understand the structure of a preprojective algebra, one often studies its Nakayama automorphism. A complete description of the Nakayama automorphism is given by Brenner, Butler and King when the algebra is given by a path algebra. We generalize this result to the species case.
	
	We show that the preprojective algebra of a representation finite species is an almost Koszul algebra. With this we know that almost Koszul complexes exist. It turns out that the almost Koszul complex for a representation finite species is given by a mapping cone of a chain map, which is homogeneous of degree $1$ with respect to a certain grading. We also study a higher dimensional analogue of representation finite hereditary algebras called $d$-representation finite algebras. One source of $d$-representation finite algebras comes from taking tensor products. By introducing a functor called the Segre product, we manage to give a complete description of the almost Koszul complex of the preprojective algebra of a tensor product of two species with relations with certain properties, in terms of the knowledge of the given species with relations. This allows us to compute the almost Koszul complex explicitly for certain species with relations more easily.
\end{abstract}

\maketitle

\tableofcontents

\section{Introduction}\label{Section - Introduction}
In the paper \cite{gabriel1973indecomposable} Gabriel gathered results regarding the problem to determine which algebras were of finite representation type. One of these results, known as Gabriel's Theorem \cite{gabriel1972unzerlegbare}, says that a path algebra over an algebraically closed field is representation finite if and only if its underlying diagram is a disjoint union of Dynkin diagrams of type ADE. Gabriel also sketched the proof in \cite{gabriel1973indecomposable} of how one could extend one direction of Gabriel's Theorem to $\K$-species when $\K$ is a perfect field, i.e. a $\K$-species is representation finite if its diagram is a disjoint union of Dynkin diagrams of type ADE. The general case was later solved completely by Dlab and Ringel in \cite{dlab1975algebras}. They proved that a species is representation finite if and only if its diagram is a finite disjoint union of Dynkin diagrams of type ABCDEFG. Thus by Dlab and Ringel we have a complete set of representation finite species. Given a hereditary finite dimensional $\K$-algebra $\Lambda$ we can consider its preprojective algebra $\Pi(\Lambda)$. It is a graded algebra with degree $0$ equal to $\Lambda$ and as a $\Lambda$-module it gives an additive generator of the category of preprojective $\Lambda$-modules. Moreover, $\Lambda$ is representation finite if and only if $\Pi(\Lambda)$ is a finite dimensional self-injective algebra. In this paper we study the properties of preprojective algebras for representation finite species, such as the Nakayama permutation, Nakayama automorphism and Koszul properties of the preprojective algebra. 

The study of the preprojective algebra of a quiver shows up in all kind of topics in mathematics such as cluster algebras \cite{geiss2013cluster}, quiver varieties \cite{nakajima1994instantons}, quantum groups \cite{lusztig1991quivers, kashiwara1996geometric}, and many more topics. This is a motivation to study the Koszul properties of the preprojective algebra $\Pi(S)$ of some representation finite species $S$. In the case when $S$ is not representation finite, the Koszul properties of $\Pi(S)$ are studied in \cite{athaiokrmy}. When $S$ is representation finite, $\Pi(S)$ is a Frobenius algebra, and so we have that $D\Pi(S) \cong \Pi(S)$ as left $\Pi(S)$-modules. This isomorphism yields an isomorphism of $\Pi(S)$-$\Pi(S)$-bimodules by introducing a twist on $\Pi(S)$ denoted by $\gamma$ and it is called the Nakayama automorphism. The paper \cite{brenner2002periodic} gave a complete description of the Nakayama automorphism $\gamma$ of the preprojective algebra of a path algebra of a Dynkin quiver. Later Grant \cite{grant2019nakayama} gave a partial description of the Nakayama automorphism of the preprojective algebra of any hereditary $\K$-algebra using the Nakayama permutation $\sigma$ and the Auslander-Reiten translation. In this paper we compute $\sigma$ for all Dynkin diagrams in Theorem~\ref{Theorem - Nakayama permutation} so we can use Grant's description to give an explicit description of $\gamma$. This is formulated in the following theorem.

\begin{mythm*}(Theorem~\ref{Theorem - description of nakayama automorphism true})
	Let $S$ be a species, with Dynkin diagram $\Delta$, over division algebras $F\subset G$. The Nakayama automorphism $\gamma$ of $\Pi(S)$ is given by
	\begin{equation*}
		\gamma(y_\alpha^i) = \begin{cases}
			y^i_{\sigma(\alpha)}, & \mbox{if }\alpha\in Q_1 \\
			\sgn(\sigma(\alpha)) y^i_{\sigma(\alpha)}, & \mbox{if }\alpha\not\in Q_1
		\end{cases}.
	\end{equation*}
\end{mythm*}

The proof of the classification of representation finite species by Dlab and Ringel shows that every representation finite species is isomorphic to a species, whose underlying diagram is a Dynkin diagram, over two division algebras $F\subset G$, therefore this theorem applies to all representation finite species. This sums up the first part of this paper.

In the second part of the paper we change focus from studying tensor algebras of species to algebras given by species with relations. More specifically, we study so called $d$-representation finite algebras. We say that a finite dimensional $\K$-algebra $\Lambda$ with $\mathrm{gldim}\Lambda \le d$ is a $d$-representation finite algebra if there exists a $d$-cluster tilting $\Lambda$-module. A hereditary representation finite algebra is $1$-representation finite. For a $d$-representation finite algebra $\Lambda$, there exists a positive integer $l_P$ for each indecomposable projective $\Lambda$-module $P$ such that $\tau_d^{-(l_P-1)}P$ is an indecomposable injective $\Lambda$-module. Here $\tau_d$ is the higher dimensional analogue of the Auslander-Reiten translation $\tau$. When $l_P=l$ for all indecomposable projective $\Lambda$-modules $P$ we say that $\Lambda$ is $l$-homogeneous. If $\Lambda_i$ is a $d_i$-representation finite $l$-homogeneous algebra for each $i\in\{1, 2\}$ and $\K$ is a perfect field, then $\Lambda_1\otimes_\K \Lambda_2$ is $l$-homogeneous and $(d_1+d_2)$-representation finite \cite[Corollary 1.5]{herschend2011n}. For the case when $d_1 = d_2$, McMahon and Williams showed that the $(2d+1)$-preprojective algebra of $\Lambda_1\otimes_\K \Lambda_2$ has a $d$-precluster tilting module \cite{MR4301041}. In general, $\Lambda_1\otimes_\K\Lambda_2$ is not $(d_1 + d_2)$-representation finite if we drop the assumption that $\Lambda_1$ and $\Lambda_2$ are $l$-homogeneous. Relaxing the assumptions on $\Lambda_1$ and $\Lambda_2$ Pasquali proved that if $\Lambda_i$ is an acyclic $d_i$-complete algebra for each $i\in \{1, 2\}$ then $\Lambda_1\otimes_\K \Lambda_2$ is $(d_1 + d_2)$-complete \cite{pasquali2019tensor}. If $\Lambda_i$ is a $d_i$-representation infinite algebra such that $\Lambda_i/J_i$ is semi-simple for each $i\in \{1, 2\}$, then $\Lambda_1\otimes_\K \Lambda_2$ is a $(d_1 + d_2)$-representation infinite algebra \cite[Theorem 2.10]{herschend2014n}. The preprojective algebra of such algebras were studied in \cite{MR4064763}. In this paper we will further study the former case, that $\Lambda_i$ is a $d_i$-representation finite $l$-homogeneous algebra, but we also assume that $\Lambda_i$ is a Koszul algebra for each $i\in \{1, 2\}$ and focus on the Koszul properties of the preprojective algebra of $\Lambda_1\otimes_\K \Lambda_2$.

Let $\Pi(\Lambda)$ be the $(d+1)$-preprojective algebra introduced by Iyama and Oppermann in \cite{iyama2013stable}. Iyama and Oppermann proved that $\Pi(\Lambda)$ is self-injective if $\Lambda$ is $d$-representation finite. The $(d+1)$-preprojective algebra $\Pi(\Lambda)$ was further studied by Grant and Iyama in \cite{grant2020higher} and they showed that when $\Lambda$ is a $d$-representation finite Koszul algebra, the $(d+1)$-preprojective algebra is an almost Koszul algebra in the sense of \cite{brenner2002periodic}. In this setting the homologies of the almost Koszul complex for $\Pi(\Lambda)$ can be described using the Nakayama automorphism. We investigate various properties for $\Pi(\Lambda_1)$ and $\Pi(\Lambda_2)$ that are related to the same properties for $\Pi(\Lambda_1\otimes \Lambda_2)$. For example, we investigate how one can describe the almost Koszul complexes of $\Pi(\Lambda_1\otimes_\K \Lambda_2)$ using the knowledge of the almost Koszul complexes of $\Pi(\Lambda_1)$ and $\Pi(\Lambda_2)$. This is formulated in the following theorem.

\begin{mythm*}(Theorem~\ref{Theorem - Koszul complex for tensor product of species})
	Let $\K$ be a perfect field. Let $\Lambda_i$ be an acyclic $d_i$-representation finite $l$-homogeneous Koszul algebra for each $i\in \{1, 2\}$. Let $S^{\Lambda_i}$ be a simple $\Pi(\Lambda_i)$-module and let $\varphi^{\Lambda_i}: Q_\bullet^{\Lambda_i}\rightarrow R_\bullet^{\Lambda_i}$ be an almost quasi-isomorphism as in Theorem~\ref{Theorem - Koszul complex is given by a mapping cone} such that $C(\varphi^{\Lambda_i})$ is the almost Koszul complex for $S^{\Lambda_i}$. The complex	
	\begin{equation*}
		C(\mathrm{Tot}(\varphi^{\Lambda_1}\totimes \varphi^{\Lambda_2}):\mathrm{Tot}(Q^{\Lambda_1}_\bullet\totimes_\K Q^{\Lambda_2}_\bullet) \rightarrow \mathrm{Tot}(R^{\Lambda_1}_\bullet \totimes_\K R^{\Lambda_2}_\bullet))
	\end{equation*}
	is the almost Koszul complex for $S^{\Lambda_1}\otimes_\K S^{\Lambda_2}\in \Pi(\Lambda_1\otimes_\K \Lambda_2)\mathrm{-mod}$.
\end{mythm*}

Here, the $C(\varphi)$ denotes the mapping cone of $\varphi$ and $\totimes_\K$ is the Segre product defined in Section~\ref{Section - Diagonal Tensor Product}. We define almost quasi-isomorphisms in Section~\ref{Section - Koszul Complex for Higher Species} and we define acyclic algebras in Section~\ref{Section - Preliminaries}. This theorem allows us to have total control of the almost Koszul complexes in $\Pi(\Lambda_1\otimes_\K \Lambda_2)$, and thus one of the corollaries is the following.

\begin{mycor*}(Corollary~\ref{Corollary - Pi(AxB) is almost Koszul})
	Let $\K$ be a perfect field. Let $\Lambda_i$ be an acyclic $d_i$-representation finite $l$-homogeneous Koszul algebra for each $i\in \{1, 2\}$. If $\Pi(\Lambda_i)$ is an $(p_i, q_i)$-almost Koszul algebra, then $\Pi(\Lambda_1\otimes_\K \Lambda_2)$ is an $(p_1 + p_2 - l + 1, q_1 + q_2 - 1)$-almost Koszul algebra.
\end{mycor*}

Note that given the assumptions in the above corollary $\Lambda_1\otimes_\K\Lambda_2$ is a $(d_1 + d_2)$-representation finite $l$-homogeneous Koszul algebra and therefore we can apply the corollary iteratively. Moreover, for a species $S$ over a Dynkin diagram $\Delta$ with Coxeter number $h$, the preprojective algebra $\Pi(S)$ is a $(h-2, 2)$-almost Koszul algebra (Corollary~\ref{Corollary - Pi(S) is an almost Koszul algebra}).

The article has the following structure. In Section~\ref{Section - Preliminaries} we introduce preliminary concepts and notations. In Section~\ref{Section - Properties for Representation Finite Species} we show that the Nakayama permutation only depends on the underlying diagram and we compute the Nakayama permutation for all Dynkin diagrams of type ABCDEFG. In Section~\ref{Section - The Preprojective Algebra} we compare different descriptions of the preprojective algebra. In Section~\ref{Section - Nakayama Automorphism} we give a description of the Nakayama automorphism for species of Dynkin type BCFG, and we extend the results from \cite{brenner2002periodic} to species of Dynkin type ADE. In Section~\ref{Section - Koszul algebras} we extend results from \cite{brenner2002periodic}. In \cite{brenner2002periodic} it is shown that the preprojective algebra of a path algebra is almost Koszul if it is of Dynkin type ADE and otherwise it is Koszul. We partially extend the former result to the species case, i.e. the preprojective algebra of a species is almost Koszul if the diagram is a Dynkin diagram of type ABCDEFG. In Section~\ref{Section - Higher preprojective algebra} we introduce $d$-representation finite algebras together with the $(d+1)$-preprojective algebra, also known as the higher preprojective algebra. In Section~\ref{Section - $n$-Representation Finite Species} we investigate various properties of the tensor product of algebras relevant to the paper. In Section~\ref{Section - Diagonal Tensor Product} we define the Segre product of graded algebras and modules and state some basic properties and the Künneth formula for the Segre product. In Section~\ref{Section - Koszul Complex for Higher Species} we investigate the structure of the almost Koszul complexes. In Section~\ref{Section - Examples} we finish the paper by computing some examples to illustrate how every main theorem comes together to yield concrete results. \\

\noindent\textbf{Acknowledgements.} The author wishes to thank his supervisor Martin Herschend for all the support during the writing of this article. The author would also like to thank Joseph Grant and Daniel Simson for their helpful comments.

\section{Preliminaries}\label{Section - Preliminaries}
Let $\K$ be a field. We first define species as it is given in \cite{berg2011structure, ringel1976representations, gabriel1973indecomposable}.

\begin{mydef}\label{Definition - Species}(Species)
	Let $Q$ be a finite quiver. A species $S=(D_i, M_\alpha)_{i\in Q_0, \alpha\in Q_1}$ is a collection of division rings $D_i$ and $M_\alpha\in D_j$-$D_i$-$\mathrm{mod}$, where $\alpha:i\rightarrow j$, such that $\mathrm{Hom}_{D_i^{op}}(M_\alpha, D_i)\cong \mathrm{Hom}_{D_j}(M_\alpha, D_j)$ as $D_j$-$D_i$-modules. We say that a species $S$ is a $\mathbb{K}$-species if all $D_i$ are finite dimensional over a common central subfield $\mathbb{K}$ and all $M_\alpha$ are finite dimensional over $\mathbb{K}$ satisfying $\lambda m = m\lambda$ for all $m\in M_\alpha, \lambda\in\mathbb{K}$. 
\end{mydef}

\begin{mydef}
	For a species $S$ let $D=\bigoplus_{i\in Q_0}D_i$ and $M=\bigoplus_{\alpha\in Q_1}M_\alpha\in D$-$D\mathrm{-mod}$. We define the tensor algebra $T(S)$ to be the tensor ring $T(D, M)$. More explicitly,
	\begin{equation*}
		T(S) = T(D, M) = D\oplus \bigoplus_{k\ge 1} M^{\otimes_D k}.
	\end{equation*}
\end{mydef}

\begin{myrem}\label{Remark - General definition of a species}
	In some contexts it is useful to allow for a more general definition of a species by requiring that $D$ has to be Morita equivalent to a sum of division rings, and not necessarily a division ring. In Section~\ref{Section - Examples} we will briefly discuss the fact that tensor products of tensor algebras of species is not necessarily a tensor algebra over a species by Definition~\ref{Definition - Species}, but it will be a species according to this more general definition. Another generalization of species called Pro-species is studied in \cite{kulshammer2017pro, li2015representations}.
\end{myrem}

In this paper we assume that all our species are $\K$-species. Since $T(S)$ is a tensor algebra, $T(S)$ has a natural $\Z$-grading which we will refer to as the path length grading. Also note that if $S$ is a $\K$-species then $T(S)$ is a $\K$-algebra. One could relax the assumption on the species from being a $\K$-species to be a polynomial identity species, i.e. when the division algebras $D_i$ are polynomial identity algebras. In this setting, it does not have to exist a common subfield $\K$ such that all $D_i$ are finitely generated over $\K$. Hence we cannot use \cite[Theorem B]{dlab1975algebras}. Therefore, for convenience, we assume that all species are $\K$-species.

We denote $e_i$ the identity element in $D_i$. The set $\{e_1, \dots, e_{|Q_0|}\}$ is a complete set of pairwise orthogonal primitive idempotents for $T(S)$.

If $Q$ is a quiver with multiple arrows $\alpha, \alpha': i\to j$ we consider a modified quiver $Q'$ which is obtained by replacing $\alpha$ and $\alpha'$ with $\beta: i\to j$. Then if $S$ is a species over $Q$, we can modify $S$ to a species $S'$ over $Q'$ by setting $M_\beta = M_\alpha\oplus M_{\alpha'}$. We see that
\begin{equation*}
	T(S) = T(D, M) = T(D, M') = T(S').
\end{equation*}
Therefore we assume that $Q$ has no multiple arrows.

A simple example of a species would be $S=(D_i, M_\alpha)_{i\in Q_0, \alpha\in Q_1}$ where $Q: 1\rightarrow 2 \rightarrow 3$ and
\begin{equation*}
	\begin{aligned}
		D_i &= \begin{cases}
			\C, & \mbox{if } i=1 \\
			\R, & \mbox{if } i=2,3
		\end{cases} \\
		M_\alpha &= \begin{cases}
			\C, & \mbox{if }\alpha:1\rightarrow 2 \\
			\R, & \mbox{if }\alpha:2\rightarrow 3
		\end{cases}.
	\end{aligned}
\end{equation*}
We often write $S$ as a decorated quivers given by decorated subquivers of the form
\begin{equation*}
	D_{s(\alpha)}\xrightarrow{M_\alpha}D_{t(\alpha)}
\end{equation*}
for all $\alpha\in Q_1$. So in our example above we would write
\begin{equation*}
	\C \xrightarrow{\C} \R \xrightarrow{\R}\R.
\end{equation*}

When $Q$ is a finite and acyclic quiver and $S$ a species over $Q$, the tensor algebra $(S)$ is a basic finite dimensional $\K$-algebra. Moreover, since $\mathrm{rad}(T(S)) \cong \bigoplus_{\substack{\alpha\in Q_1 \\ s(\alpha) = i}} (T(S)e_{t(\alpha)})^{\dim_{D_{t(\alpha)}}M_\alpha}$, and thus projective, we can apply \cite[Theorem 2.35]{lam2012lectures} to see that $T(S)$ is indeed hereditary.

\begin{mydef}\label{Definition - Diagram of a species}
	Let $S$ be a species over an acyclic quiver $Q$. Then the diagram $\Delta$ of $S$ is defined to have its vertices as $Q_0$ and for every $\alpha\in Q_1$ we have an edge with valuation
	\begin{equation*}
		(\dim_{D_{s(\alpha)}}(M_\alpha), \dim_{D_{t(\alpha)}}(M_\alpha)).
	\end{equation*}
	If there is an edge with valuation $(1, k)$ then we write $s(\alpha)\xRightarrow{(k)}t(\alpha)$, if we have the valuation $(k, 1)$ then we write $t(\alpha)\xRightarrow{(k)}s(\alpha)$ and if we have the valuation $(1, 1)$ we simply write the edge without valuation, i.e. $\begin{tikzcd}[column sep = 13]
		s(\alpha) \arrow[r, -] & t(\alpha).
	\end{tikzcd}$
\end{mydef}

\begin{mydef}
	A representation $V=(V_i, \phi_\alpha)$ over $S$ is a collection of $D_i$-modules $V_i$ and $D_{t(\alpha)}$-module morphisms
	\begin{equation*}
	\phi_\alpha: M_\alpha\otimes_{V_s(\alpha)} V_{s(\alpha)}\rightarrow V_{t(\alpha)}.
	\end{equation*}
	A morphism $f: V\rightarrow V'$ between two representations $V$ and $V'$ over a species $S$ is given by a family of morphisms $f_i: V_i\rightarrow V_i'$ such that
	\begin{equation*}
		\begin{tikzcd}
		M_\alpha\otimes_{s(\alpha)}V_{s(\alpha)} \arrow[r, "1\otimes f_{s(\alpha)}"] \arrow[d, "\phi_\alpha"] & M_\alpha\otimes_{s(\alpha)}V_{s(\alpha)}' \arrow[d, "\phi'_\alpha"] \\
		V_{t(\alpha)} \arrow[r, "f_{t(\alpha)}"] & V_{t(\alpha)}'
		\end{tikzcd}
	\end{equation*}
	commutes for all $\alpha\in Q_1$. The set of representations over $S$ together with the morphism forms an abelian category and is denoted by $\mathfrak{Rep}S$. The full subcategory of $\mathfrak{Rep}S$ consisting of representations $V$ where $\dim_\mathbb{K}V_i<\infty$ for all $i\in Q_0$ is denoted by $\mathfrak{rep}S$.
\end{mydef}

In this paper we will focus on the module category $T(S)\mathrm{-mod}$, which is the category consisting of finitely generated $T(S)$-modules, and usually when we relate representations over $S$ and modules over $T(S)$ we use the following result.

\begin{myprop}\cite[Proposition 10.1]{dlab1975algebras}
	Let $S$ be a species. The category $\mathfrak{Rep}S$ is equivalent to the category $T(S)\mathrm{-Mod}$.
\end{myprop}

Here $T(S)\mathrm{-Mod}$ is the category consisting of all $T(S)$-modules. When we restrict ourself to a finite quiver we have the following.

\begin{mycor}\cite[Corollary 2.2]{berg2011structure}
	Let $S$ be a species over $Q$, where $Q$ is a finite quiver. The category $\mathfrak{rep}S$ is equivalent to the category $T(S)\mathrm{-mod}$.
\end{mycor}

We say that a species $S$ is representation finite if there is only a finite number of finitely generated indecomposable $T(S)$-modules up to isomorphism. Gabriel characterized all representation finite species $S$ of the form $D_i = M_\alpha = F$ for some field $F$ in \cite{gabriel1980auslander}. Gabriel's result was extended by Dlab and Ringel in the following theorem.

\begin{mythm}\cite[Theorem B]{dlab1975algebras})
	A species $S$ is representation finite if and only if $\Delta$ is a finite disjoint union of Dynkin diagrams.
\end{mythm}

The Dynkin diagrams are given in Figure~\ref{Figure - Dynkin Diagrams 1}.

\begin{figure}[h]
	\vspace{-0.5em}\begin{equation*}
		\begin{aligned}
		&\begin{tikzcd}
		(A_n) & 1 \arrow[-, r] & 2 \arrow[r, -] & \cdots \arrow[r, -] & n-1 \arrow[r, -] & n \\
		\end{tikzcd} \\[-2em]
		&\begin{tikzcd}
		(B_n) & 1 & 2 \arrow["(2)"', l, Rightarrow] \arrow[r, -] & \cdots \arrow[r, -] & n-1 \arrow[r, -] & n \\
		\end{tikzcd} \\[-2em]
		&\begin{tikzcd}
		(C_n) & 1 \arrow["(2)", r, Rightarrow] & 2 \arrow[r, -] & \cdots \arrow[r, -] & n-1 \arrow[r, -] & n \\
		\end{tikzcd} \\[-2em]
		&\begin{tikzcd}[row sep = 0em]
		& 1 \\
		(D_n) & & 3 \arrow[-, r] \arrow[lu, -] \arrow[ld, -] & \cdots \arrow[r, -] & n-1 \arrow[r, -]& n \\
		&2 \\
		\end{tikzcd} \\[-0.5em]
		&\begin{tikzcd}[row sep=0em]
		& 1 \arrow[r, -] & 2 \arrow[r, -] & 3 \arrow[r, -] \arrow[dd, -] & 4 \arrow[r, -] & 5 \\
		(E_6) \\
		& & & 6 \\
		\end{tikzcd} \\[-0.3em]
		&\begin{tikzcd}[row sep=0em]
		& 1 \arrow[r, -] & 2 \arrow[r, -] & 3 \arrow[r, -] \arrow[dd, -] & 4 \arrow[r, -] & 5 \arrow[r, -] & 6 \\
		(E_7) \\
		& & & 7 \\
		\end{tikzcd} \\[-0.3em]
		&\begin{tikzcd}[row sep=0em]
		& 1 \arrow[r, -] & 2 \arrow[r, -] & 3 \arrow[r, -] \arrow[dd, -] & 4 \arrow[r, -] & 5 \arrow[r, -] & 6 \arrow[r, -] & 7 \\
		(E_8) \\
		& & & 8 \\
		\end{tikzcd} \\[-0.3em]
		&\begin{tikzcd}
		(F_4) & 1 \arrow[r, -] & 2 \arrow["(2)", r, Rightarrow] & 3 \arrow[r, -] & 4 \\
		\end{tikzcd} \\[-2em]
		&\begin{tikzcd}
		(G_2) & 1 \arrow["(3)", r, Rightarrow] & 2
		\end{tikzcd}
		\end{aligned}
	\end{equation*}
	\caption{Dynkin Diagrams}\label{Figure - Dynkin Diagrams 1}
\end{figure}
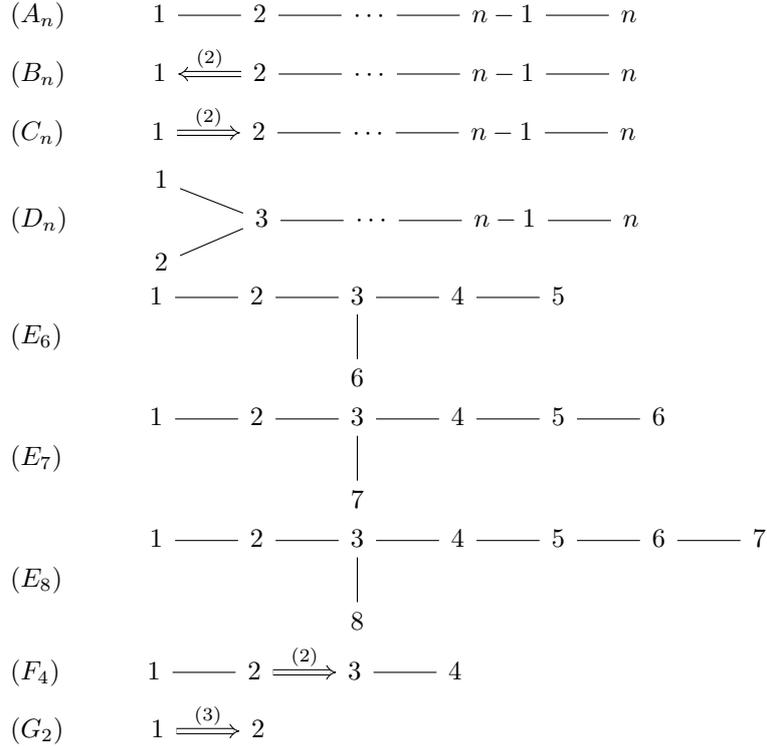

Given a species $S$ where $\Delta$ is a Dynkin diagram. Then reading from the proof of \cite[Theorem B]{dlab1975algebras} we can assume that $S$ is of the form where $D_i, M_\alpha\in \{F, G\}$, where $F$ and $G$ are division algebras satisfying $F\subset G$. If $\Delta$ is of type ADE, we call $\Delta$ simply laced, then $F=G$. When $\Delta$ is non-simply laced, the valuation $k = \dim_\K G/\dim_\K F$. More explicitly, each case is described in Figure~\ref{Figure - Dynkin Diagrams 2}. In this description we did not specify the orientation of $Q$. This is due to the fact that the orientation of $Q$ only determines the module structure on all $M_\alpha$'s. 

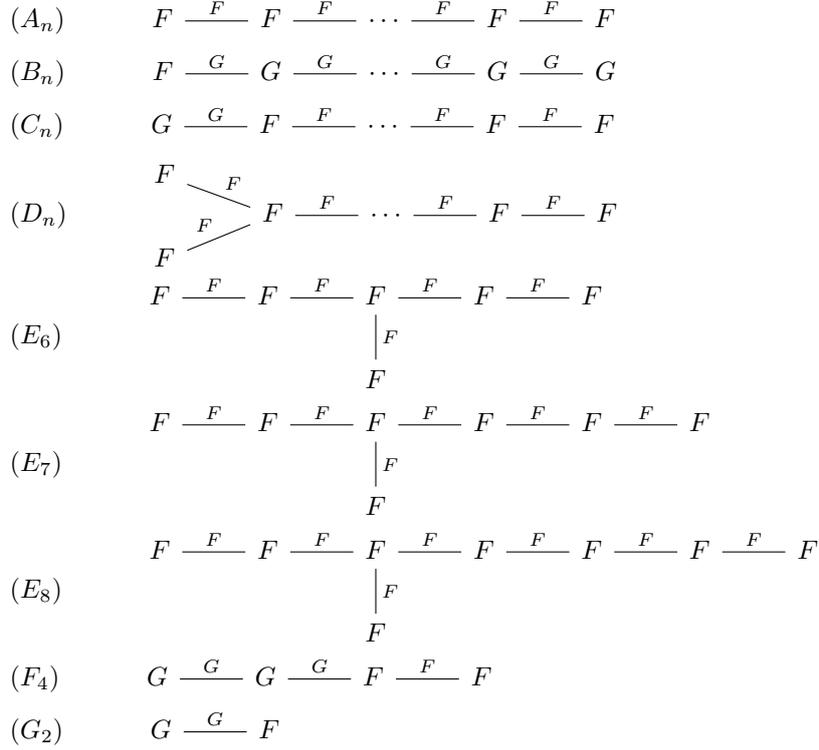
\begin{figure}[h]
	\vspace{-0.5em}\begin{equation*}
		\begin{aligned}
			&\begin{tikzcd}
			(A_n) & F \arrow["F", -, r] & F \arrow["F", r, -] & \cdots \arrow["F", r, -] & F \arrow["F", r, -] & F \\
			\end{tikzcd} \\[-2em]
			&\begin{tikzcd}
			(B_n) & F & G \arrow["G"', -, l] \arrow["G", r, -] & \cdots \arrow["G", r, -] & G \arrow["G", r, -] & G \\
			\end{tikzcd} \\[-2em]
			&\begin{tikzcd}
			(C_n) & G \arrow["G", -, r] & F \arrow["F", r, -] & \cdots \arrow["F", r, -] & F \arrow["F", r, -] & F \\
			\end{tikzcd} \\[-2em]
			&\begin{tikzcd}[row sep = 0em]
			& F \\
			(D_n) & & F \arrow["F", -, r] \arrow["F"', lu, -] \arrow["F"', ld, -] & \cdots \arrow["F", r, -] & F \arrow["F", r, -]& F \\
			&F \\
			\end{tikzcd} \\[-0.5em]
			&\begin{tikzcd}[row sep=0em]
			& F \arrow["F", r, -] & F \arrow["F", r, -] & F \arrow["F", r, -] \arrow["F", dd, -] & F \arrow["F", r, -] & F \\
			(E_6) \\
			& & & F \\
			\end{tikzcd} \\[-0.3em]
			&\begin{tikzcd}[row sep=0em]
			& F \arrow["F", r, -] & F \arrow["F", r, -] & F \arrow["F", r, -] \arrow["F", dd, -] & F \arrow["F", r, -] & F \arrow["F", r, -] & F \\
			(E_7) \\
			& & & F \\
			\end{tikzcd} \\[-0.3em]
			&\begin{tikzcd}[row sep=0em]
			& F \arrow["F", r, -] & F \arrow["F", r, -] & F \arrow["F", r, -] \arrow["F", dd, -] & F \arrow["F", r, -] & F \arrow["F", r, -] & F \arrow["F", r, -] & F \\
			(E_8) \\
			& & & F \\
			\end{tikzcd} \\[-0.3em]
			&\begin{tikzcd}
			(F_4) & G \arrow["G", r, -] & G \arrow["G", r, -] & F \arrow["F", r, -] & F \\
			\end{tikzcd} \\[-2em]
			&\begin{tikzcd}
			(G_2) & G \arrow["G", r, -] & F
			\end{tikzcd}
		\end{aligned}
	\end{equation*}
	\caption{Description of species over Dynkin diagrams}\label{Figure - Dynkin Diagrams 2}
\end{figure}

For a species $S$ the Auslander-Reiten quiver $\Gamma_S$ is defined by setting the vertices to be isomorphism classes of indecomposable objects in $T(S)-\mathrm{mod}$. Given two indecomposable objects in $X, Y\in T(S)\mathrm{-mod}$, then there is an arrow $[X]\xrightarrow{d_{XY}}[Y]$ in $\Gamma_S$ if
\begin{equation*}
	\mathrm{rad}(X, Y)/\mathrm{rad}^2(X, Y)\not=0
\end{equation*}
and the valuation is given by
\begin{equation*}
	d_{XY} = \dim_\mathrm{\K}(\mathrm{rad}(X, Y)/\mathrm{rad}^2(X, Y)).
\end{equation*}
This construction is introduced in \cite{auslander1997representation} with a slight modification. In \cite{auslander1997representation} the valuation $(a, b)$ indicates that there is a minimal right almost split morphism $X^a\oplus M\rightarrow Y$ where $X$ is not a summand of $M$, and a minimal left almost split morphism $X\rightarrow Y^b\oplus N$ where $Y$ is not a summand of $N$. To see the connection between these two valuations, we introduce the following.

\begin{mydef}
	Let $M\in \Lambda\mathrm{-mod}$ be indecomposable for some $\K$-algebra $\Lambda$. Then we define
	\begin{equation*}
		\delta(M) = \dim_\K (\mathrm{End}_\Lambda(M) / \mathrm{rad}_\Lambda(M)).
	\end{equation*}
\end{mydef}

\begin{myrem}
	Let $S$ be a species where $\Delta$ is a Dynkin diagram. Then
	\begin{equation*}
		\delta(P_i) = \dim_\K (D_i)=\begin{cases}
			\dim_\K(F), & \mbox{if }D_i = F \\
			\dim_\K(G), & \mbox{if }D_i = G
		\end{cases}.
	\end{equation*}
\end{myrem}

Now $(a, b) = (d_{XY}/\delta(X), d_{XY}/\delta(Y))$. 

We define a functor $\tau = D\circ \mathrm{Tr}$, where $D = \mathrm{Hom}_\mathbb{K}(-, \mathbb{K})$ denotes the $\mathbb{K}$-dual and $\mathrm{Tr}$ is the Auslander-Bridger transpose \cite[Chapter 4.1]{auslander1997representation}. By \cite[Proposition 1.9]{auslander1997representation} there is an equivalence
\begin{equation*}
	\begin{tikzcd}
		T(S)\mathrm{-}\underline{\mathrm{mod}} \arrow[r, shift left, "\tau"] & T(S)\mathrm{-}\overline{\mathrm{mod}} \arrow[l, shift left, "\tau^{-1}"].
	\end{tikzcd}
\end{equation*}
If $S$ is a species over an acyclic quiver $Q$, then $T(S)$ is hereditary and $\tau = D\mathrm{Ext}^1_{T(S)}(-, T(S))$ is even defined on the category $T(S)\mathrm{-mod}$, and  the inverse is given by $\tau^{-1} = \mathrm{Ext}^1_{T(S)}(DT(S), T(S))\otimes_{T(S)} -$. We call $\tau$ the Auslander-Reiten translation of $\Gamma_S$.

Recall that for every species we have a complete set of orthogonal idempotents $\{e_1, \dots, e_n\}$ induced by the identity elements in $D_1, D_2, \dots, D_{|Q_0|}$. With these idempotents we can write down the indecomposable projective modules $P_i$ and the indecomposable injective modules $I_i$ as 
\begin{equation*}
	P_i = T(S)e_i, \quad I_i = D(e_iT(S)).
\end{equation*}

For a species $S$ we define the preprojective component of $\Gamma_S$ to be the full subquiver of $\Gamma_S$ consisting of all $X\in (\Gamma_S)_0$ such that $\tau^a X$ is projective for some $a\in \Z_{\ge 0}$. Similarly, we define the preinjective component of $\Gamma_S$ to be the full subquiver of $\Gamma_S$ consisting of all $Y\in (\Gamma_S)_0$ such that $\tau^{-b}Y$ is injective for some $b\in \Z_{\ge 0}$.

\begin{myprop}\label{Proposition - nakayama permutation and integer l_i}
	Let $S$ be a species with a Dynkin diagram $\Delta$. Then there exists a permutation $\sigma: Q_0\rightarrow Q_0$ and integers $l_i$ such that
	\begin{equation*}
		P_{\sigma(i)} = \tau^{l_i - 1}I_i.
	\end{equation*}
\end{myprop}

\begin{proof}
	Since $S$ is representation finite $\Gamma_S$ has finitely many vertices. Therefore, by \cite[Theorem 2.1]{auslander1997representation}, the preprojective component of $\Gamma_S$ coincides with the preinjective component of $\Gamma_S$.
\end{proof}

\begin{mydef}
	Let $S$ be a species with Dynkin diagram $\Delta$. We call $\sigma$ in Proposition~\ref{Proposition - nakayama permutation and integer l_i} the Nakayama permutation of $\Delta$. Moreover, if $l_i=l$ for some integer $l\in \Z$, then we call $S$ $l$-homogeneous.
\end{mydef}

With these integers $l_i$ we can explicitly describe the Auslander-Reiten quiver for representation finite species. Let $S$ be a representation finite species over a quiver $Q$. Then by \cite[Proposition 1.15]{auslander1997representation}, the projective modules in $\Gamma_S$ form a subquiver which is isomorphic to the opposite of $Q$. In fact, the quiver $Q$, the valuations in $\Delta$ and the length of each $\tau$-orbit (i.e. the numbers $l_i$) is enough information to write down $\Gamma_S$. This follows from \cite[Theorem 2.1]{auslander1997representation} together with the fact that $\Gamma_S$ is the preprojective component of $\Gamma_S$.

We also introduce the Coxeter transformation. Since the Coxeter transformation is defined on the Grothendieck group we first define the Grothendieck group.

\begin{mydef}
	Let $S$ be a species. We define the Grothendieck group $\textbf{K}_0(T(S)\mathrm{-mod}) = [T(S)\mathrm{-mod}] / R$, where $[T(S)\mathrm{-mod}]$ is the free abelian group on the isomorphism classes $[M]$ of finitely generated $T(S)$-modules $M$ and $R$ is the subgroup generated by expressions $[A] + [C] - [B]$ whenever there is an exact sequence $0\rightarrow A\rightarrow B\rightarrow C\rightarrow 0$ of $T(S)$-modules.
\end{mydef}

First we note that $\{[D_i]\}_{i=0}^{|Q_0|}$ is a basis for $\textbf{K}_0(T(S)\mathrm{-mod})$ by \cite[Chapter 1, Theorem 1.7]{auslander1997representation}. We also have two other basis given by $\{[P_i]\}_{i=0}^{|Q_0|}$ and $\{[I_i]\}_{i=0}^{|Q_0|}$ by \cite[Chapter 8, Lemma 2.1]{auslander1997representation}.

\begin{mydef}
	Let $S$ be a species. We define the Coxeter transformation $c: \textbf{K}_0(T(S)\mathrm{-mod})\rightarrow \textbf{K}_0(T(S)\mathrm{-mod})$ by $[P_i]\mapsto -[I_i]$.
\end{mydef}

We call a non-zero element in $\textbf{K}_0(T(S)\mathrm{-mod})$ positive, respectively negative, if the coordinates are greater than or equal to $0$, respectively less than or equal to $0$ in the basis $\{[D_i]\}_{i=0}^{|Q_0|}$. For hereditary $\K$-algebras the Coxeter transformation is closely related to the Auslander-Reiten translation, and since $T(S)$ is a hereditary $\K$-algebra when $Q$ is finite and acyclic we have the following proposition.

\begin{myprop}\label{Proposition - Coxeter transformation properties}\cite[Chapter 8, Proposition 2.2]{auslander1997representation}
	Let $S$ be a species over a finite acyclic quiver $Q$. We have the following.
	\begin{enumerate}
		\item If $M\in T(S)\mathrm{-mod}$ is a non-projective indecomposable module, then $c[M] = [\tau M]$.
		\item Let $M$ be an indecomposable $T(S)$-module. Then $M$ is projective if and only if $c[M]$ is negative.
		\item If $M$ is an indecomposable module, then $c[M]$ is either positive or negative.
		\item If $M\in T(S)\mathrm{-mod}$ is a non-injective indecomposable module, then $c^{-1}[M] = [\tau^{-1} M]$.
		\item Let $M$ be an indecomposable $T(S)$-module. Then $M$ is injective if and only if $c^{-1}[M]$ is negative.
		\item If $M$ is an indecomposable module, then $c^{-1}[M]$ is either positive or negative.
	\end{enumerate}
\end{myprop}

The second and fifth condition in Proposition~\ref{Proposition - Coxeter transformation properties} yields a connection to the Auslander-Reiten quiver by considering all of the elements $[M]$, where $M$ is indecomposable $T(S)$-module, in $\textbf{K}_0(T(S)\mathrm{-mod})$.

For a $\K$-algebra $\Lambda$ we denote the bounded derived category of $\Lambda\mathrm{-mod}$ by $\mathcal{D}^b(\Lambda\mathrm{-mod})$.

\begin{mydef}
	Let $\Lambda$ be a finite dimensional $\K$-algebra with $\mathrm{gldim}(\Lambda) <\infty$.
	\begin{enumerate}
		\item We define the Nakayama functor $\nu: \mathcal{D}^b(\Lambda\mathrm{-mod})\to \mathcal{D}^b(\Lambda\mathrm{-mod})$ as the composition
		\begin{equation*}
			\nu = D\circ \textbf{R}\mathrm{Hom}_{\mathcal{D}^b(\Lambda\mathrm{-mod})}(-, \Lambda).
		\end{equation*}
		\item The Auslander-Reiten translation in the derived category is defined as $\nu_1 = \nu\circ [-1]$, where $[-1]$ is the shift functor.
	\end{enumerate}
\end{mydef}

Let $\phi$ be an endomorphism of a $\K$-algebra $\Lambda$. Using the derived tensor product we obtain an endomorphism functor $\Lambda_\phi\overset{\textbf{L}}{\otimes}_\Lambda-$ of $\mathcal{D}^b(\Lambda\mathrm{-mod})$, where $\Lambda_\phi$ is the $\Lambda$-$\Lambda$-bimodule $\Lambda$ where the right action is twisted by $\phi$, i.e. the right action is given by $a\cdot b = a\phi(b)$.

\begin{mydef}\cite[Definition 0.3]{herschend2011n}
	Let $\Lambda$ be a finite dimensional $\K$-algebra with $\mathrm{gldim}(\Lambda)<\infty$. We say that $\Lambda$ is twisted $\frac{m}{l}$-Calabi-Yau if there is an isomorphism
	\begin{equation*}
		\nu^l\simeq [m]\circ (\Lambda_\phi\overset{\textbf{L}}{\otimes}_{\Lambda}-)
	\end{equation*}
	of functors for some integer $l\not=0$ and an algebra endomorphism $\phi$ of $\Lambda$. In this case $\phi$ is always an automorphism.
\end{mydef}

We also need to define acyclic algebras since our main results in Section~\ref{Section - Koszul Complex for Higher Species} are only proved for acyclic algebras.

\begin{mydef}\cite[Definition 2.5]{pasquali2019tensor}
	We say that $\Lambda$ is cyclic if there exist indecomposable projective $\Lambda$-modules $P_1, \dots, P_m$ with non-zero non-isomorphisms $P_1\to P_2, \dots, P_{m-1}\to P_m, P_m\to P_1$ for some $m\ge 1$. We call $\Lambda$ acyclic if it is not cyclic.
\end{mydef}

\begin{myrem}
	Let $S$ be a species over the quiver $Q$. The tensor algebra $T(S)$ is acyclic if and only if $Q$ is acyclic.
\end{myrem}

\section{Properties for Representation Finite Species}\label{Section - Properties for Representation Finite Species}
In this section we compute the Nakayama permutation for all Dynkin diagrams. As an application we determine which species are $l$-homogeneous for each integer $l$. The Nakayama permutation is described in the following main theorem of this section.

\begin{mythm}\label{Theorem - Nakayama permutation}
	Let $S$ be a species with Dynkin diagram $\Delta$. Then the Nakayama permutation is given by:
	\begin{equation*}
		\begin{aligned}
		&\sigma_A(i) = n+1-i, &&\mbox{if }\Delta=A_n \\
		&\sigma_D(i)=i, &&\mbox{if }\Delta=D_{2n} \\
		&\sigma_D(i) = \begin{cases}
		2, & \mbox{if }i=1 \\
		1, & \mbox{if }i=2 \\
		i, & \mbox{otherwise} \\ 
		\end{cases}, &&\mbox{if }\Delta=D_{2n+1} \\
		&\sigma_{E_{6}}(i) = \begin{cases}
		6-i, & \mbox{if }i\not=6 \\
		6, & \mbox{if }i=6
		\end{cases}, &&\mbox{if }\Delta=E_6 \\
		&\sigma_{E_{7}}(i) = i, &&\mbox{if }\Delta=E_7 \\
		&\sigma_{E_{8}}(i) = i, &&\mbox{if }\Delta=E_8 \\
		&\sigma_\Delta(i) = i, &&\mbox{if }\Delta=B_n, C_n, F_4, G_2
		\end{aligned}
	\end{equation*}
\end{mythm}

We will later see the importance of this theorem when we describe the Nakayama automorphism in Section~\ref{Section - Nakayama Automorphism}.

Before we prove this theorem we will state a couple of lemmas.

\begin{mylemma}\label{lemma - map between projectives and injectives}
	Let $S$ be a species. If there is an arrow
	\begin{equation*}
		[P_i]\xrightarrow{d_{ij}}[P_j]
	\end{equation*}
	in $\Gamma_S$, then there is an arrow
	\begin{equation*}
		[I_i]\xrightarrow{d_{ij}}[I_j]
	\end{equation*}
	in $\Gamma_S$.
\end{mylemma}

\begin{proof}
	Let $S$ be a species. Note that
	\begin{equation*}
		D:T(S)\mbox{-mod}\xrightarrow{\sim}T(S)^{op}\mbox{-mod}.
	\end{equation*}
	Therefore, we have that
	\begin{equation*}
		\mathrm{Hom}_{T(S)}(T(S)e_i, T(S)e_j) \cong e_iT(S)e_j \cong \mathrm{Hom}_{T(S)^{op}}(e_jT(S), e_iT(S)) \cong \mathrm{Hom}_{T(S)}(D(e_iT(S)), D(e_jT(S))).
	\end{equation*}
	The claim now follows from the structure of $\Gamma_S$ described in Section~\ref{Section - Preliminaries}.
\end{proof}

\begin{mylemma}\label{lemma - Nakayama induces diagram automorphism}
	Let $\Delta$ be a Dynkin diagram. The Nakayama permutation $\sigma$ of $\Delta$ induces a diagram automorphism of $\Delta$.
\end{mylemma}

\begin{proof}
	Let $S$ be a species with diagram $\Delta$. Assume that there exist an arrow $[P_i]\rightarrow [P_j]$ in $\Gamma_S$. We need to show that there is an arrow $[P_{\sigma(i)}]\rightarrow [P_{\sigma(j)}]$ or $[P_{\sigma(j)}]\rightarrow [P_{\sigma(i)}]$ in $\Gamma_S$. Note that there is an arrow $[I_i]\rightarrow [I_j]$ by Lemma~\ref{lemma - map between projectives and injectives}. By Proposition~\ref{Proposition - nakayama permutation and integer l_i} there exists an integer $l_i$ such that $\tau^{l_i-1}I_i = P_{\sigma(i)}$. We only need to show that $l_i\le l_j\le l_i+1$. Indeed, if $l_j=l_i$ then there is an arrow $[P_{\sigma(i)}]\rightarrow [P_{\sigma(j)}]$ in $\Gamma_S$, and if $l_j=l_i+1$ then there is an arrow $[P_{\sigma(j)}]\rightarrow [P_{\sigma(i)}]$ in $\Gamma_S$, which is immediate since if $\tau^{l_i-1}I_j$ is not projective, then there exists an almost split sequence ending in $\tau^{l_i-1}I_j$ by \cite[Chapter 5, Corollary 2.4]{auslander1997representation}.
	
	Assume towards contradiction that $l_j>l_i+1$. Then $\tau^{l_i}I_j$ is not a projective module. Let $f:\tau^{l_i}I_j\rightarrow \tau^{l_i-1}I_i$ be an irreducible map. Using the fact that $T(S)$ is hereditary, we get that $\im f$ is projective. Thus $\hat{f}:\tau^{l_i}I_j\rightarrow \im f$ is a split epimorphism, which contradicts the fact that $\tau^{l_i}I_j$ is indecomposable and non-projective. A similar argument can be used when $l_j<l_i$. Thus $l_i\le l_j\le l_i+1$.
\end{proof}

The following proposition is a special case of \cite[Theorem 10.1 and 10.5]{schofield1985representations}.

\begin{myprop}
	Let $S$ be a species. There is a short exact sequence of $T(S)$-$T(S)$-bimodules
	\begin{equation}\label{eq - bimodule res of T(S)}
		0\rightarrow \bigoplus_{\substack{\alpha\in Q_1 \\ s(\alpha) = i}}T(S)e_{t(\alpha)}\otimes_{D_{t(\alpha)}} M_\alpha \otimes_{D_{s(\alpha)}}e_{s(\alpha)}T(S) \xrightarrow{d} \bigoplus_{i\in Q_0}T(S)e_i\otimes_{D_i} e_iT(S)\xrightarrow{mult} T(S)\rightarrow 0,
	\end{equation}
	where $d(a\otimes b\otimes c) = ab\otimes c - a\otimes bc$ and $mult$ denotes the natural multiplication.
\end{myprop}

\begin{mylemma}\label{lemma - Q does not depend on K}
	Let $\Delta$ be a Dynkin diagram of type ADE. The Nakayama permutation $\sigma$ only depends on $\Delta$.
\end{mylemma}

\begin{proof}
	It is enough to show that the Coxeter transformation only depends on the Dynkin diagram $\Delta$, since we can determine the structure of $\Gamma_S$ using the Coxeter transformation by Proposition~\ref{Proposition - Coxeter transformation properties}. The description of the Auslander-Reiten quivers by \cite{gabriel1980auslander} for the ADE cases gives us the information that $\sigma$ only depends on $\Delta$ when $T(S)$ is a path algebra over a quiver $Q$ with diagram $\Delta$. This reduces our problem to showing that for a given $S$ with diagram $\Delta$, $\sigma$ only depends on $Q$ and not the species $S$ over $Q$.
	
	Let $S$ be a species with diagram $\Delta$. Applying $-\otimes_{T(S)}D_i$ on the sequence (\ref{eq - bimodule res of T(S)}) yields and exact sequence
	\begin{equation*}
		0\rightarrow \bigoplus_{\substack{\alpha\in Q_1 \\ s(\alpha) = i}}T(S)e_{t(\alpha)}\otimes_{D_{t(\alpha)}}M_\alpha \rightarrow \bigoplus_{i\in Q_0}T(S)e_i\rightarrow D_i\rightarrow 0.
	\end{equation*}
	Moreover, this is the minimal projective resolution of $D_i$, which implies that $[P_i] + \sum_{\substack{\alpha\in Q_1 \\ s(\alpha) = i}}[P_{t(\alpha)}] = [D_i]$. The construction of this sequence only depends on $Q$, therefore the coordinates of $[P_i]$ only depends on $Q$ in $\textbf{K}_0(T(S)\mathrm{-mod})$. By the dual argument we show that the coordinates of $[I_i]$ only depends on $Q$ in $\textbf{K}_0(T(S)\mathrm{-mod})$. Hence $\Gamma_S$ only depends on $Q$.
\end{proof}

\begin{mylemma}\label{lemma - F does not change}
	Let $\Lambda$ be a finite dimensional $\K$-algebra. If $M$ is an indecomposable and non-projective $\Lambda$-module then $\delta(M)= \delta(\tau M)$.
\end{mylemma}

\begin{proof}
	Recall that $\tau: \Lambda\mbox{-}\underline{\mbox{mod}}\xrightarrow{\sim}\Lambda\mbox{-}\overline{\mbox{mod}}$. In particular,
	\begin{equation*}
		\underline{\mathrm{End}}_\Lambda(M) \cong \overline{\mathrm{End}}_\Lambda(\tau M).
	\end{equation*}
	
	Let $\alpha\in \mathrm{End}_\Lambda(M)$ be such that the residue class of $\alpha$ is zero in $\underline{\mathrm{End}}_\Lambda(M)$. In other words, we are in the situation
	\begin{equation*}
		\begin{tikzcd}
		M \arrow["\alpha", r] \arrow["\beta", d] & M \\
		P \arrow["\gamma"', ru] &
		\end{tikzcd}
	\end{equation*}
	where $\alpha = \gamma\beta$ and $P$ projective (not necessarily indecomposable). Since $M$ is not projective we have that $\beta$ and $\gamma$ are not split. Therefore $\beta\in \mathrm{rad}(M, P)$ and $\gamma \in \mathrm{rad}(P, M)$. This implies that $\alpha\in \mathrm{rad}(M, M)$. In other words, every endomorphism of $M$ that factors through a projective object is in the radical. Therefore we have the following
	\begin{equation*}
		\mathrm{End}_\Lambda(M)/\mathrm{rad}(\mathrm{End}_\Lambda(M))\cong \underline{\mathrm{End}}_\Lambda(M)/\mathrm{rad}(\underline{\mathrm{End}}_\Lambda(M)).
	\end{equation*}
	A similar argument can be used to show that
	\begin{equation*}
		\mathrm{End}_\Lambda(\tau M)/\mathrm{rad}(\mathrm{End}_\Lambda(\tau M))\cong \overline{\mathrm{End}}_\Lambda(\tau M)/\mathrm{rad}(\overline{\mathrm{End}}_\Lambda(\tau M)).
	\end{equation*}
	Hence, composing the three isomorphism gives us the result.
\end{proof}

We finally have everything we need to prove Theorem~\ref{Theorem - Nakayama permutation}.

\begin{proof}(Theorem~\ref{Theorem - Nakayama permutation})
	Let $S$ be a species where $\Delta$ is a Dynkin diagram. With Lemma~\ref{lemma - Q does not depend on K} we can use the description of the Auslander-Reiten quivers by \cite{gabriel1980auslander} for the $ADE$ cases.
	
	Now let $\Delta$ be non-simply laced. We define $N(i)$ to be the number of neighbours of $i$ in $\Delta$. Note that $\Delta$ satisfies $N(i)\le 2$ for all $i\in \Delta_0$. Since there is only one arrow of valuation higher than one in $\Delta$, using Lemma~\ref{lemma - map between projectives and injectives} and Lemma~\ref{lemma - F does not change} we have that the Nakayama permutation fixes the vertices that are connected to that arrow. Now we use an inductive argument to show that the rest of the vertices are also fixed. Let $i\xRightarrow{(k)}j$ be the arrow in $\Delta$ such that $k\not=1$ and $\delta(P_i)=k\delta(P_j)$. Then $\sigma(i)=i$ and $\sigma(j)=j$ as discussed. By Lemma~\ref{lemma - Nakayama induces diagram automorphism}, neighbours must be mapped to neighbours. In other words, $j'$, a neighbour of $i$ different from $j$, must be mapped to $j'$ or $j$, and since $j$ is fixed $j'$ must be fixed too. Inductively we can use this argument to show that every vertex is fixed.
\end{proof}

The following corollary is a generalization of \cite[Proposition 3.2 a)]{herschend2011n}.

\begin{mycor}\label{Corollary - l homogeneous species}
	Let $S$ be a species where $\Delta$ is a Dynkin diagram. Then $S$ is $l$-homogeneous if $Q$ is stable under $\sigma$. Moreover, for the different cases, the integer $l$ is
	\begin{center}
		\begin{tabular}{c|c|c|c|c|c|c|c|c|c}
			$\Delta$ & $A_n$ & $B_n$ & $C_n$ & $D_n$ & $E_6$ & $E_7$ & $E_8$ & $F_4$ & $G_2$ \\\hline
			$l$ & $\frac{n+1}{2}$ & $n$ & $n$ & $n-1$ & $6$ & $9$ & $15$ & $6$ & $3$
		\end{tabular}
	\end{center}
\end{mycor}

\begin{proof}
	By Lemma~\ref{lemma - Q does not depend on K} the $ADE$ cases follow from Gabriel's description of the AR-quivers \cite{gabriel1980auslander}.
	
	In the non-simply laced cases $\sigma$ is the identity by Theorem~\ref{Theorem - Nakayama permutation} and so $S$ is always $l$-homogeneous by \cite[Proposition 2.1]{herschend2011n}.
	
	The number of indecomposable modules in $T(S)\mathrm{-mod}$ is the same as the number of positive roots for $\Delta$. Thus $m = \frac{n h}{2}$ is the number of indecomposable modules, where $h$ is the Coxeter number for $\Delta$. Since $S$ is $l$-homogeneous, $l = \frac{m}{n}$ because $l$ is the number of indecomposable module in each $\tau$-orbit. Hence $l = \frac{h}{2}$, which gives the values in the above table.
\end{proof}

\section{Preprojective Algebras}\label{Section - The Preprojective Algebra}
In this section we introduce the preprojective algebra of a hereditary representation finite $\K$-algebra.

\begin{mydef}\cite{baer1987preprojective}\label{Definition - preprojective algebra BGL perspective}
	Let $\Lambda$ be a hereditary representation finite $\K$-algebra. The preprojective algebra $\Pi(\Lambda)$ is defined by
	\begin{equation}\label{Definition - eq preprojective algebra decomposition}
		\Pi(\Lambda) = \bigoplus_{i=0}^\infty \Pi_i
	\end{equation}
	where $\Pi_i = \Hom_\Lambda(\Lambda, \tau^{-i}\Lambda)$, with multiplication
	\begin{equation*}
		\begin{aligned}
			\Pi_i\times \Pi_j&\rightarrow \Pi_{i+j}, \\
			(u, v)&\mapsto uv = (\tau^{-i}(v)\circ u: \Lambda\rightarrow \tau^{-(i+j)}\Lambda).
		\end{aligned}
	\end{equation*}
\end{mydef}

\begin{myrem}\label{Remark - two natural grading on preprojective algebra}
	The preprojective algebra has a natural $\Z$-grading due to (\ref{Definition - eq preprojective algebra decomposition}). We will call this grading the $\star$-grading to avoid confusion. As it turns out in Section~\ref{Section - Koszul Complex for Higher Species} the $\star$-grading will play an important role when describing the almost Koszul complexes.
\end{myrem}

In the paper \cite{Dlab_1980} there is a more explicit description of the preprojective algebra of a species. Let $S$ be a species and let $\alpha\in Q_1$. There exist $x_1, \dots, x_n\in M_\alpha$ and $f_1, \dots, f_n\in \mathrm{Hom}_{D_{s(\alpha)}^{op}}(M_\alpha, D_{s(\alpha)})$ such that for every $x\in M_\alpha$ we have
\begin{equation*}
	x = \sum_{i=1}^n f_i(x)x_i.
\end{equation*}

\begin{mydef}
	Let $x_i$ and $f_i$ be as above. Then, the element
	\begin{equation*}
		c_\alpha = \sum_{i=1}^n x_i\otimes_{D_{s(\alpha)}} f_i\in M_\alpha\otimes_{D_{s(\alpha)}} \mathrm{Hom}_{D_{s(\alpha)}^{op}}(M_\alpha, D_{s(\alpha)})
	\end{equation*}
	is called the Casimir element of $M_\alpha\otimes_{D_{s(\alpha)}} \mathrm{Hom}_{D_{s(\alpha)}^{op}}(M_\alpha, D_{s(\alpha)})$.
\end{mydef} 

The Casimir element does not depend on the choice of $x_i$ and $f_i$ by \cite[Lemma 1.1]{Dlab_1980}.

\begin{mydef}\cite{Dlab_1980}\label{Definition - preprojective algebra of a species}
	Let $S$ be a species where $\Delta$ is a Dynkin diagram.
	\begin{enumerate}
		\item The double quiver $\overline{Q}$ is defined to be $\overline{Q}_0=Q_0$ and $\overline{Q}_1 = Q_1 \cup Q_1^*$ where
		\begin{equation*}
		Q_1^* = \{\alpha^*: j\rightarrow i\mid \alpha:i\rightarrow j\in Q_1 \}.
		\end{equation*}
		\item Let $\overline{S}$ be the species over $\overline{Q}$ where $\overline{D_i}=D_i$ for all $i\in \overline{Q}_0$ and $\overline{M}_\alpha = M_\alpha$ when $\alpha\in Q_1$ and $\overline{M}_{\alpha^*} = \mathrm{Hom}_{D_{s(\alpha)}^{op}}(M_\alpha, D_{s(\alpha)})$.
		\item For each $\alpha\in \overline{Q}_1$ let $c_\alpha$ be the Casimir element of $\overline{M}_\alpha\otimes_{D_{s(\alpha)}} \overline{M}_{\alpha^*}$. Define
		\begin{equation*}
		c = \sum_{\alpha\in \overline{Q}_1}\mathrm{sgn}(\alpha)c_\alpha,
		\end{equation*}
		where
		\begin{equation*}
		\mathrm{sgn}(\alpha) = \begin{cases}
		1, & \alpha\in Q_1 \\
		-1, & \mbox{else}.
		\end{cases}
		\end{equation*}
		We define the preprojective algebra of $S$ as $\Pi(S)=T(\overline{S})/\langle c\rangle$.
	\end{enumerate}
\end{mydef}

\begin{myrem}\label{Remark - Graded species}
	Let $S=(D, M)$ be a species. Given a decomposition 
	\begin{equation*}
		M = \bigoplus_{i\in \Z_{\ge 0}}M_i
	\end{equation*}
	we may consider $T(S)$ as a $\Z$-graded algebra by setting elements in $D$ to have degree zero and setting elements in $M_i$ to have degree $i$. This is equivalent to choosing a grading on the quiver $Q$ of $S$ by setting the arrow in $Q$ associated to $M_i$ to have degree $i$. Note that for such a graded species $S$ the quiver $Q$ may have multiple arrows between two vertices, but there will not be multiple arrows of the same degree.
	
	We will consider several different gradings for our species. For instance, setting $M = M_1$ we get a grading on $T(S)$ which corresponds to the path length grading on $T(S)$. We sometimes consider the preprojective algebra $\Pi(S) = T(\overline{S})/\langle c\rangle$ as a $\Z^2$-graded algebra. The first grading is induced by the decomposition $\overline{M} = \overline{M}_1$, and the second grading is induced by the decomposition
	\begin{equation*}
		\overline{M} = \overline{M}_0\oplus \overline{M}_1, \quad \overline{M}_0 =  \bigoplus_{\alpha\in Q}\overline{M}_\alpha, \quad \overline{M}_1 = \bigoplus_{\alpha^*\in Q^*}\overline{M}_{\alpha^*}.
	\end{equation*}
	In other words, the first grading corresponds to the path length grading and the second grading corresponds to the $\star$-grading defined in Remark~\ref{Remark - two natural grading on preprojective algebra}.
\end{myrem}

For a representation finite species $S$ over division algebras $F\subset G$, let $d_{s(\alpha)} = \dim_{D_{s(\alpha)}}\overline{M}_\alpha$ and $d_{t(\alpha)} = \dim_{D_{t(\alpha)}}\overline{M}_\alpha$. We choose elements $y_\alpha^1, \dots, y_\alpha^{\max(d_{s(\alpha)}, d_{t(\alpha)})}\in \overline{M}_\alpha$ such that
$\{y_\alpha^1, \dots, y_\alpha^{d_{s(\alpha)}} \}$ is a basis for $\overline{M}_\alpha$ when $\overline{M}_\alpha$ is viewed as a $D_{s(\alpha)}$-module and $\{y_\alpha^1, \dots, y_\alpha^{d_{t(\alpha)}} \}$ is a basis for $\overline{M}_\alpha$ when $\overline{M}_\alpha$ is viewed as a $D_{t(\alpha)}$-module. Moreover, we do this such that $c = \sum_{i=1}^{\dim_{D_{s(\alpha)}}\overline{M_\alpha}}y_\alpha^i\otimes_{D_{s(\alpha)}}y_{\alpha^*}^i$ is a Casimir element.

\begin{myex}\label{example - S species of Dynkin type C_3}
	Let $S$ be the species
	\begin{equation*}
		\C \xrightarrow{\C}\R_1 \xrightarrow{\R}\R_2.
	\end{equation*}
	Then $\overline{S}$ is given as
	\begin{equation*}
		\begin{tikzcd}
		\C \arrow[r, "\C", shift left] & \R_1 \arrow[l, "\C^*", shift left] \arrow[r, "\R", shift left] & \R_2 \arrow[l, "\R^*", shift left]
		\end{tikzcd}.
	\end{equation*}
	Pick the basis $1_\C, i_\C$ in $\C$ and $1_\R$ in $\R$, and let us denote the dual basis with $*$ as a superscript. The preprojective algebra of $S$ is
	\begin{equation*}
		\Pi(S) = T(\overline{S})/\langle c\rangle,
	\end{equation*}
	where
	\begin{equation*}
		c = - 1_\C^*\otimes_{\R_1} 1_\C - i_\C^*\otimes_{\R_1} i_\C + 1_\C\otimes_\C 1_\C^* - 1_\R^*\otimes_{\R_2} 1_\R + 1_\R \otimes_{\R_1} 1_\R^*.
	\end{equation*}
	
	Let us also describe $\Pi(T(S))$. The Auslander-Reiten quiver $\Gamma_S$ is given by
	\begin{equation*}
		\begin{tikzcd}[column sep={50pt,between origins}]
			& & P_1 \arrow[dr, red, "2"] \arrow[rr, dashed, -] & & \tau^{-1}P_1 \arrow[dr, red, "2"] \arrow[rr, dashed, -] & & I_1 \\
			& P_2 \arrow[ur, "2"] \arrow[dr, red] \arrow[rr, dashed, -] & & \tau^{-1}P_2 \arrow[ur, "2"] \arrow[dr, red] \arrow[rr, dashed, -] & & I_2 \arrow[ur, "2"] \\
			P_3 \arrow[ur] \arrow[rr, dashed, -] & & \tau^{-1}P_3 \arrow[ur] \arrow[rr, dashed, -] & & I_3 \arrow[ur]
		\end{tikzcd}
	\end{equation*}
	where the red arrows denotes the morphisms of $\star$-degree $1$. Elements in $\Pi(T(S))$ can be seen as linear combinations of paths in $\Gamma_S$. Thus we see similarities between $\Pi(S)$ and $\Pi(T(S))$ by comparing $Q_1$ with the arrows between projective modules, $Q_1^\star$ with the red arrows from projective modules and the relations $\langle c\rangle$ with the relations in $\Gamma_S$ coming from the Auslander-Reiten sequences.
\end{myex}

With Example \ref{example - S species of Dynkin type C_3} as motivation we now derive a relation between $\Pi(T(S))$ and $\Pi(S)$. In Section~\ref{Section - Preliminaries} we gave an explicit description of the Auslander-Reiten quiver $\Gamma_S$ for a representation finite species $S$. Thus we can represent all indecomposable modules in $\Gamma_S$ as translates of indecomposable projective modules. Let
\begin{equation*}
	\begin{tikzcd}[row sep = 10pt, column sep = 50pt]
		s(\alpha_1) \arrow[ddr, "\alpha_1"] & & t(\beta_1) \\
		s(\alpha_2) \arrow[dr, "\alpha_2"'] & & t(\beta_2) \\
		& i \arrow[uur, "\beta_1"] \arrow[ur, "\beta_2"'] \arrow[dr] \arrow[ddr, "\beta_m"'] \\
		\vdots \arrow[ur] & & \vdots \\
		s(\alpha_n) \arrow[uur, "\alpha_n"'] & & t(\beta_m) \\
	\end{tikzcd}
\end{equation*}
be the subquiver of $Q$ consisting of all outgoing and ingoing arrows to $i$. By \cite[Proposition 1.15]{auslander1997representation} we have a corresponding picture in $\Gamma_S$. 
\begin{equation*}
	\begin{tikzcd}[row sep = 10pt, column sep = 70pt]
		P_{t(\beta_1)} \arrow[ddr, "d_{P_{t(\beta_1)}, P_i}"] & & P_{s(\alpha_1)} \\
		P_{t(\beta_2)} \arrow[dr, "d_{P_{t(\beta_2)}, P_i}"'] & & P_{s(\alpha_2)} \\
		& P_i \arrow[uur, "d_{P_i, P_{s(\alpha_1)}}"] \arrow[ur, "d_{P_i, P_{s(\alpha_2)}}"'] \arrow[dr] \arrow[ddr, "d_{P_i, P_{s(\alpha_n)}}"'] \\
		\vdots \arrow[ur] & & \vdots \\
		P_{t(\beta_m)} \arrow[uur, "d_{P_{t(\beta_m)}, P_i}"'] & & P_{s(\alpha_m)} \\
	\end{tikzcd}
\end{equation*}

Let us now consider the almost split sequence starting at $P_i$.
\begin{equation}\label{eq - almost split seq in AR quiver}
	\begin{tikzcd}
		& \bigoplus_{\substack{\alpha\in Q_1 \\ t(\alpha) = i}}P_{s(\alpha)}^{d_{P_i, P_{s(\alpha)}}/\delta(P_{s(\alpha)})} \arrow[dr, "g_i'"] \\
		P_i \arrow[ur, "f_i"] \arrow[dr, "g_i"] & & \tau^{-1}P_i \\
		& \bigoplus_{\substack{\alpha\in Q_1 \\ s(\alpha) = i}}\tau^{-1}P_{t(\alpha)}^{d_{P_{t(\alpha)}, P_i}/\delta(P_{t(\alpha)})} \arrow[ur, "-f_i'"]
	\end{tikzcd}
\end{equation}
We write $f_i=[f_\alpha]_{\substack{\alpha\in Q_1 \\ s(\alpha)=i}}$, where $f_\alpha = \begin{bmatrix}
	f_\alpha^1 & f_\alpha^2 & \cdots & f_\alpha^{d_{P_i, P_{s(\alpha)}}/\delta(P_{s(\alpha)})}
\end{bmatrix}$. We express $g_i, f_i'$ and $g_i'$ in a similar fashion. From the fact that (\ref{eq - almost split seq in AR quiver}) is an almost split sequence, we know that $g_i'\circ f_i - f_i'\circ g_i = 0$. We want to compare this equality to $e_ice_i$, where $c$ is the Casimir element for $\Pi(S)$. Let us write out both expressions and compare them.
\begin{equation}\label{eq - fg - gf almost split sequence}
	\begin{aligned}
		g_i'\circ f_i - f_i'\circ g_i &= \sum_{\substack{\alpha\in Q_1 \\ t(\alpha)=i}}\sum_{k=1}^{d_{P_i, P_{s(\alpha)}/\delta(P_{s(\alpha)})}} {g'}_\alpha^k\circ f_\alpha^k - \sum_{\substack{\alpha\in Q_1 \\ s(\alpha)=i}}\sum_{k=1}^{d_{P_i, P_{t(\alpha)}/\delta(P_{t(\alpha)})}} {f'}_\alpha^k\circ g_\alpha^k = \\
		&= \sum_{\substack{\alpha\in Q_1 \\ t(\alpha)=i}}\sum_{k=1}^{d_{P_i, P_{s(\alpha)}/\delta(P_{s(\alpha)})}} f_\alpha^k {g'}_\alpha^k - \sum_{\substack{\alpha\in Q_1 \\ s(\alpha)=i}}\sum_{k=1}^{d_{P_i, P_{t(\alpha)}/\delta(P_{t(\alpha)})}} g_\alpha^k \tau({f'}_\alpha^k)
	\end{aligned}
\end{equation}
Note that the second equality holds since the multiplication in the $\Pi(T(S))$ is in the reversed order compared to composition. On the other hand
\begin{equation}\label{eq - fg - gf casimir element almost split}
	e_i c e_i = \sum_{\substack{\alpha\in Q_1 \\ t(\alpha) = i}}\sum_{i=1}^{\dim_{D_{s(\alpha)}}\overline{M}_\alpha} y_\alpha^k\otimes_{D_{s(\alpha)}}y_{\alpha^*}^k - \sum_{\substack{\alpha\in Q_1 \\ s(\alpha) = i}}\sum_{i=1}^{\dim_{D_{t(\alpha)}}\overline{M}_\alpha} y_{\alpha^*}^k\otimes_{D_{t(\alpha)}}y_\alpha^k.
\end{equation}
First note that if $\alpha:i\rightarrow j$ then
\begin{equation*}
	\mathrm{rad}_{T(S)}(P_j, P_i) = \mathrm{Hom}_{T(S)}(P_j, P_i) \cong e_jT(S)e_i = \overline{M}_\alpha,
\end{equation*}
and also that $\mathrm{rad}^2_{T(S)}(P_i, P_j) = 0$ since $i$ and $j$ are neighbours in $Q$ and $\Delta$ is a tree. Therefore
\begin{equation*}
	\dim_{D_{t(\alpha)}}\overline{M}_\alpha = \dim_\K \mathrm{rad}_{T(S)}(P_j, P_i) / \dim_\K D_{t(\alpha)} = d_{P_{t(\alpha)}, P_i}/\delta(P_{t(\alpha)}).
\end{equation*}
Similarly, if $\alpha:j\rightarrow i$ then
\begin{equation*}
	\dim_{D_{s(\alpha)}}\overline{M}_\alpha = d_{i, P_{s(\alpha)}}/\delta(P_{s(\alpha)}).
\end{equation*}
Hence we know the ranges of the sums in (\ref{eq - fg - gf almost split sequence}) and (\ref{eq - fg - gf casimir element almost split}) are the same. Let us now fix $f_i$ for all $i\in Q_0$, for example one can choose $f_\alpha^k = -\cdot y_\alpha^k$. Using the equivalence $\tau: T(S)\mathrm{-\underline{mod}}\rightarrow T(S)\mathrm{-\overline{mod}}$, and the fact that $\tau^{-1}: T(S)\mathrm{-mod}\to T(S)\mathrm{-mod}$ is given by $\tau^{-1} = \mathrm{Ext}^1_{T(S)}(DT(S), T(S))\otimes_{T(S)}-$, we know that $f_i'\in \tau^{-1}(\mathrm{rad}_{T(S)}(T(S), T(S)))$ for all $i$, and therefore we can choose ${f_\alpha'}^k = \tau^{-1}(f_\alpha^k)$. Now choose $g_i$ and $g_i'$ such that (\ref{eq - almost split seq in AR quiver}) becomes an almost split sequence for each $i\in Q_0$. Comparing (\ref{eq - fg - gf almost split sequence}) and (\ref{eq - fg - gf casimir element almost split}) we can construct a well-defined isomorphism stated in the following proposition.

\begin{myprop}\label{Proposition - explcit desc of iso between pi(lambda) and pi(S)}
	Let $S$ be a species, where $\Delta$ is a Dynkin diagram, over division algebras $F\subset G$. There is an isomorphism $\Pi(S)\xrightarrow{\sim}\Pi(T(S))$, where $\Pi(T(S))$ is the preprojective algebra of $T(S)$ defined in \cite{baer1987preprojective}, given by
	\begin{equation*}
		\begin{aligned}
			\Pi(S)&\rightarrow \Pi(T(S)) \\
			e_i&\mapsto 1_{D_i} \\
			y_\alpha^k &\mapsto \begin{cases}
				f_\alpha^k & \mbox{if }\alpha\in Q_1 \\
				g_\alpha^k & \mbox{if }\alpha\in Q_1^*.
			\end{cases}
		\end{aligned}
	\end{equation*}
\end{myprop}

In the path algebra case, i.e. when $\Delta$ is of Dynkin type ADE, the above result is proven by Ringel in \cite[Theorem A]{ringel1998preprojective}. It is also proven for a generalization of species called phylum in \cite[Theorem 7.6]{gao2020functorial}. A species $S$ give rise to a phylum by considering the adjoint functors
\begin{equation*}
	\begin{aligned}
		X = \left(\bigoplus_{\alpha\in Q_1}\overline{M}_\alpha\right) \otimes_{D_{s(\alpha)}} -: D\mathrm{-mod} &\to D\mathrm{-mod}, \\
		Y = \left(\bigoplus_{\alpha\in Q_1}\overline{M}_{\alpha^\star}\right) \otimes_{D_{t(\alpha)}} - : D\mathrm{-mod} &\to D\mathrm{-mod}.
	\end{aligned}
\end{equation*}

\begin{myrem}
	Recall that $\Pi(S)$ is $\Z^2$-graded. We can define a $\Z^2$-grading on $\Pi(T(S))$, motivated by Proposition~\ref{Proposition - explcit desc of iso between pi(lambda) and pi(S)}, in the following way. We define a $\Z^2$-grading on $\Pi(T(S))$ by setting the first grading to be
	\begin{equation*}
		\begin{aligned}
			\Pi(T(S))_0 &= T(S)_0 = D, \\
			\Pi(T(S))_1 &= T(S)_1 \oplus \bigoplus_{\alpha\in Q_1} \mathrm{Hom}_{T(S)}(T(S)e_{s(\alpha)}, \tau^{-1}T(S)e_{t(\alpha)}), \\
			\Pi(T(S))_n &= (\Pi(T(S))_1)^n,
		\end{aligned}
	\end{equation*}
	and letting the second grading be the $\star$-grading. Then the morphism in Proposition~\ref{Proposition - explcit desc of iso between pi(lambda) and pi(S)} is compatible with the $\Z^2$-gradings since it induces isomorphisms in the degrees $(0, 0)$, $(1, 0)$ and $(1, 1)$ in which both $\Pi(S)$ and $\Pi(T(S))$ are generated.
\end{myrem}

\section{Nakayama Automorphism}\label{Section - Nakayama Automorphism}
In this section we study the Nakayama automorphism of the preprojective algebra $\Pi(\Lambda)$, where $\Lambda$ is a basic hereditary representation finite $\K$-algebra. The existence of a Nakayama automorphism is a consequence of $\Pi(\Lambda)$ being a Frobenius algebra. In the paper \cite{brenner2002periodic} they give an explicit formula for the Nakayama automorphism for preprojective algebras $\Pi(\K Q)$, where $Q$ is a quiver of Dynkin type ADE. Thus showing that $\Pi(\K Q)$ is self-injective. Iyama and Oppermann later showed, in a more general setting, that the preprojective algebra of a representation finite hereditary $\K$-algebra is self-injective in the paper \cite{iyama2013stable}. An alternative proof of this statement was given in \cite{grant2019nakayama}.

\begin{mythm}\cite[Theorem 4.8]{brenner2002periodic}\cite[Theorem 3.1]{grant2019nakayama}
	Let $\Lambda$ be basic, hereditary and representation finite, then $\Pi(\Lambda)$ is a Frobenius algebra.
\end{mythm}

For a Frobenius algebra $\Pi(\Lambda)$ we know that $\Pi(\Lambda)\cong D\Pi(\Lambda)$ as left $\Pi(\Lambda)$-modules, but this isomorphism can be made into an isomorphism of $\Pi(\Lambda)$-$\Pi(\Lambda)$-bimodules by introducing a twist on $D\Pi(\Lambda)$, i.e. $\Pi(\Lambda)\cong D\Pi(\Lambda)_\gamma$ as $\Pi(\Lambda)$-$\Pi(\Lambda)$-bimodules. This $\gamma$ is unique up to an inner automorphism, and it is called the Nakayama automorphism of $\Pi(\Lambda)$.

From \cite{grant2019nakayama} we also have a description of the Nakayama automorphism $\gamma$. By \cite[Proposition 3.4]{grant2019nakayama} we have an isomorphism
\begin{equation}\label{eq - isomorphism from grant proposition 3.4}
	\begin{aligned}
		\mathrm{Hom}_\Lambda(P_i, \tau^{-r}P_j) &\xrightarrow{\sim} D\mathrm{Hom}_\Lambda(\tau^{-r}P_j, I_i) \xrightarrow{\sim} D\mathrm{Hom}_\Lambda(P_j, \tau^rI_i) \xrightarrow{\sim} \\ 
		&\xrightarrow{\sim} DD\mathrm{Hom}_\Lambda(\tau^rI_i, I_j)\xrightarrow{\sim} \mathrm{Hom}_\Lambda(\tau^rI_i, I_j).
	\end{aligned}
\end{equation}

In the \cite{baer1987preprojective} perspective, let 
\begin{equation*}
f:P_i\rightarrow \tau^{-r}P_j
\end{equation*}
be a map. Choosing an isomorphism $P_{\sigma(i)} \cong \tau^{l_i-1}I_i$ and combining with the isomorphism in (\ref{eq - isomorphism from grant proposition 3.4}) we define
\begin{equation*}
\gamma(f):P_{\sigma(i)}\rightarrow \tau^{l_i - l_j - r}P_{\sigma(j)}.
\end{equation*}

Note that if $f = \id_{P_i}$, then $r = 0$ and $i = j$ and therefore $\gamma(\id_{P_i}) = \id_{P_{\sigma(i)}}$.

\begin{mythm}{\normalfont \cite[Theorem 3.8]{grant2019nakayama}}\label{Theorem - Grant Nakayama automorphism}
	The morphism $\gamma$ defined above is a Nakayama automorphism of $\Pi(\Lambda)$.
\end{mythm}

The aim of this section is to prove the following result.

\begin{mythm}\label{Theorem - description of nakayama automorphism true}
	Let $S$ be a species, where $\Delta$ is a Dynkin diagram, over division algebras $F\subset G$. There is a Nakayama automorphism $\gamma$ of $\Pi(S)$ given by
		\begin{equation*}
			\gamma(y_\alpha^i) = \begin{cases}
			y^i_{\sigma(\alpha)}, & \mbox{if }\alpha\in Q_1 \\
			\sgn(\sigma(\alpha)) y^i_{\sigma(\alpha)}, & \mbox{if }\alpha\not\in Q_1
			\end{cases},
		\end{equation*}
		and $\gamma(e_i) = e_{\sigma(i)}$ for all $i\in Q_0$.
\end{mythm}

Note that Theorem \ref{Theorem - description of nakayama automorphism true} is a generalization of \cite[Corollary 4.7]{brenner2002periodic}. 

\begin{proof}
	We prove the following statements separately.
	\begin{enumerate}
		\item There is a Nakayama automorphism $\gamma$ of $\Pi(S)$ given by
		\begin{equation*}
			\gamma(y_\alpha^i) = \begin{cases}
			y^i_{\sigma(\alpha)}, & \mbox{if }\alpha\in Q_1 \\
			\sgn(\sigma(\alpha))f y^i_{\sigma(\alpha)}, & \mbox{if }\alpha\not\in Q_1
			\end{cases},
		\end{equation*}
		for some constant $f\in F$, and $\gamma(e_i) = e_{\sigma(i)}$ for all $i\in Q_0$.
		\item $f=1$.
	\end{enumerate}
	\textbf{Proof of (1):} We will divide the proof in two parts. The first part will consist of the proof for when $\Delta$ is simply laced, and the other part will cover when $\Delta$ is non-simply laced.
	
	Assume that $\Delta$ is simply laced. Since $\Delta$ is a Dynkin diagram of type ADE we have that $D_i = F$ for all $i\in Q_0$ and also $M_\alpha = F$ for all $\alpha\in Q_1$. Using Theorem~\ref{Theorem - Grant Nakayama automorphism} together with the isomorphism in Proposition~\ref{Proposition - explcit desc of iso between pi(lambda) and pi(S)} we see that the Nakayama automorphism is defined on the generators as
	\begin{equation*}
		\gamma(y_\alpha) = f^1_\alpha y_{\sigma(\alpha)} f^2_\alpha,
	\end{equation*}
	where $f^1_\alpha, f^2_\alpha\in F$, and note that $\gamma(e_i) = e_{\sigma(i)}$. Since $M_\alpha = F$ for all $\alpha\in Q_1$, we have that $\dim_{D_{s(\alpha)}}\overline{M}_\alpha = \dim_F F = 1$ and thus we can choose $y_\alpha = 1_F$ as a basis for $\overline{M}_\alpha$. Everything commutes with $1_F$ in $F$ and therefore
	\begin{equation*}
		\gamma(y_\alpha) = f_\alpha y_{\sigma(\alpha)},
	\end{equation*}
	where $f_\alpha = f^1_\alpha f^2_\alpha$. Since the Nakayama automorphism is unique up to an inner automorphism, we can use inner automorphisms to simplify $\gamma$. We define a map
	\begin{equation*}
		\begin{aligned}
			g(-, -): Q_0\times D^\times&\rightarrow \mathrm{Inn}(\Pi(S)), \\
			(j, d = (d_1, d_2, \dots, d_{|Q_0|}))&\mapsto \left[x\mapsto \left(\sum_{\substack{i\in Q_0 \\ i\not=j}}e_i + d_j^{-1}e_j\right)x\left(\sum_{\substack{i\in Q_0 \\ i\not=j}}e_i + d_je_j\right)\right].
		\end{aligned}
	\end{equation*}
	We can compute explicitly that $g(t(\beta), f_\beta)\circ \gamma$ is a Nakayama automorphism given by
	\begin{equation}\label{eq - inner automorphism applied to nakayama automorphism}
		(g(t(\beta), f_\beta)\circ \gamma)(y_\alpha) = \begin{cases}
			y_\alpha, & \mbox{if }\alpha = \beta \\
			f_\beta f_\alpha y_\alpha, & \mbox{if }s(\alpha) = t(\beta) \\
			f_\beta^{-1} f_\alpha y_\alpha, & \mbox{if }t(\alpha) = t(\beta) \\
			f_\alpha y_\alpha, & \mbox{else}
		\end{cases}
	\end{equation}
	
	Note that by replacing $\gamma$ with $(g(t(\beta), f_\beta)\circ \gamma)(y_\alpha)$ we change $f_\beta$ to $1$ and leave $f_\alpha$ unchanged for all $\alpha$ not adjacent to $t(\beta)$. We can also remove $f_\beta$ by considering the Nakayama automorphism $g(s(\beta), f^{-1}_\beta)\circ \gamma$. 
	
	Since $Q$ is a tree we can enumerate $Q_0$ by picking a root $0$ and then proceed to label the other vertices by $1, 2, \dots$ such that for each $i\ge 1$ there exists an arrow $\alpha:i\rightarrow j\in Q_1$ or $\alpha: j\rightarrow i\in Q_1$, where $j<i$, and let us denote said arrow by $\alpha_i$. We claim that this induces an enumeration on the arrows in $Q$. It is enough to show that $\alpha_i$ is unique since $Q$ has $|Q_0|-1$ arrows. Let $\Delta_i$ be the full subdiagram of $\Delta$ with vertices $\{0, 1, \dots, i-1\}$, where $i\le |Q_0|$. Now assume that there is an edge $a$ in $\Delta$ between $i$ and $j$ and an edge $a'$ between $i$ and $j'$, where $j,j'<i$. Since $\Delta_i$ is connected there is a path $\omega$ from $j$ to $j'$. Extending $\omega$ with $a$ and $a'$ yields a cycle in $\Delta$ which is absurd since $Q$ is a tree. This proves uniqueness of $\alpha_i$.
	
	We can now use induction on the arrows $y_{\alpha_i}$ to simplify $\gamma$. Assume that $\gamma(y_{\alpha_i}) = y_{\sigma(\alpha_i)}$ for all $i\le n$. Without loss of generality we can assume that $t(y_{\alpha_{n+1}}) = n+1$, since the proof for $s(y_{\alpha_{n+1}}) = n+1$ uses the same argument. Let us now consider $g(n+1, f_{\alpha_{n+1}})\circ \gamma$. By (\ref{eq - inner automorphism applied to nakayama automorphism}) we see that $(g(n+1, f_{\alpha_{n+1}})\circ \gamma)(y_{\alpha_{n+1}}) = y_{\sigma(\alpha_{n+1})}$. Since $y_{\alpha_{n+1}}$ is the unique arrow connecting $\{1, \dots, n\}$ to $n+1$ we still have that $(g(n+1, f_{\alpha_{n+1}})\circ \gamma)(y_{\alpha_i}) = y_{\sigma(\alpha_i)}$ for all $i\le n$. Therefore, by inductively modifying $\gamma$ we get
	\begin{equation*}
		\gamma(y_\alpha) = \begin{cases}
			y_{\sigma(\alpha)}, & \mbox{if }\alpha\in Q_1 \\
			f_\alpha y_{\sigma(\alpha)}, & \mbox{if }\alpha\not\in Q_1
		\end{cases}.
	\end{equation*}
	
	It is left to show that $f_\alpha$ does not depend on $\alpha$ up to a sign. This is shown by comparing $c$ and $\gamma(c)$. Recall that
	\begin{equation*}
		c = \sum_{\alpha\in \overline{Q}_1}\sgn(\alpha)y_\alpha\otimes_{D_{s(\alpha)}}y_{\alpha^*}.
	\end{equation*}
	Now applying $\gamma$ gives
	\begin{equation*}
		\begin{aligned}
			\gamma(c) &= \sum_{\alpha\in 		Q_1}y_{\sigma(\alpha)}\otimes_{D_{s(\alpha)}}f_{\alpha^*}y_{\sigma(\alpha)^*} - \sum_{\alpha^*\in Q^*_1}f_{\alpha^*}y_{\sigma(\alpha)^*}\otimes_{D_{s(\alpha^*)}}y_{\sigma(\alpha)}
		\end{aligned}
	\end{equation*}
	
	Note that $\gamma$ can be viewed as a graded morphism where the grading is given by $\deg(y_\alpha) = 1$ for all $\alpha\in \overline{Q_1}$ and $\deg(e_i) = 0$ for all $i\in Q_0$. Let $\gamma': T(\overline{S})\rightarrow T(\overline{S})$ be defined by having the same constants as $\gamma$. Since $\gamma'$ induces $\gamma$ we get that $\gamma'$ sends $\langle c \rangle$ to $\langle c \rangle$. We get that
	\begin{equation*}
		\gamma'(e_i c e_i) = e_{\sigma(i)}\gamma'(c)e_{\sigma(i)}\in e_{\sigma(i)}T(\overline{S})_2e_{\sigma(i)}\cap \langle c \rangle = e_{\sigma(i)}T(\overline{S})_0ce_{\sigma(i)}
	\end{equation*}
	since $c$ has degree $2$. Now $e_{\sigma(i)}T(\overline{S})_0 = T(\overline{S})_0 e_{\sigma(i)} = D_i = F$ yields
	\begin{equation*}
		\gamma'(e_ice_i) \in F\langle e_{\sigma(i)}ce_{\sigma(i)}\rangle.
	\end{equation*}
	Let $f_i\in F$ be a constant such that
	\begin{equation*}
		\gamma'(e_ice_i) = \sum_{\substack{\alpha\in \overline{Q}_1, \\ s(\alpha) = i \mbox{ or } t(\alpha)=i}} -\sgn(\sigma(\alpha)) f_i y_{\sigma(\alpha)}\otimes_{D_{s(\alpha)}} y_{\sigma(\alpha)^*}.
	\end{equation*}
	Thus
	\begin{equation*}
		f_{\alpha^*} = -\sgn(\alpha)\sgn(\sigma(\alpha))f_i = -\sgn(\alpha^*)\sgn(\sigma(\alpha^*))f_i = \sgn(\sigma(\alpha^*))f_i
	\end{equation*}
	for all $\alpha\in Q^*_1$ such that $s(\alpha) = i$ or $t(\alpha)=i$. Since $Q$ is connected, $f_i = f_j$ for all $i,j\in Q_0$, and therefore
	\begin{equation*}
		\gamma(y_\alpha) = \begin{cases}
			y_{\sigma(\alpha)}, & \mbox{if }\alpha\in Q_1 \\
			\sgn(\sigma(\alpha))f y_{\sigma(\alpha)}, & \mbox{if }\alpha\not\in Q_1
		\end{cases},
	\end{equation*}
	for some $f\in F$.
	
	Now we assume that $\Delta$ is non-simply laced. We want to reduce the problem so that we can reuse the logic we used for when $\Delta$ is simply laced. Similar to the simply laced case, we can assume that $y_\alpha^1 = 1_F$. Since $\Delta$ is a Dynkin diagram of BCFG, there exists a unique arrow $\alpha_0\in Q_1$ such that $M_{\alpha_0} = {}_FG_G$ or $M_{\alpha_0} = {}_GG_F$. We will assume the latter since both cases use similar arguments. Then we know $\dim_{D_{t(\alpha)}} M_\alpha = 1$ for all $\alpha\in Q_1$. Therefore we can write
	\begin{equation*}
		\gamma(y^i_\alpha) = d_\alpha y^i_{\sigma(\alpha)},
	\end{equation*}
	for all $\alpha\in Q_1$ where $d_\alpha \in D_{t(\alpha)}$. Let us now enumerate the vertices and apply the same argument as in the simply laced case to simplify $\gamma$. Pick $0 = s(\alpha_0)$ as our root and let the other vertices be defined in the same way as before. Then the same induction argument works for this case, since $\dim_{D_{t(\alpha)}}M_\alpha = 1$ and $\dim_{D_{s(\beta)}} M_\beta = \dim_{D_{t(\beta)}}M_\beta = 1$ for all $\alpha\not=\beta\in Q_1$, and we can write
	\begin{equation*}
		\gamma(y^i_\alpha) = \begin{cases}
			y^i_{\sigma(\alpha)}, & \mbox{if }\alpha\in Q_1, \\
			d_\alpha y^i_{\sigma(\alpha)}, & \mbox{if }\alpha\not\in Q_1.
		\end{cases}
	\end{equation*}
	
	Now we want to show that $d_\alpha\in F$ for all $\alpha\in Q_1$ and that $d_\alpha$ does not depend on $\alpha$. Let $\gamma'$ be defined in the same way as in the simply laced case. We again compare $c$ and $\gamma'(c)$. Recall that
	\begin{equation*}
		c = \sum_{\alpha\in \overline{Q}_1}\sum_{i=1}^{\dim_{D_{s(\alpha)}}(\overline{M}_\alpha)}\sgn(\alpha)y^i_\alpha\otimes_{D_{s(\alpha)}}y^i_{\alpha^*}.
	\end{equation*}
	Now applying $\gamma'$ gives
	\begin{equation*}
		\begin{aligned}
			\gamma'(c) =& \sum_{\alpha\in Q_1}\sum_{i=1}^{\dim_{D_{s(\alpha)}}(\overline{M}_\alpha)}y^i_{\sigma(\alpha)}\otimes_{D_{s(\alpha)}}d_{\alpha^*}y^i_{\sigma(\alpha)^*} - \sum_{\alpha^*\in Q^*_1}\sum_{i=1}^{\dim_{D_{s(\alpha^*)}}(\overline{M}_{\alpha^*})}d_{\alpha^*}y^i_{\sigma(\alpha)^*}\otimes_{D_{s(\alpha^*)}}y^i_{\sigma(\alpha)}
		\end{aligned}
	\end{equation*}
	
	The same argument as in the simply laced case is used to show that $\gamma'(c) \in \langle c\rangle$. More explicitly,
	\begin{equation}\label{eq - ideal of c maps to <c>}
		\gamma'(e_i ce_i) = -d_ie_{\sigma(i)}ce_{\sigma(i)},
	\end{equation}
	for some $d_i\in D_i$. If we show that $\gamma(c_\alpha) = d_{\alpha^*} c_{\sigma(\alpha)}$ and $\gamma(c_{\alpha^*}) = d_\alpha c_{\sigma(\alpha)^*}$ for all $\alpha\in Q_1$, then we can use the same argument as in the simply laced case to say that $d_{\alpha^*} = d\sgn(\sigma(\alpha^*))$ for some $d\in G$, and also use that $d\in D_i$ for all $i\in Q_0$ to say that $d\in F$.
	
	Now we prove that $\gamma'(c_\alpha) = d_{\alpha^*} c_{\sigma(\alpha)}$ and $\gamma'(c_{\alpha^*}) = d_{\alpha^*} c_{\sigma(\alpha)^*}$ for all $\alpha\in Q_1$. The latter is immediate from
	\begin{equation*}
		\begin{aligned}
			\gamma'(c_{\alpha^*}) = \gamma\left(\sum_{i=1}^{\dim_{D_{s(\alpha^*)}}(\overline{M}_{\alpha^*})}y^i_{\alpha^*}\otimes_{D_{s(\alpha^*)}}y^i_\alpha\right) = \sum_{i=1}^{\dim_{D_{s(\alpha^*)}}(\overline{M}_{\alpha^*})}d_{\alpha^*}y^i_{\sigma(\alpha)^*}\otimes_{D_{s(\alpha^*)}}y^i_\alpha = d_{\alpha^*}c_{\sigma(\alpha)^*}.
		\end{aligned}
	\end{equation*}
	To prove that $\gamma'(c_\alpha) = d_{\alpha^*} c_{\sigma(\alpha)}$ we first do a similar computation
	\begin{equation*}
		\begin{aligned}
			\gamma'(c_\alpha) &=\gamma\left(\sum_{i=1}^{\dim_{D_{s(\alpha)}}(\overline{M}_\alpha)}y^i_\alpha\otimes_{D_{s(\alpha)}}y^i_{\alpha^*}\right) = \sum_{i=1}^{\dim_{D_{s(\alpha)}}(\overline{M}_\alpha)}y^i_{\sigma(\alpha)}\otimes_{D_{s(\alpha)}}d_{\alpha^*}y^i_{\sigma(\alpha)^*} = \\
			&= \sum_{i=1}^{\dim_{D_{s(\alpha)}}(\overline{M}_\alpha)}d^i_{\alpha^*}y^i_{\sigma(\alpha)}\otimes_{D_{s(\alpha)}}y^i_{\sigma(\alpha)^*},
		\end{aligned}
	\end{equation*}
	where $d_{\alpha^*}^i$ is defined via $d^i_{\alpha^*}y^i_{\sigma(\alpha)} = y^i_{\sigma(\alpha)}d_{\alpha^*}$. By (\ref{eq - ideal of c maps to <c>}) we get $d_{\alpha^*}^i = d_{\alpha^*}^j$ for all $i,j = 1, \dots, \dim_{D_{s(\alpha)}}(\overline{M}_\alpha)$. Moreover, since $y^1_{\alpha} = 1_F$ we have $d^1_{\alpha^*} = d_{\alpha^*}$. Hence $\gamma'(c_\alpha) = d_{\alpha^*} c_{\sigma(\alpha)}$.
	
	\textbf{Proof of (2):} For this part we will use an alternative description of Frobenius algebras that uses a bilinear form. In our case, there is a non-degenerate bilinear form $\beta: \Pi(S)\times \Pi(S) \to \K$ such that $\beta(xy, z) = \beta(x, yz)$ and $\beta(x, y) = \beta(y, \gamma(x))$ for all $x,y,z\in \Pi(S)$. Let $\alpha: i\to j\in \overline{Q}_1$ and fix a non-zero element $u_i\in \mathrm{Soc}(\Pi(S)e_i)$. Since $\Pi(S)y_\alpha^1\subset \Pi(S)e_i$ and that $\mathrm{Soc}(\Pi(S)e_i)$ is simple, we have that $\mathrm{Soc}(\Pi(S)e_i)\subset \Pi(S)y_\alpha^1$. In other words, $u_i = v_1y_\alpha^1$ for some element $v_1\in \Pi(S)$. Dually, we have that $\mathrm{Soc}(\Pi(S)e_i)\subset y_{\sigma(\alpha^*)}^1\Pi(S)$ and thus $u_i = y_{\sigma(\alpha^*)}^1v_2$ for some element $v_2\in \Pi(S)$. Consider the element $u_j = y_{\sigma(\alpha)}^1 v_1$. We now claim that we are done if we can show that
	\begin{equation*}
		u_j = -\sgn(\alpha^*)\sgn(\sigma(\alpha^*))v_2 y_{\alpha^*}^1.
	\end{equation*}
	To summarise
	\begin{equation*}
		\begin{aligned}
			v_1y_\alpha^1 &= u_i = y_{\sigma(\alpha^*)}^1v_2, \\
			y_{\sigma(\alpha)}^1 v_1 &= u_j = -\sgn(\alpha^*)\sgn(\sigma(\alpha^*))v_2 y_{\alpha^*}^1.
		\end{aligned}
	\end{equation*}
	Indeed, if we show that $\beta(z, u_i) = \beta(z, u_if)$, then since $\beta$ is non-degenerate we have that $u_i = u_if$ and therefore $f=1$. Note that if $\beta(z, u_i) = 0$ for all $z\in \Pi(S)$ if and only if $\beta(z, u_i) = 0$ for all $z\in D_{\sigma(i)}$ since $u_i\in \mathrm{Soc}(\Pi(S)e_i)$. Hence without loss of generality we can assume that $z\in D_{\sigma(i)}$. Also recall that $y_{\sigma(\alpha)}^1 = 1_F$, and therefore $zy_{\sigma(\alpha^*)}^1 = y_{\sigma(\alpha^*)}^1 z'$ and $z' y_{\sigma(\alpha)}^1 = y_{\sigma(\alpha)}^1z$ for some $z'\in D$. The computation
	\begin{equation}\label{eq - beta and f=1}
		\begin{aligned}
			\beta(z, u_i) &= \beta(z, y_{\sigma(\alpha^*)}^1v_2) = \beta(zy_{\sigma(\alpha^*)}^1, v_2) = \beta(y_{\sigma(\alpha^*)}^1 z', v_2) = \beta(y_{\sigma(\alpha^*)}^1, z'v_2) = \\
			&= \begin{cases}
				\beta(z'v_2, y_{\alpha^*}^1), & \mbox{if }\sigma(\alpha^*)\in Q_1 \\
				\beta(z'v_2, f\sgn(\alpha^*)y_{\alpha^*}^1), & \mbox{if }\sigma(\alpha^*)\not\in Q_1
			\end{cases} = \\
			&= \begin{cases}
				\beta(z', -\sgn(\alpha^*) u_j), & \mbox{if }\sigma(\alpha^*)\in Q_1 \\
				\beta(z', u_jf), & \mbox{if }\sigma(\alpha^*)\not\in Q_1
			\end{cases} = \\
			&= \begin{cases}
				-\sgn(\alpha^*)\beta(z', y_{\sigma(\alpha)}^1 v_1), & \mbox{if }\sigma(\alpha^*)\in Q_1 \\
				\beta(z', y_{\sigma(\alpha)}^1 v_1f), & \mbox{if }\sigma(\alpha^*)\not\in Q_1
			\end{cases} = \\
			&= \begin{cases}
				-\sgn(\alpha^*)\beta(y_{\sigma(\alpha)}^1, zv_1), & \mbox{if }\sigma(\alpha^*)\in Q_1 \\
				\beta(y_{\sigma(\alpha)}^1, zv_1f), & \mbox{if }\sigma(\alpha^*)\not\in Q_1
			\end{cases} = \\
			&= \begin{cases}
				-\sgn(\alpha^*)\beta(zv_1, \sgn(\alpha)fy_{\alpha}^1), & \mbox{if }\sigma(\alpha^*)\in Q_1 \\
				\beta(zv_1f, y_{\alpha}^1), & \mbox{if }\sigma(\alpha^*)\not\in Q_1
			\end{cases} = \\
			&= \begin{cases}
				-\sgn(\alpha^*)\sgn(\alpha)\beta(z, v_1y_{\alpha}^1f), & \mbox{if }\sigma(\alpha^*)\in Q_1 \\
				\beta(z, v_1y_{\alpha}^1f), & \mbox{if }\sigma(\alpha^*)\not\in Q_1
			\end{cases} = \\
			&= \beta(z, u_if)
		\end{aligned}
	\end{equation}
	shows that $f=1$. It is left to show that the claim holds in each case.
	
	In the simply laced case we can use the calculations done in \cite{brenner2002periodic}. Let $S$ be a species of type ADE such that $\Pi(S) = \K Q$ and let $\alpha:i\to j\in \overline{Q}_1$. In the proof of \cite[Proposition 4.5 (b)]{brenner2002periodic} where they computed the constant $C=-1$, they chose $u_i$ and $u_j$ such that $v_1 y_\alpha^1 = u_i = y_{\sigma(\alpha^*)}^1 v_2$ and $y_{\sigma(\alpha)}^1 v_1 = u_j = C\sgn(\alpha^*)\sgn(\sigma(\alpha^*))v_2 y_{\alpha^*}^1 = -\sgn(\alpha^*)\sgn(\sigma(\alpha^*))v_2 y_{\alpha^*}^1$ for some elements $v_1, v_2\in \Pi(S)$. The sign $-\sgn(\alpha^*)\sgn(\sigma(\alpha^*))$ only depends on the different signs appearing in front of terms in the Casimir element $c$. Since the orientation of $Q$ decides the signs appearing in $c$ we can use the same computations to prove $u_j = -\sgn(\alpha^*)\sgn(\sigma(\alpha^*))v_2 y_{\alpha^*}^1$ for arbitrary species of type ADE. Thus the claim holds.
	
	Before we prove the claim for the non-simply laced cases we will introduce some notation. Recall that $y_\alpha^1 = 1_F$ for all $\alpha\in \overline{Q}$. Fix $\alpha\in \overline{Q}_1$ such that $D_{s(\alpha)} = G$ and assume that $2\dim_\K F = \dim_\K G$. Let $a, a'\in G$ be such that $y_\alpha^2 = y_\alpha^1 a$ and $y_{\alpha^*}^2 = a' y_{\alpha^*}^1$. By \cite[Lemma 1.1]{Dlab_1980} the Casimir elements $c_\alpha$ and $c_{\alpha^*}$ does not depend on the chose of basis. In particular,
	\begin{equation}\label{eq - casimir elements in bgl 1}
		c_\alpha = y_\alpha^1 \otimes_G y_{\alpha^*}^1 = y_\alpha^2 \otimes_G y_{\alpha^*}^2 = y_\alpha^1 a \otimes_G a' y_{\alpha^*}^1 = y_\alpha^1 aa' \otimes_G y_{\alpha^*}^1
	\end{equation}
	and thus $a' = a^{-1}$. Note that
	\begin{equation}\label{eq - casimir elements in bgl 2}
		\begin{aligned}
			ac_{\alpha^*} a^{-1} &= a(y_{\alpha^*}^1\otimes_{D_{t(\alpha)}} y_\alpha^1 + y_{\alpha^*}^2\otimes_{D_{t(\alpha)}} y_\alpha^2)a^{-1} = \\
			&= a(y_{\alpha^*}^1\otimes_{D_{t(\alpha)}} y_\alpha^1 + a^{-1}y_{\alpha^*}^1\otimes_{D_{t(\alpha)}} y_\alpha^1 a)a^{-1} = \\
			&= a y_{\alpha^*}^1\otimes_{D_{t(\alpha)}} y_\alpha^1a^{-1} + y_{\alpha^*}^1\otimes_{D_{t(\alpha)}} y_\alpha^1.
		\end{aligned}
	\end{equation}
	We define $y_\alpha^0 = y_\alpha^1a^{-1}$ and $y_{\alpha^*}^0 = ay_{\alpha^*}^1$ so we can rewrite \eqref{eq - casimir elements in bgl 2} as
	\begin{equation}\label{eq - casimir elements in bgl 3}
		ac_{\alpha^*}a^{-1} = y_{\alpha^*}^1\otimes_{D_{t(\alpha)}} y_\alpha^1 + y_{\alpha^*}^0\otimes_{D_{t(\alpha)}} y_\alpha^0.
	\end{equation}
	Also note that
	\begin{equation}\label{eq - casimir elements in bgl 4}
		y_\alpha^2\otimes_G y_{\alpha^*}^1 = y_\alpha^1\otimes_G y_{\alpha^*}^0, \quad y_\alpha^1\otimes_G y_{\alpha^*}^2 = y_\alpha^0\otimes_G y_{\alpha^*}^1.
	\end{equation}
	
	We want to represent the elements in $\Pi(S)$ in the \cite{baer1987preprojective} perspective and thus we introduce the following notation:
	\begin{enumerate}
		\item The element $y_\alpha^1$ will be represented by an undecorated edge from $s(\alpha)$ to $t(\alpha)$. 
		\item The element $y_\alpha^2$ will be represented by an edge from $s(\alpha)$ to $t(\alpha)$ decorated with a black circle.
		\item The element $y_\alpha^0$ will be represented by an edge from $s(\alpha)$ to $t(\alpha)$ decorated with a red circle.
		\item The element $y_\alpha^3$ will be represented by an edge from $s(\alpha)$ to $t(\alpha)$ decorated with disk.
	\end{enumerate}
	In this notation we can rewrite \eqref{eq - casimir elements in bgl 1} and $c_{\alpha^*}$ as
	\begin{equation*}
		c_\alpha = \left(\vcenter{\hbox{\begin{tikzpicture}[scale=0.9]
			\node at (0, 1) {$s(\alpha)$};
			\node at (0, 2) {$t(\alpha)$};
			
			\draw (0.5,2) -- (1,1) -- (1.5,2);
		\end{tikzpicture} }}\right) = \left(\vcenter{\hbox{\begin{tikzpicture}[scale=0.9]
			\node at (0, 1) {$s(\alpha)$};
			\node at (0, 2) {$t(\alpha)$};
			
			\draw (0.5,2) -- node (A1) {} (1,1) -- node (A2) {} (1.5,2);
			
			\filldraw[fill = white] (A1) circle (0.1);
			\filldraw[fill = white] (A2) circle (0.1);
		\end{tikzpicture} }}\right), \quad c_{\alpha^*} = \left(\vcenter{\hbox{\begin{tikzpicture}[scale=0.9]
			\node at (0, 1) {$s(\alpha)$};
			\node at (0, 2) {$t(\alpha)$};
			
			\draw (0.5,1) -- (1,2) -- (1.5,1);
		\end{tikzpicture} }}\right) + \left(\vcenter{\hbox{\begin{tikzpicture}[scale=0.9]
			\node at (0, 1) {$s(\alpha)$};
			\node at (0, 2) {$t(\alpha)$};
			
			\draw (0.5,1) -- node (A1) {} (1,2) -- node (A2) {} (1.5,1);
			
			\filldraw[fill = white] (A1) circle (0.1);
			\filldraw[fill = white] (A2) circle (0.1); 
		\end{tikzpicture} }}\right)
	\end{equation*}
	\eqref{eq - casimir elements in bgl 3} as
	\begin{equation*}
		ac_{\alpha^*}a^{-1} = \left(\vcenter{\hbox{\begin{tikzpicture}[scale=0.9]
			\node at (0, 1) {$s(\alpha)$};
			\node at (0, 2) {$t(\alpha)$};
			
			\draw (0.5,1) -- (1,2) -- (1.5,1);
		\end{tikzpicture} }}\right) + \left(\vcenter{\hbox{\begin{tikzpicture}[scale=0.9]
			\node at (0, 1) {$s(\alpha)$};
			\node at (0, 2) {$t(\alpha)$};
			
			\draw (0.5,1) -- node (A1) {} (1,2) -- node (A2) {} (1.5,1);
			
			\filldraw[color = red, fill = white] (A1) circle (0.1);
			\filldraw[color = red, fill = white] (A2) circle (0.1); 
		\end{tikzpicture} }}\right)
	\end{equation*}
	and \eqref{eq - casimir elements in bgl 4} as
	\begin{equation*}
		\left(\vcenter{\hbox{\begin{tikzpicture}[scale=0.9]
			\node at (0, 1) {$s(\alpha)$};
			\node at (0, 2) {$t(\alpha)$};
			
			\draw (0.5,2) -- (1,1) -- node (A1) {} (1.5,2);
			
			\filldraw[fill = white] (A1) circle (0.1);
		\end{tikzpicture} }}\right) = \left(\vcenter{\hbox{\begin{tikzpicture}[scale=0.9]
			\node at (0, 1) {$s(\alpha)$};
			\node at (0, 2) {$t(\alpha)$};
			
			\draw (0.5,2) -- node (A2) {} (1,1) -- (1.5,2);
			
			\filldraw[color = red, fill = white] (A2) circle (0.1);
		\end{tikzpicture} }}\right), \quad \left(\vcenter{\hbox{\begin{tikzpicture}[scale=0.9]
			\node at (0, 1) {$s(\alpha)$};
			\node at (0, 2) {$t(\alpha)$};
			
			\draw (0.5,2) -- node (A1) {} (1,1) -- (1.5,2);
			
			\filldraw[fill = white] (A1) circle (0.1);
		\end{tikzpicture} }}\right) = \left(\vcenter{\hbox{\begin{tikzpicture}[scale=0.9]
			\node at (0, 1) {$s(\alpha)$};
			\node at (0, 2) {$t(\alpha)$};
			
			\draw (0.5,2) -- (1,1) -- node (A2) {} (1.5,2);
			
			\filldraw[color = red, fill = white] (A2) circle (0.1);
		\end{tikzpicture} }}\right).
	\end{equation*}
	\begin{myrem}
		Note that we will only use $(1)$ and $(2)$ for all of the non-simply laced cases, $(3)$ will be used for all of the non-simply laced cases except for type $G$ and $(4)$ will only be used in type G.
	\end{myrem}
	For example, if $S$ is a species over $Q: 1\xrightarrow{\alpha}2\xrightarrow{\beta}3$, then the element $y_\beta^1y_{\beta^*}^1y_{\beta}^1y_\alpha^1$ will be represented by
	\begin{equation*}
		\dpath{3}{1,2,3,2,3}{}{}{},
	\end{equation*}
	where the numbers on the left hand side indicates the vertices in $Q_0$. If the element $y_\beta^0y_{\beta^*}^1y_{\beta}^3y_\alpha^2$ would exist it would be represented by
	\begin{equation*}
		\dpath{3}{1,2,3,2,3}{1}{4}{2}.
	\end{equation*}
	
	In all of the non-simply laced cases we label the vertices in the same way as in Figure~\ref{Figure - Dynkin Diagrams 1}.
	
	Let $S$ be a species of type B. Computing $e_kce_k$, for all $k\in \{1, 2, \dots, n\}$, under the isomorphism in Proposition~\ref{Proposition - explcit desc of iso between pi(lambda) and pi(S)} we have, in our notation, that the relations in $\Pi(S)$ are given by
	\begin{equation*}
		\dpath{2}{1,2,1}{}{}{} + \dpath{2}{1,2,1}{1,2}{}{} = 0, \quad \left(\vcenter{\hbox{\begin{tikzpicture}[scale=0.9]
			\node at (0, 1) {$k$};
			\node at (0, 2) {$k+1$};
			\node at (0, 3) {$k+2$};
			
			\draw (0.5,2) -- (1,1) -- (1.5,2);
		\end{tikzpicture} }}\right) = \epsilon_k \left(\vcenter{\hbox{\begin{tikzpicture}[scale=0.9]
			\node at (0, 1) {$k$};
			\node at (0, 2) {$k+1$};
			\node at (0, 3) {$k+2$};
			
			\draw (0.5,2) -- (1,3) -- (1.5,2);
		\end{tikzpicture} }}\right), \quad \left(\vcenter{\hbox{\begin{tikzpicture}[scale=0.9]
			\node at (0.2, 1) {$n-1$};
			\node at (0.2, 2) {$n$};
			
			\draw (0.5,2) -- (1,1) -- (1.5,2);
		\end{tikzpicture} }}\right) = 0
	\end{equation*}
	where $\epsilon_k$ is a sign that depends on the orientation of $Q$. Let $\alpha:n\to n-1\in \overline{Q}_1$. Consider the non-zero elements
	\begin{equation*}
		u_n = \left(\vcenter{\hbox{\begin{tikzpicture}[scale=0.9]
			\node at (0.2, 1) {1};
			\node at (0.2, 2) {2};
			\node at (0.2, 3) {3};
			\node at (0.2, 4) {$n-2$};
			\node at (0.2, 5) {$n-1$};
			\node at (0.2, 6) {$n$};
			
			\draw (0.5, 6) -- (1,5) -- (1.5, 4);
			\draw[dotted] (1.5,4) -- (2,3);
			\draw (2,3) -- (2.5,2) -- (3, 1) -- node (A1) {} (3.5,2) -- (4,3);
			\draw[dotted] (4,3) -- (4.5,4);
			\draw (4.5,4) -- (5, 5) -- (5.5, 6);
			
			\filldraw[fill=white] (A1) circle (0.1);
		\end{tikzpicture} }}\right), \quad u_{n-1} = \left(\vcenter{\hbox{\begin{tikzpicture}[scale=0.9]
			\node at (0.2, 1) {1};
			\node at (0.2, 2) {2};
			\node at (0.2, 3) {3};
			\node at (0.2, 4) {$n-2$};
			\node at (0.2, 5) {$n-1$};
			\node at (0.2, 6) {$n$};
			
			\draw (1,5) -- (1.5, 4);
			\draw[dotted] (1.5,4) -- (2,3);
			\draw (2,3) -- (2.5,2) -- (3, 1) -- node (A1) {} (3.5,2) -- (4,3);
			\draw[dotted] (4,3) -- (4.5,4);
			\draw (4.5,4) -- (5, 5) -- (5.5, 6) -- (6, 5);
			
			\filldraw[fill=white] (A1) circle (0.1);
		\end{tikzpicture} }}\right)
	\end{equation*}
	in $\mathrm{Soc}(\Pi(S))$.
	\begin{figure}
		\begin{equation*}
			\begin{aligned}
				u_{n-1} &= \left(\vcenter{\hbox{\begin{tikzpicture}[scale=0.9]
					\node at (0.2, 1) {1};
					\node at (0.2, 2) {2};
					\node at (0.2, 3) {3};
					\node at (0.2, 4) {$n-2$};
					\node at (0.2, 5) {$n-1$};
					\node at (0.2, 6) {$n$};
					
					\draw (1,5) -- (1.5, 4);
					\draw[dotted] (1.5,4) -- (2,3);
					\draw (2,3) -- (2.5,2) -- (3, 1) -- node (A1) {} (3.5,2) -- (4,3);
					\draw[dotted] (4,3) -- (4.5,4);
					\draw (4.5,4) -- (5, 5) -- (5.5, 6) -- (6, 5);
					
					\filldraw[fill=white] (A1) circle (0.1);
				\end{tikzpicture} }}\right) = \epsilon\left(\vcenter{\hbox{\begin{tikzpicture}[scale=0.9]
					\node at (0.2, 1) {1};
					\node at (0.2, 2) {2};
					\node at (0.2, 3) {3};
					\node at (0.2, 4) {$n-2$};
					\node at (0.2, 5) {$n-1$};
					\node at (0.2, 6) {$n$};
					
					\draw (1,5) -- (1.5, 4);
					\draw[dotted] (1.5,4) -- (2,3);
					\draw (2,3) -- (2.5,2) -- (3,1) -- node (A1) {} (3.5,2) -- (4,1) -- (4.5, 2) -- (5, 3);
					\draw[dotted] (5,3) -- (5.5,4);
					\draw (5.5, 4) -- (6, 5);
					
					\filldraw[fill=white] (A1) circle (0.1);
				\end{tikzpicture} }}\right) = \\
				&= \epsilon\left(\vcenter{\hbox{\begin{tikzpicture}[scale=0.9]
					\node at (0.2, 1) {1};
					\node at (0.2, 2) {2};
					\node at (0.2, 3) {3};
					\node at (0.2, 4) {$n-2$};
					\node at (0.2, 5) {$n-1$};
					\node at (0.2, 6) {$n$};
					
					\draw (1,5) -- (1.5, 4);
					\draw[dotted] (1.5,4) -- (2,3);
					\draw (2,3) -- (2.5,2) -- node (A1) {} (3,1) -- (3.5,2) -- (4,1) -- (4.5, 2) -- (5, 3);
					\draw[dotted] (5,3) -- (5.5,4);
					\draw (5.5, 4) -- (6, 5);
					
					\filldraw[color = red, fill=white] (A1) circle (0.1);
				\end{tikzpicture} }}\right) = - \epsilon\left(\vcenter{\hbox{\begin{tikzpicture}[scale=0.9]
					\node at (0.2, 1) {1};
					\node at (0.2, 2) {2};
					\node at (0.2, 3) {3};
					\node at (0.2, 4) {$n-2$};
					\node at (0.2, 5) {$n-1$};
					\node at (0.2, 6) {$n$};
					
					\draw (1,5) -- (1.5, 4);
					\draw[dotted] (1.5,4) -- (2,3);
					\draw (2,3) -- (2.5,2) -- node (A1) {} (3,1) -- node (A2) {} (3.5,2) -- node (A3) {} (4,1) -- (4.5, 2) -- (5, 3);
					\draw[dotted] (5,3) -- (5.5,4);
					\draw (5.5, 4) -- (6, 5);
					
					\filldraw[color = red, fill=white] (A1) circle (0.1);
					\filldraw[color = red, fill=white] (A2) circle (0.1);
					\filldraw[color = red, fill=white] (A3) circle (0.1);
				\end{tikzpicture} }}\right) = \\
				&= -\epsilon\left(\vcenter{\hbox{\begin{tikzpicture}[scale=0.9]
					\node at (0.2, 1) {1};
					\node at (0.2, 2) {2};
					\node at (0.2, 3) {3};
					\node at (0.2, 4) {$n-2$};
					\node at (0.2, 5) {$n-1$};
					\node at (0.2, 6) {$n$};
					
					\draw (1,5) -- (1.5, 4);
					\draw[dotted] (1.5,4) -- (2,3);
					\draw (2,3) -- (2.5,2) -- (3,1) -- (3.5,2) -- node (A1) {} (4,1) -- (4.5, 2) -- (5, 3);
					\draw[dotted] (5,3) -- (5.5,4);
					\draw (5.5, 4) -- (6, 5);
					
					\filldraw[color = red, fill=white] (A1) circle (0.1);
				\end{tikzpicture} }}\right) = -\epsilon\left(\vcenter{\hbox{\begin{tikzpicture}[scale=0.9]
					\node at (0.2, 1) {1};
					\node at (0.2, 2) {2};
					\node at (0.2, 3) {3};
					\node at (0.2, 4) {$n-2$};
					\node at (0.2, 5) {$n-1$};
					\node at (0.2, 6) {$n$};
					
					\draw (1,5) -- (1.5, 4);
					\draw[dotted] (1.5,4) -- (2,3);
					\draw (2,3) -- (2.5,2) -- (3,1) -- (3.5,2) -- (4,1) -- node (A1) {} (4.5, 2) -- (5, 3);
					\draw[dotted] (5,3) -- (5.5,4);
					\draw (5.5, 4) -- (6, 5);
					
					\filldraw[fill=white] (A1) circle (0.1);
				\end{tikzpicture} }}\right) = \\
				&= -\left(\vcenter{\hbox{\begin{tikzpicture}[scale=0.9]
					\node at (0.2, 1) {1};
					\node at (0.2, 2) {2};
					\node at (0.2, 3) {3};
					\node at (0.2, 4) {$n-2$};
					\node at (0.2, 5) {$n-1$};
					\node at (0.2, 6) {$n$};
					
					\draw (1, 5) -- (1.5, 6) -- (2,5) -- (2.5, 4);
					\draw[dotted] (2.5,4) -- (3,3);
					\draw (3,3) -- (3.5,2) -- (4,1) -- node (A1) {} (4.5, 2) -- (5, 3);
					\draw[dotted] (5,3) -- (5.5,4);
					\draw (5.5, 4) -- (6, 5);
					
					\filldraw[fill=white] (A1) circle (0.1);
				\end{tikzpicture} }}\right)
			\end{aligned}
		\end{equation*}\caption{Computation for $f=1$ in type B}\label{Figure - computation type B}
	\end{figure}
	Reading from Figure~\ref{Figure - computation type B}, where $\epsilon = \prod_{k=1}^{n-2}\epsilon_k$, we get that
	\begin{equation*}
		y_{\alpha^*}^1 v_1 = u_n = v_2 y_\alpha^1, \quad y_{\alpha}^1 v_1 = u_{n-1} = -v_2 y_{\alpha^*}^1
	\end{equation*}
	for some elements $v_1,v_2\in \Pi(S)$. Hence $f = 1$.
	
	Let $S$ be a species of type C. Similarly as in type B the elements $e_kce_k$ for $k\in \{2, 3, \dots, n\}$ give rise to the relations in $\Pi(S)$ given by
	\begin{equation*}
		\begin{aligned}
			\dpath{2}{1,2,1}{}{}{} = 0, \quad &\left(\vcenter{\hbox{\begin{tikzpicture}[scale=0.9]
				\node at (0.2, 1) {1};
				\node at (0.2, 2) {2};
				\node at (0.2, 3) {3};
				
				\draw (0.5,2) -- (1,1) -- (1.5,2);
			\end{tikzpicture} }}\right) + \left(\vcenter{\hbox{\begin{tikzpicture}[scale=0.9]
				\node at (0.2, 1) {1};
				\node at (0.2, 2) {2};
				\node at (0.2, 3) {3};
				
				\draw (0.5,2) -- node (A1) {} (1,1) -- node (A2) {} (1.5,2);
				
				\filldraw[fill = white] (A1) circle (0.1);
				\filldraw[fill = white] (A2) circle (0.1);
			\end{tikzpicture} }}\right) = \epsilon_1 \left(\vcenter{\hbox{\begin{tikzpicture}[scale=0.9]
				\node at (0.2, 1) {1};
				\node at (0.2, 2) {2};
				\node at (0.2, 3) {3};
				
				\draw (0.5,2) -- (1,3) -- (1.5,2);
			\end{tikzpicture} }}\right), \\ 
			\left(\vcenter{\hbox{\begin{tikzpicture}[scale=0.9]
				\node at (0.1, 1) {$k$};
				\node at (0.1, 2) {$k+1$};
				\node at (0.1, 3) {$k+2$};
				
				\draw (0.5,2) -- (1,1) -- (1.5,2);
			\end{tikzpicture} }}\right) &= \epsilon_k \left(\vcenter{\hbox{\begin{tikzpicture}[scale=0.9]
				\node at (0.1, 1) {$k$};
				\node at (0.1, 2) {$k+1$};
				\node at (0.1, 3) {$k+2$};
				
				\draw (0.5,2) -- (1,3) -- (1.5,2);
			\end{tikzpicture} }}\right), \quad \left(\vcenter{\hbox{\begin{tikzpicture}[scale=0.9]
				\node at (0.2, 1) {$n-1$};
				\node at (0.2, 2) {$n$};
				
				\draw (0.5,2) -- (1,1) -- (1.5,2);
			\end{tikzpicture} }}\right) = 0
		\end{aligned}
	\end{equation*}
	where $\epsilon_k$ is a sign that depends on the orientation of $Q$. Similarly to the type B case we assume that $\alpha:n\to n-1\in \overline{Q}_1$ and consider the non-zero elements
	\begin{equation*}
		u_n = \left(\vcenter{\hbox{\begin{tikzpicture}[scale=0.9]
			\node at (0.2, 1) {1};
			\node at (0.2, 2) {2};
			\node at (0.2, 3) {3};
			\node at (0.2, 4) {$n-2$};
			\node at (0.2, 5) {$n-1$};
			\node at (0.2, 6) {$n$};
			
			\draw (0.5, 6) -- (1,5) -- (1.5, 4);
			\draw[dotted] (1.5,4) -- (2,3);
			\draw (2,3) -- (2.5,2) -- (3, 1) -- (3.5,2) -- (4,3);
			\draw[dotted] (4,3) -- (4.5,4);
			\draw (4.5,4) -- (5, 5) -- (5.5, 6);
		\end{tikzpicture} }}\right), \quad u_{n-1} = \left(\vcenter{\hbox{\begin{tikzpicture}[scale=0.9]
			\node at (0.2, 1) {1};
			\node at (0.2, 2) {2};
			\node at (0.2, 3) {3};
			\node at (0.2, 4) {$n-2$};
			\node at (0.2, 5) {$n-1$};
			\node at (0.2, 6) {$n$};
			
			\draw (1,5) -- (1.5, 4);
			\draw[dotted] (1.5,4) -- (2,3);
			\draw (2,3) -- (2.5,2) -- (3, 1) -- (3.5,2) -- (4,3);
			\draw[dotted] (4,3) -- (4.5,4);
			\draw (4.5,4) -- (5, 5) -- (5.5, 6) -- (6, 5);
		\end{tikzpicture} }}\right).
	\end{equation*}
	in $\mathrm{Soc}(\Pi(S))$. Let $\epsilon = \prod_{k=1}^{n-2} \epsilon_k$.
	\begin{figure}[h]
		\begin{equation*}
			\begin{aligned}
				u_{n-1} &= \left(\vcenter{\hbox{\begin{tikzpicture}[scale=0.9]
					\node at (0.5, 1) {1};
					\node at (0.5, 2) {2};
					\node at (0.5, 3) {3};
					\node at (0.5, 4) {$n-2$};
					\node at (0.5, 5) {$n-1$};
					\node at (0.5, 6) {$n$};
					
					\draw (1,5) -- (1.5, 4);
					\draw[dotted] (1.5,4) -- (2,3);
					\draw (2,3) -- (2.5,2) -- (3, 1) -- (3.5,2) -- (4,3);
					\draw[dotted] (4,3) -- (4.5,4);
					\draw (4.5,4) -- (5, 5) -- (5.5, 6) -- (6, 5);
				\end{tikzpicture} }}\right) = \epsilon\left(\vcenter{\hbox{\begin{tikzpicture}[scale=0.9]
					\node at (0.5, 1) {1};
					\node at (0.5, 2) {2};
					\node at (0.5, 3) {3};
					\node at (0.5, 4) {$n-2$};
					\node at (0.5, 5) {$n-1$};
					\node at (0.5, 6) {$n$};
					
					\draw (1,5) -- (1.5, 4);
					\draw[dotted] (1.5,4) -- (2,3);
					\draw (2,3) -- (2.5,2) -- (3,1) -- (3.5,2) -- node (A1) {} (4,1) -- node (A2) {} (4.5, 2) -- (5, 3);
					\draw[dotted] (5,3) -- (5.5,4);
					\draw (5.5, 4) -- (6, 5);
					
					\filldraw[fill=white] (A1) circle (0.1);
					\filldraw[fill=white] (A2) circle (0.1);
				\end{tikzpicture} }}\right) = \\
				&= \left(\vcenter{\hbox{\begin{tikzpicture}[scale=0.9]
					\node at (0.5, 1) {1};
					\node at (0.5, 2) {2};
					\node at (0.5, 3) {3};
					\node at (0.5, 4) {$n-2$};
					\node at (0.5, 5) {$n-1$};
					\node at (0.5, 6) {$n$};
					
					\draw (1, 5) -- (1.5, 6) -- (2,5) -- (2.5, 4);
					\draw[dotted] (2.5,4) -- (3,3);
					\draw (3,3) -- (3.5,2) -- node (A2) {} (4,1) -- node (A1) {} (4.5, 2) -- (5, 3);
					\draw[dotted] (5,3) -- (5.5,4);
					\draw (5.5, 4) -- (6, 5);
					
					\filldraw[fill=white] (A1) circle (0.1);
					\filldraw[fill=white] (A2) circle (0.1);
				\end{tikzpicture} }}\right) = - \left(\vcenter{\hbox{\begin{tikzpicture}[scale=0.9]
					\node at (0.5, 1) {1};
					\node at (0.5, 2) {2};
					\node at (0.5, 3) {3};
					\node at (0.5, 4) {$n-2$};
					\node at (0.5, 5) {$n-1$};
					\node at (0.5, 6) {$n$};
					
					\draw (1, 5) -- (1.5, 6) -- (2,5) -- (2.5, 4);
					\draw[dotted] (2.5,4) -- (3,3);
					\draw (3,3) -- (3.5,2) -- (4,1) -- (4.5, 2) -- (5, 3);
					\draw[dotted] (5,3) -- (5.5,4);
					\draw (5.5, 4) -- (6, 5);
				\end{tikzpicture} }}\right)
			\end{aligned}
		\end{equation*}\caption{Computation for $f=1$ in type C}\label{Figure - computation type C}
	\end{figure}
	The computation in Figure~\ref{Figure - computation type C} show that
	\begin{equation*}
		y_{\alpha^*}^1 v_1 = u_n = v_2 y_\alpha^1, \quad y_{\alpha}^1 v_1 = u_{n-1} = -v_2 y_{\alpha^*}^1
	\end{equation*}
	for some elements $v_1, v_2\in \Pi(S)$. Which in turn implies that $f=1$.
	
	Let $S$ be a species of type F. Similarly as in Type B, the preprojective algebra $\Pi(S)$ is described with the relations
	\begin{equation}\label{eq - type F case relations}
		\begin{aligned}
			\dpath{4}{4,3,4}{}{}{}=0, \quad &\dpath{4}{3,4,3}{}{}{} = \epsilon\dpath{4}{3,2,3}{}{}{}, \\
			\dpath{4}{2,3,2}{}{}{} + \dpath{4}{2,3,2}{1,2}{}{} &= \epsilon'\dpath{4}{2,1,2}{}{}{}, \quad \dpath{4}{1,2,1}{}{}{} = 0.
		\end{aligned}
	\end{equation}
	where $\epsilon$ and $\epsilon'$ are signs that depend on the orientation of $Q$. Let $\alpha: 4\to 3\in Q_1$ and choose the following non-zero elements
	\begin{equation*}
		u_3 = \dpath{4}{3,2,1,2,3,4,3,2,3,4,3}{8}{}{}, \quad 
		u_4 = \dpath{4}{4,3,2,1,2,3,4,3,2,3,4}{9}{}{}.
	\end{equation*}
	In this case it is not clear why $u_3$ and $u_4$ are non-zero and therefore we will begin by giving the outline of the proof why $u_3$ and $u_4$ are non-zero. The proof for $u_3$ is similar to the proof for $u_4$ and thus we only show that $u_4$ is non-zero. Using the relations \eqref{eq - type F case relations} we have four possible candidates for $u_4$ which is
	\begin{equation*}
		\begin{aligned}
			u_4^1 = \dpath{4}{4,3,2,1,2,3,4,3,2,3,4}{9}{}{},& \quad u_4^2 = \dpath{4}{4,3,2,1,2,3,4,3,2,3,4}{4}{}{} \\
			u^3_4 = \dpath{4}{4,3,2,1,2,3,4,3,2,3,4}{4,9}{}{},& \quad u_4^4 = \dpath{4}{4,3,2,1,2,3,4,3,2,3,4}{}{}{}.
		\end{aligned}
	\end{equation*}
	First we note that
	\begin{equation*}
		u_4^2 = \dpath{4}{4,3,2,1,2,3,4,3,2,3,4}{4}{}{} = \dpath{4}{4,3,2,1,2,3,4,3,4,3,4}{4}{}{} = 0
	\end{equation*}
	and similarly $u_4^4 = 0$. 
	\begin{figure}[h]
		\begin{equation*}
			\begin{aligned}
				u^3_4 &= \dpath{4}{4,3,2,1,2,3,4,3,2,3,4}{4,9}{}{} = \dpath{4}{4,3,2,1,2,3,4,3,2,3,4}{5,9}{}{} = \\
				&= \epsilon'\dpath{4}{4,3,2,3,2,3,4,3,2,3,4}{5,9}{}{} + \epsilon'\dpath{4}{4,3,2,3,2,3,4,3,2,3,4}{3,9}{}{} = \\
				&= \epsilon'\dpath{4}{4,3,4,3,2,3,4,3,2,3,4}{5,9}{}{} + \epsilon'\dpath{4}{4,3,2,3,4,3,4,3,2,3,4}{3,9}{}{} = 0
			\end{aligned}
		\end{equation*}\caption{Computation for $u_4^3 = 0$}\label{Figure - u_4^3 = 0}
	\end{figure}
	Reading from Figure~\ref{Figure - u_4^3 = 0} we have that $u_4^3 = 0$. The computation in Figure~\ref{Figure - computation type F} shows that
	\begin{equation*}
		u_3 = - \dpath{4}{3,4,3,2,1,2,3,4,3,2,3}{10}{}{}.
	\end{equation*}
	\begin{figure}
		\begin{equation*}
			\begin{aligned}
				u_3 &= \dpath{4}{3,2,1,2,3,4,3,2,3,4,3}{8}{}{} = \epsilon'\dpath{4}{3,2,3,2,3,4,3,2,3,4,3}{2,3,8}{}{} = \\
				&= \epsilon' \dpath{4}{3,2,3,2,3,2,3,2,3,2,3}{2,3,8}{}{} = \dpath{4}{3,2,3,2,3,2,1,2,3,2,3}{2,3,8}{}{} = \\
				&= \dpath{4}{3,2,3,2,3,2,1,2,3,2,3}{2}{4, 5}{} = -\dpath{4}{3,2,3,2,3,2,1,2,3,2,3}{2}{}{} = \\
				&= \dpath{4}{3,2,3,2,3,2,1,2,3,2,3}{2}{8,9}{} = \dpath{4}{3,2,3,2,3,2,1,2,3,2,3}{2,5,10}{}{} = \\
				&= \epsilon' \dpath{4}{3,2,3,2,3,2,3,2,3,2,3}{2,5,10}{}{} = \epsilon \epsilon' \dpath{4}{3,2,3,2,3,2,3,4,3,2,3}{2,5,10}{}{} = \\
				&= \epsilon \epsilon' \dpath{4}{3,2,3,2,3,2,3,4,3,2,3}{2,10}{6}{} = \epsilon \dpath{4}{3,2,3,2,1,2,3,4,3,2,3}{2,10}{6}{} = \\
				&= -\epsilon \dpath{4}{3,2,3,2,1,2,3,4,3,2,3}{10}{}{} = -\dpath{4}{3,4,3,2,1,2,3,4,3,2,3}{10}{}{}
			\end{aligned}
		\end{equation*}\caption{Computation for $f=1$ in type F}\label{Figure - computation type F}
	\end{figure}
	Thus we can conclude that
	\begin{equation*}
		y_{\alpha^*}^1 v_1 = u_4 = v_2 y_\alpha^1, \quad y_{\alpha}^1 v_1 = u_3 = -v_2 y_{\alpha^*}^1
	\end{equation*}
	for some elements $v_1, v_2\in \Pi(S)$. Hence $f=1$.
	
	Let $S$ be a species of type G. Similarly as in type B, its preprojective algebra $\Pi(S)$ is described using the relations
	\begin{equation*}
		\dpath{2}{1,2,1}{}{}{} + \dpath{2}{1,2,1}{1,2}{}{} + \dpath{2}{1,2,1}{}{}{1,2} = 0, \quad \dpath{2}{2,1,2}{}{}{} = 0.
	\end{equation*}
	Let $\alpha:1\to 2$ and choose the non-zero elements
	\begin{equation*}
		u_1 = \dpath{2}{1,2,1,2,1}{1,2}{}{}, \quad u_2 = \dpath{2}{2,1,2,1,2}{2,3}{}{}.
	\end{equation*}
	Similarly as in the previous cases the computation
	\begin{equation*}
		u_1 = \dpath{2}{1,2,1,2,1}{1,2}{}{} = - \dpath{2}{1,2,1,2,1}{1,2}{}{3,4} = \dpath{2}{1,2,1,2,1}{}{}{3,4} = -\dpath{2}{1,2,1,2,1}{3,4}{}{}
	\end{equation*}
	shows that
	\begin{equation*}
		y_{\alpha^*}^1 v_1 = u_1 = v_2 y_\alpha^1, \quad y_{\alpha}^1 v_1 = u_2 = -v_2 y_{\alpha^*}^1
	\end{equation*}
	for some elements $v_1, v_2\in \Pi(S)$ and thus $f=1$.
\end{proof}

\section{Koszul algebras}\label{Section - Koszul algebras}
The goal of this section is to show that the preprojective algebra of a representation finite species of Dynkin type $\Delta$ with Coxeter number $h$ is $(h-2, 2)$-almost Koszul. We will introduce the definition of the preprojective algebra using the derived category, which in turn will be used to compute the almost Koszul complex for $D_i\in \Pi(S)\mathrm{-mod}$.

\begin{mydef}\cite[Definition 1.1.2 and Definition 1.2.1]{MR1322847}
	Let $\Lambda=\bigoplus_{i\in \Z_{\ge 0}} \Lambda_i$ be a graded $\K$-algebra, where $\Lambda_0$ is semi-simple. We say that $\Lambda$ is a Koszul algebra, if there exists a graded exact complex
	\begin{equation*}
		\dots \rightarrow P_1\rightarrow P_0\rightarrow \Lambda_0 \rightarrow 0,
	\end{equation*}
	of $\Lambda$-modules where $P_i$ is projective and is generated by its component of degree $i$ for all $i\ge0$.
\end{mydef}

\begin{mylemma}\label{Corollary - T(S)xT(S) Koszul}
	Let $S$ be a species. Then $T(S)$ is a Koszul algebra.
\end{mylemma}

\begin{proof}
	Applying $-\otimes_{T(S)} T(S)_0$ to (\ref{eq - bimodule res of T(S)}) yields an exact sequence of $T(S)$-modules
	\begin{equation}\label{eq - Koszul complex for T(S)_0}
		\begin{aligned}
			P_\bullet: 0\xrightarrow{\quad} &\bigoplus_{\alpha\in Q_1}T(S)e_{t(\alpha)}\otimes_{D_{t(\alpha)}} M_\alpha \xrightarrow{\quad mult\quad }\bigoplus_{i\in Q_0}T(S)e_i\rightarrow T(S)_0\xrightarrow{\quad} 0.
		\end{aligned}
	\end{equation}
	This is the Koszul complex of $T(S)$, in other words, $T(S)$ is a Koszul algebra. 
\end{proof}

We will now shift our focus to the bounded derived category of $T(S)$. As we will see in this section, we can describe the preprojective algebra in the bounded derived category and use the theory of Auslander-Reiten triangles due to \cite{happel1987derived} to compute the almost Koszul complex for all simple $\Pi(S)$-modules. Using $\nu_1$ we have an alternate description of the preprojective algebra $\Pi(S)$ as follows
\begin{equation*}
	\Pi(S) = \bigoplus_{i\ge 0}\mathrm{Hom}_{\mathcal{D}^b(T(S)\mathrm{-mod})}(T(S), \nu_1^{-i}T(S)).
\end{equation*}
This definition is indeed equivalent to the definition given by \cite{baer1987preprojective}. This can be seen by noting that $T(S)\mathrm{-mod}$ can be viewed as a subcategory of $\mathcal{D}^b(T(S)\mathrm{-mod})$ and the fact that
\begin{equation*}
	\mathrm{Hom}_{\mathcal{D}^b(T(S)\mathrm{-mod})}(T(S), \nu_1^{-i}T(S)) \cong \mathrm{Hom}_{T(S)}(T(S), \tau^{-i}T(S))
\end{equation*}
for all $i\in \Z_{\ge 0}$. We denote the Auslander-Reiten quiver for $\mathcal{D}^b(T(S)\mathrm{-mod})$ by $\Gamma_{\mathcal{D}(S)}$. The relation between $T(S)\mathrm{-mod}$ and $\mathcal{D}^b(T(S)\mathrm{-mod})$ can be seen via $\Gamma_{\mathcal{D}(S)}$. More explicitly, we can describe $\Gamma_{\mathcal{D}(S)}$ using $\Gamma_S$ as follows. The vertex set is
\begin{equation*}
	(\Gamma_{\mathcal{D}(S)})_0 = \bigcup_{i\in \Z} (\Gamma_S[i])_0,
\end{equation*}
and the arrow set is
\begin{equation*}
	(\Gamma_{\mathcal{D}(S)})_1 = V\cup \bigcup_{i\in \Z} (\Gamma_S[i])_1,
\end{equation*}
where $V$ is a set of arrows connecting $\Gamma_S[i]$ with $\Gamma_S[i+1]$ for all $i\in \Z$. A description of $V$ will be given with the following example. Let $S$ be the species $\C \xrightarrow{\C}\R\xrightarrow{\R}\R$ over the quiver $1\rightarrow 2\rightarrow 3$. The quiver $\Gamma_{\mathcal{D}(S)}$ is
\begin{equation*}
	\begin{tikzcd}[column sep={25pt,between origins}]
		& & \cdots \arrow[dr] & & I_1[-1] \arrow[dr, "2", red] & & P_1 \arrow[dr, "2"] & & \nu_1^{-1}P_1 \arrow[dr, "2"] & & I_1 \arrow[dr, "2", red] & & P_1[1] \arrow[dr] \\
		& \cdots \arrow[dr] & & I_2[-1] \arrow[ur, "2"] \arrow[dr, red] & & P_2 \arrow[ur, "2"] \arrow[dr] & & \nu_1^{-1}P_2 \arrow[ur, "2"] \arrow[dr] & & I_2 \arrow[ur, "2"] \arrow[dr, red] & & P_2[1] \arrow[ur, "2"] \arrow[dr] & & \cdots \\
		& & I_3[-1] \arrow[ur] & & P_3 \arrow[ur] & & \nu_1^{-1} P_3 \arrow[ur] & & I_3 \arrow[ur] & & P_3[1] \arrow[ur] & & \cdots
	\end{tikzcd}
\end{equation*}
Here the red arrows are the arrows in $V$. Note that if there is an arrow $P_s\to P_t$ with valuation $d_{P_s, P_t}$, then there is an arrow $I_t[i]\to P_s[i+1]$ with valuation $d_{I_t[i], P_s[i+1]} = d_{P_s, P_t}$, for all $i\in \Z$ and these give all arrows in $V$.

The almost split sequences in $T(S)\mathrm{-mod}$ relate to Auslander-Reiten triangles in $\mathcal{D}^b(T(S)\mathrm{-mod})$ in the sense that for any almost split sequence
\begin{equation*}
	0\rightarrow M\rightarrow X\rightarrow N\rightarrow 0
\end{equation*}
in $T(S)\mathrm{-mod}$, there is an Auslander-Reiten triangle of the form
\begin{equation*}
	M\rightarrow X\rightarrow N\rightarrow M[1]
\end{equation*}
in $\mathcal{D}^b(T(S)\mathrm{-mod})$. We will use this relation between almost split sequences and Auslander-Reiten triangles to prove that representation finite species are twisted fractionally Calabi-Yau.

\begin{mylemma}\label{Lemma - l_i + l_sigma(i) = h}
	Let $S$ be a representation finite species of Dynkin type $\Delta$ with Coxeter number $h$. Then $l_i + l_{\sigma(i)} = h$ for all $i\in Q_0$. In particular, $\nu_1^{-h}P_i = P_i[2]$.
\end{mylemma}

\begin{proof}
	Let $a_i = l_i + l_{\sigma(i)}$. By Theorem~\ref{Theorem - Nakayama permutation} we know that $\sigma^2 = \id$. Thus $\nu_1^{-a_i}P_i = P_i[2]$ for all $i\in Q_0$. Now if there is an arrow from $i\to j$ in $Q$ then $\mathrm{Hom}_{T(S)}(P_j, P_i)\neq 0$, which means that
	\begin{equation*}
		0\neq\mathrm{Hom}_{T(S)}(\nu_1^{-a_ia_j}P_j, \nu_1^{-a_ia_j}P_i) = \mathrm{Hom}_{T(S)}(P_j, P_i[2(a_j - a_i)]),
	\end{equation*}
	and so $a_j = a_i$. Hence $a_i = a$ for all $i\in Q_0$ and for some constant $a\in \Z$. On the other hand, setting $n = |Q_0|$, we get
	\begin{equation}\label{eq - average length of tau orbit is h}
		na = \sum_{i\in Q_0} a_i = 2\sum_{i\in Q_0}l_i = 2n\frac{h}{2} = nh
	\end{equation}
	and therefore $a = h$. The third equality in (\ref{eq - average length of tau orbit is h}) holds since that there are $n\frac{h}{2}$ indecomposable $T(S)$-modules and that $l_i$ is the number of indecomposable $T(S)$-modules in the $\tau$-orbit of $P_i$.
\end{proof}

\begin{myprop}
	Let $S$ be a representation finite species of Dynkin type $\Delta$ with Coxeter number $h$. Then $T(S)$ is twisted $\frac{h-2}{h}$-Calabi-Yau.
\end{myprop}

\begin{proof}
	By \cite[Proposition 4.3]{herschend2011n} it is enough to show that $\nu^hT(S) \cong T(S)[h-2]$. This can be seen via the following computation
	\begin{equation*}
		\nu^{h}T(S) \cong (\nu_1\circ [1])^hT(S) = \nu_1^hT(S)[h] \simeq T(S)[h-2].
	\end{equation*}
	The third isomorphism holds by Lemma~\ref{Lemma - l_i + l_sigma(i) = h}.
\end{proof}

Next we show for every representation finite species $S$ that all paths in $\Gamma_{\mathcal{D}(S)}$ from $X$ to $Y$ have the same length $W(X, Y)$. To do this we first introduce integers $W(i, j)$ for all $i, j\in Q_0$. We will define $W$ in several steps.

First we define $W(\omega)$ where $\omega$ is a path in $\overline{Q}$. Given a path $\omega$ in $\overline{Q}$ we can split it into a product of arrows in $\overline{Q}_1$ as follows
\begin{equation*}
	\omega = \alpha_m\alpha_{m-1}\dots\alpha_1, \quad \alpha_i\in \overline{Q}_1.
\end{equation*}
Let us define
\begin{equation}\label{eq - W(omega)}
	W(\omega) = \sum_{i=1}^mW(\alpha_i), \quad W(\alpha_i) = \begin{cases}
		-1, & \mbox{if }\alpha_i\in Q_1 \\
		1, & \mbox{if }\alpha_i\not\in Q_1
	\end{cases}
\end{equation}

\begin{myprop}\label{Proposition - W(omega) = W(omega')}
	Let $Q$ be an acyclic quiver and let $\omega$ and $\omega'$ be two paths from $i$ to $j$ in $\overline{Q}$. Then $W(\omega) = W(\omega')$. Thus we write $W(i, j) = W(\omega)$.
\end{myprop}

\begin{proof}
	First we show that if $\omega_{\mathrm{cyc}}$ is a cycle in $\overline{Q}$ then $W(\omega_{\mathrm{cyc}}) = 0$. First we write
	\begin{equation*}
		\omega_{\mathrm{cyc}} = \alpha_m \alpha_{m-1}\cdots \alpha_1,
	\end{equation*}
	where $\alpha_i\in \overline{Q}_1$ for all $i\in \{1, \dots, m\}$. Since $Q$ is acyclic and without multiple arrows there exists $k\in \{1, \dots, m\}$ such that $\alpha_k = \alpha_{k-1}^*$. Thus
	\begin{equation*}
		W(\omega_{\mathrm{cyc}}) = W(\omega_{\mathrm{cyc}}'),
	\end{equation*}
	where
	\begin{equation*}
		\omega_{\mathrm{cyc}}' = \alpha_m \alpha_{m-1}\cdots \alpha_{k+1}\alpha_{k-2}\cdots \alpha_1.
	\end{equation*}
	Since $\omega_{\mathrm{cyc}}'$ is a cycle in $\overline{Q}$ we have that $W(\omega_{\mathrm{cyc}}) = 0$ by induction on the length of $\omega_{\mathrm{cyc}}$.
	
	Now let $\omega$ and $\omega'$ be path from $i$ to $j$ in $\overline{Q}$. Then $\omega^*\omega'$ is a cycle, and therefore $W(\omega^*\omega') = 0$. Rewriting the left hand side we get
	\begin{equation*}
		0 = W(\omega^*\omega') = W(\omega') + W(\omega^*) = W(\omega') - W(\omega). \hfill \qedhere
	\end{equation*}
\end{proof}

Since $T(S)$ is hereditary we know that the projective $T(S)$-modules in $\Gamma_{\mathcal{D}(S)}$ form a subquiver which is isomorphic to $Q^*$. So if there is a path in $\Gamma_{\mathcal{D}(S)}$ from $P_i$ to $P_j$, it must have length $W(i, j)$. We can define $W$ on all indecomposable objects in $\mathcal{D}^b(T(S)\mathrm{-mod})$ preserving this property by observing that any path in $\Gamma_{\mathcal{D}(S)}$ from $X$ to $\nu_1^{-1}X$ has length $2$, for all indecomposable objects $X\in \mathcal{D}^b(T(S)\mathrm{-mod})$. Let $X, Y\in \mathcal{D}^b(T(S)\mathrm{-mod})$. Then since $S$ is representation finite there exist $k_1, k_2\in \Z_{\ge 0}$ such that $X = \nu_1^{-k_1}P_i$ and $Y = \nu_1^{-k_2}P_j$ for some $i, j\in Q_0$. We define
\begin{equation*}
	W(X, Y) = 2(k_2 - k_1) + W(i, j).
\end{equation*}

Let $Z = \nu_1^{-k_3}P_m$ for some $m\in Q_0$. We claim that $W(X, Z) = W(X, Y) + W(Y, Z)$. Indeed, we have that
\begin{equation*}
	\begin{aligned}
		W(X, Z) &= 2(k_3 - k_1) + W(i, m) = \\
		&=2(k_2 - k_1) + W(i, j) + 2(k_3 - k_2) + W(j, m) + W(i, m) - W(j, m) - W(i, j) = \\
		&= W(X, Y) + W(Y, Z) + W(i, m) - W(j, m) - W(i, j).
	\end{aligned}
\end{equation*}
Now let $\omega$ be a path in $\overline{Q}$ from $i$ to $j$ and $\omega'$ be a path in $\overline{Q}$ from $j$ to $m$. Then $\omega'\omega$ is a path in $\overline{Q}$ from $i$ to $m$, and by Proposition~\ref{Proposition - W(omega) = W(omega')}, $W(i, m) = W(\omega'\omega) = W(\omega') + W(\omega) = W(j, m) + W(i, j)$ which proves our claim.

\begin{myprop}
	Let $X, Y\in \mathcal{D}^b(T(S)\mathrm{-mod})$ be indecomposable modules. Any path from $X$ to $Y$ in $\mathcal{D}^b(T(S)\mathrm{-mod})$ has length $W(X, Y)$.
\end{myprop}

\begin{proof}
	As above, $X = \nu_1^{-k_1}P_i$ and $Y = \nu_1^{-k_2}P_j$ for some $i, j\in Q_0$ and $k_1, k_2\in \Z_{\ge 0}$. By construction of $W$ the claim holds when $k_1 = k_2 = 0$. Let $Z\in \mathcal{D}^b(T(S)\mathrm{-mod})$ and let $k_3$ be such that $Z = \nu_1^{-k_3}P_l$ for some $l\in Q_0$. Assume that $\alpha: l\to j\in \overline{Q}_1$ and that there is a path of length $1$ from $Z$ to $Y$. We need to show that $W(X, Y) = W(X, Z) + 1$. Since there is a path of length $1$ from $Z$ to $Y$ we get that $|k_2 - k_3|\le 1$ and thus we have two cases. Either $k_3 = k_2$ or $k_3 = k_2 - 1$. If $k_3 = k_2$ then $\alpha\not\in Q_1$ and therefore $W(X, Y) = W(X, Z) + W(Z, Y) = W(X, Z) + W(l, j) = W(X, Z) + 1$. If $k_3 = k_2 - 1$ then $\alpha\in Q_1$ and thus $W(X, Y) = W(X, Z) + W(Z, Y) = W(X, Z) + 2(k_2 - k_3) + W(l, j) = W(X, Y) + 2 - 1 = W(X, Y) + 1$. Now the claim follows by induction on the length of the path from $X$ to $Y$.
\end{proof}

\begin{myprop}\label{Proposition - W(X, X[1]) = h}
	Let $X\in \mathcal{D}^b(T(S)\mathrm{-mod})$. Then $W(X, X[1]) = h$.
\end{myprop}

\begin{proof}
	Without loss of generality we can assume that $X=P_i$ since
	\begin{equation*}
		W(X, X[1]) = W(\nu_1^kX, \nu_1^kX[1])
	\end{equation*}
	for all $k\in \Z$. By Lemma~\ref{Lemma - l_i + l_sigma(i) = h} we have that $\nu_1^{-h}P_i = P_i[2]$ for all $i\in Q_0$. Now $W(P_i, P_i[2]) = W(P_i, \nu_1^{-h}P_i) = 2h$ and $W(P_i, P_i[2]) = 2W(P_i, P_i[1])$ give the claim.
\end{proof}

In the derived category we can look that the following Auslander-Reiten triangle
\begin{equation}\label{eq - triangle from I_i to P_i[1]}
	I_i\rightarrow N\rightarrow P_i[1] \rightarrow I_i[1]
\end{equation}
where $\nu_1 P_i[1] = I_{i}$. Then $W(I_i, P_i[1]) = 2$ and $W(P_i, I_i) = h-2$ by Proposition~\ref{Proposition - W(X, X[1]) = h}. This holds for every $i\in Q_0$. The following proposition uses this to compute the degree of the socle of the preprojective algebra.

\begin{myprop}
	Let $S$ be a representation finite species. Then
	\begin{equation*}
		\deg \mathrm{Soc}(\Pi(S)) = h-2.
	\end{equation*}
\end{myprop}

\begin{proof}
	If we shift (\ref{eq - triangle from I_i to P_i[1]}) to the right we get
	\begin{equation}\label{eq - AR triangle starting at projective}
		P_i\rightarrow I_i\rightarrow N\rightarrow P_i[1].
	\end{equation}
	Up to multiplication by a scalar the first morphism in (\ref{eq - AR triangle starting at projective}) is given by
	\begin{equation}\label{eq - map from P_i to I_i}
		P_i\twoheadrightarrow \mathrm{Top}(P_i)=\mathrm{Soc}(I_i)\rightarrowtail I_i,
	\end{equation}
	and it is an element of $\Pi(S)e_i$. Since (\ref{eq - AR triangle starting at projective}) is obtained by shifting an Auslander-Reiten triangle starting at $I_i$, any radical map from $I_i$ can be factored through $I_i\rightarrow N$. The exactness of the triangle yields that any radical morphism composed with (\ref{eq - map from P_i to I_i}) is the zero map. Hence (\ref{eq - map from P_i to I_i}) lies in the socle of $\Pi(S)e_i$. It is left to show that the $\mathrm{Soc}(\Pi(S)e_i)$ is $D$-generated by (\ref{eq - map from P_i to I_i}). Since $\Pi(S)$ is self-injective
	\begin{equation*}
		\dim_{D_{\sigma(i)}} \mathrm{Soc}(\Pi(S)e_i) = 1,
	\end{equation*}
	for all $i\in Q_0$. This proves the proposition.
\end{proof}

\begin{myprop}\label{Proposition - Koszul resolution in Dynkin case}
	Let $S$ be a species of Dynkin type $\Delta$. For every $i\in Q_0$, there is an exact sequence of left $\Pi(S)$-modules
	\begin{equation}\label{eq - Koszul resolution in Dynkin case sequence}
		0\rightarrow D_{\sigma(i)}\rightarrow \Pi(S)e_i \rightarrow \bigoplus_{\substack{\alpha\in \overline{Q}_1 \\ s(\alpha) = i}}(\Pi(S)e_{t(\alpha)})^{\dim_{D_{t(\alpha)}}\overline{M}_\alpha} \rightarrow \Pi(S)e_i\rightarrow D_i\rightarrow 0,
	\end{equation}
	where $D_i$ is the simple module at $i\in Q_0$.
\end{myprop}

\begin{proof}
	To construct the sequence (\ref{eq - Koszul resolution in Dynkin case sequence}) we start by considering Auslander-Reiten triangles in $\mathcal{D}^b = \mathcal{D}^b(\mathrm{mod-}T(S))$ (i.e. we for the moment switch to considering right $T(S)$-modules).
	
	By \cite[Theorem]{Dlab_1980}, the preprojective algebra does not depend on the orientation of $Q$. Therefore, we choose an orientation of $Q$ such that $P_i$ is simple. Consider the Auslander-Reiten triangle
	\begin{equation}\label{eq - AR triangle starting at Projective}
		P_i\rightarrow \bigoplus_{\substack{\alpha\in Q_1 \\ s(\alpha) = i}}P_{t(\alpha)}^{d_{P_{t(\alpha)}, P_i}/\delta(P_{t(\alpha)})}\rightarrow \nu_1^{-1}P_i \rightarrow P_i[1].
	\end{equation}
	The fact that the middle term of (\ref{eq - AR triangle starting at Projective}) is a direct sum of projective modules is a consequence of $P_i$ being simple. Note that this triangle consists of two length $1$ morphisms and one length $h-2$ morphism $\nu_1^{-1}P_i\rightarrow P_i[1]$, which follows from the discussion for (\ref{eq - triangle from I_i to P_i[1]}). Let $0\le j\le l_k$. We apply $\mathrm{Hom}_{\mathcal{D}^b}(-, \nu_1^{-j}P_k)$ on (\ref{eq - AR triangle starting at Projective}) to get the long exact sequence
	\begin{equation}\label{eq - hom applied to AR triangle oldversion}
		\begin{aligned}
			\cdots &\rightarrow \bigoplus_{\substack{\alpha\in Q_1 \\ s(\alpha) = i}}\mathrm{Hom}_{\mathcal{D}^b}(P_{t(\alpha)}^{d_{P_{t(\alpha)}, P_i}/\delta(P_{t(\alpha)})}[1], \nu_1^{-j}P_k) \rightarrow \mathrm{Hom}_{\mathcal{D}^b}(P_i[1], \nu_1^{-j}P_k) \rightarrow \\
			&\rightarrow \mathrm{Hom}_{\mathcal{D}^b}(\nu_1^{-1}P_i, \nu_1^{-j}P_k) \rightarrow \bigoplus_{\substack{\alpha\in Q_1 \\ s(\alpha) = i}} \mathrm{Hom}_{\mathcal{D}^b}(P_{t(\alpha)}^{d_{P_{t(\alpha)}, P_i}/\delta(P_{t(\alpha)})}, \nu_1^{-j}P_k) \rightarrow \\
			&\rightarrow \mathrm{Hom}_{\mathcal{D}^b}(P_i, \nu_1^{-j}P_k) \rightarrow \mathrm{Hom}_{\mathcal{D}^b}(\nu_1^{-1}P_i[-1], \nu_1^{-j}P_k) \rightarrow \bigoplus_{\substack{\alpha\in Q_1 \\ s(\alpha) = i}} \mathrm{Hom}_{\mathcal{D}^b}(P_{t(\alpha)}^{d_{P_{t(\alpha)}, P_i}/\delta(P_{t(\alpha)})}[-1], \nu_1^{-j}P_k)\rightarrow \cdots
		\end{aligned}.
	\end{equation}
	
	First note that
	\begin{equation}\label{eq - hom is simple in AR triangle 1}
		\mathrm{Hom}_{\mathcal{D}^b}(P_i[1], \nu_1^{-j}P_k) = \begin{cases}
			D_{\sigma(i)}, & \mbox{if } j=l_k \mbox{ and }k=\sigma(i) \\
			\mathrm{rad}_{\mathcal{D}^b}(P_i[1], P_{\sigma(k)}[1]), & \mbox{if } j = l_k \mbox{ and } k\not=\sigma(i) \\
			0, & \mbox{otherwise}
		\end{cases},
	\end{equation}
	and since (\ref{eq - AR triangle starting at Projective}) is an Auslander-Reiten triangle starting $P_i$ we have that every morphism in $\mathrm{rad}_{\mathcal{D}^b}(P_i, P_{\sigma(k)})$ factors through the first morphism in (\ref{eq - AR triangle starting at Projective}). Hence
	\begin{equation}\label{eq - hom is simple in AR triangle new}
		\coker\left(\bigoplus_{\substack{\alpha\in Q_1 \\ s(\alpha) = i}}\mathrm{Hom}_{\mathcal{D}^b}(P_{t(\alpha)}^{d_{P_{t(\alpha)}, P_i}/\delta(P_{t(\alpha)})}[1], \nu_1^{-j}P_k) \rightarrow \mathrm{Hom}_{\mathcal{D}^b}(P_i[1], \nu_1^{-j}P_k)\right) = \begin{cases}
			D_{\sigma(i)}, & \mbox{if } j=l_k \mbox{ and }k=\sigma(i) \\
			0, & \mbox{otherwise}.
		\end{cases}
	\end{equation}
	
	We can therefore consider the non-brutal truncation of (\ref{eq - hom applied to AR triangle oldversion}), i.e.
	\begin{equation}\label{eq - hom applied to AR triangle}
		\begin{aligned}
			0 &\rightarrow \coker\left(\bigoplus_{\substack{\alpha\in Q_1 \\ s(\alpha) = i}}\mathrm{Hom}_{\mathcal{D}^b}(P_{t(\alpha)}^{d_{P_{t(\alpha)}, P_i}/\delta(P_{t(\alpha)})}[1], \nu_1^{-j}P_k) \rightarrow \mathrm{Hom}_{\mathcal{D}^b}(P_i[1], \nu_1^{-j}P_k)\right) \rightarrow \\
			&\rightarrow \mathrm{Hom}_{\mathcal{D}^b}(\nu_1^{-1}P_i, \nu_1^{-j}P_k) \rightarrow \bigoplus_{\substack{\alpha\in Q_1 \\ s(\alpha) = i}} \mathrm{Hom}_{\mathcal{D}^b}(P_{t(\alpha)}^{d_{P_{t(\alpha)}, P_i}/\delta(P_{t(\alpha)})}, \nu_1^{-j}P_k) \rightarrow \\
			&\rightarrow \mathrm{Hom}_{\mathcal{D}^b}(P_i, \nu_1^{-j}P_k) \rightarrow \mathrm{Hom}_{\mathcal{D}^b}(\nu_1^{-1}P_i[-1], \nu_1^{-j}P_k) \rightarrow \bigoplus_{\substack{\alpha\in Q_1 \\ s(\alpha) = i}} \mathrm{Hom}_{\mathcal{D}^b}(P_{t(\alpha)}^{d_{P_{t(\alpha)}, P_i}/\delta(P_{t(\alpha)})}[-1], \nu_1^{-j}P_k)\rightarrow \cdots.
		\end{aligned}
	\end{equation}
	
	Also note that
	\begin{equation}\label{eq - hom is simple in AR triangle 2}
		\mathrm{Hom}_{\mathcal{D}^b}(\nu_1^{-1}P_i[-1], \nu_1^{-j}P_k) = \begin{cases}
			D_{i}, & \mbox{if } j=0 \mbox{ and }k=i \\
			0, & \mbox{otherwise}
		\end{cases},
	\end{equation}
	which follows form the fact that we chose our orientation of $Q$ such that $P_i$ is simple. Since $\mathrm{Hom}_{\mathcal{D}^b}(T(S)[-1], T(S)) = 0$, the direct sum
	\begin{equation}\label{eq - hom is simple in AR triangle 3}
		\bigoplus_{\substack{\alpha\in Q_1 \\ s(\alpha) = i}} \mathrm{Hom}_{\mathcal{D}^b}(P_{t(\alpha)}^{d_{P_{t(\alpha)}, P_i}/\delta(P_{t(\alpha)})}[-1], \nu_1^{-j}P_k) = 0.
	\end{equation}
	
	Let us now take the direct sum of (\ref{eq - hom applied to AR triangle}) for values $k\in Q_0$ and $0\le j\le l_k$ to get the following sequence
	\begin{equation}\label{eq - direct sum of AR sequences}
		\begin{aligned}
			0 &\rightarrow \bigoplus_{k\in Q_0}\bigoplus_{j=0}^{l_k} \coker\left(\bigoplus_{\substack{\alpha\in Q_1 \\ s(\alpha) = i}}\mathrm{Hom}_{\mathcal{D}^b}(P_{t(\alpha)}^{d_{P_{t(\alpha)}, P_i}/\delta(P_{t(\alpha)})}[1], \nu_1^{-j}P_k) \rightarrow \mathrm{Hom}_{\mathcal{D}^b}(P_i[1], \nu_1^{-j}P_k)\right) \rightarrow \\ &\rightarrow \bigoplus_{k\in Q_0}\bigoplus_{j=0}^{l_k} \mathrm{Hom}_{\mathcal{D}^b}(\nu_1^{-1}P_i, \nu_1^{-j}P_k)\rightarrow \bigoplus_{\substack{\alpha\in Q_1 \\ s(\alpha) = i}} \bigoplus_{k\in Q_0}\bigoplus_{j=0}^{l_k} \mathrm{Hom}_{\mathcal{D}^b}(P_{t(\alpha)}^{d_{P_{t(\alpha)}, P_i}/\delta(P_{t(\alpha)})}, \nu_1^{-j}P_k) \rightarrow \\ 
			&\rightarrow \bigoplus_{k\in Q_0}\bigoplus_{j=0}^{l_k} \mathrm{Hom}_{\mathcal{D}^b}(P_i, \nu_1^{-j}P_k) \rightarrow \bigoplus_{k\in Q_0}\bigoplus_{j=0}^{l_k} \mathrm{Hom}_{\mathcal{D}^b}(\nu_1^{-1}P_i[-1], \nu_1^{-j}P_k) \rightarrow 0\rightarrow \cdots.
		\end{aligned}
	\end{equation}
	
	From the definition of the preprojective algebra we know
	\begin{equation*}
		\bigoplus_{k\in Q_0}\bigoplus_{j=0}^{l_k-1}\mathrm{Hom}_{\mathcal{D}^b}(P_i, \nu_1^{-j}P_k) = \Pi(S)e_i,
	\end{equation*}
	as a vector space. Naturally, we can view it as an $\Pi(S)$-module, and since all of the maps in (\ref{eq - direct sum of AR sequences}) are given by composition by certain morphisms, we can view the sequence as being a sequence of $\Pi(S)$-modules. Using the fact that $\mathrm{Hom}_{\mathcal{D}^b}(T(S), T(S)[1]) = 0$ it follows that
	\begin{equation*}
		\mathrm{Hom}_{\mathcal{D}^b}(P_i, \nu_1^{-l_k}P_k) = 0.
	\end{equation*}
	Using (\ref{eq - hom is simple in AR triangle new}), (\ref{eq - hom is simple in AR triangle 2}), (\ref{eq - hom is simple in AR triangle 3}) and the fact that $d_{P_{t(\alpha)}, P_i}/\delta(P_{t(\alpha)}) = \dim_{D_{t(\alpha)}}\overline{M}_\alpha$ we can rewrite (\ref{eq - direct sum of AR sequences}) to
	\begin{equation*}
		0 \rightarrow D_{\sigma(i)} \rightarrow \Pi(S)e_i\rightarrow \bigoplus_{\substack{\alpha\in \overline{Q}_1 \\ s(\alpha) = i}} (\Pi(S)e_{t(\alpha)})^{\dim_{D_{t(\alpha)}}\overline{M}_\alpha} \rightarrow \Pi(S)e_i \rightarrow D_i \rightarrow 0.
	\end{equation*}
\end{proof}

\begin{myrem}
	Note that taking the $\star$-degree $0$ part of the complex in Proposition~\ref{Proposition - Koszul resolution in Dynkin case} yields the Koszul complex (\ref{eq - Koszul complex for T(S)_0}).
\end{myrem}

\begin{myrem}
	In the case where the species $S$ is not of finite representation type, then $\Pi(S)$ is Koszul and the Koszul sequence is similar to \ref{eq - Koszul resolution in Dynkin case sequence} but with the last term removed. This is formulated in \cite[Proposition 8.8]{athaiokrmy}.
\end{myrem}

\begin{mydef}\cite[Definition 3.1]{brenner2002periodic}
	Let $\Lambda=\bigoplus_{i=0}^p \Lambda_i$ be a graded $\K$-algebra, where $\Lambda_0$ is semi-simple. We say that $\Lambda$ is an almost Koszul, or $(p,q)$-almost Koszul algebra, if there exists a graded exact complex
	\begin{equation*}
		0\rightarrow W \rightarrow P_q \rightarrow \dots \rightarrow P_1\rightarrow P_0\rightarrow \Lambda_0 \rightarrow 0,
	\end{equation*}
	of $\Lambda$-modules where $P_i$ is projective and is generated by its component of degree $i$ for all $i=0, \dots, q$. Here $W=\Lambda_p\otimes_{\Lambda_0} (P_q)_q$.
\end{mydef}

\begin{mycor}\label{Corollary - Pi(S) is an almost Koszul algebra}
	If $S$ is a species of Dynkin type $\Delta$, then $\Pi(S)$ is $(h-2, 2)$-almost Koszul.
\end{mycor}

\begin{proof}
	From the definition of the preprojective algebra it is easy to see that it is indeed a quadratic algebra. The sequence in Proposition~\ref{Proposition - Koszul resolution in Dynkin case} is the almost Koszul complex for $S_i$. All the morphism are of degree $1$ except for the fourth morphism which is of degree $h-2$. This tells us that the third syzygy is $S_{\sigma(i)}$, and therefore $\Pi(S)$ is $(h-2, 2)$-almost Koszul.
\end{proof}

\section{Higher Preprojective Algebras}\label{Section - Higher preprojective algebra}
In this section we recall the higher analogue of the preprojective algebra and the $l$-homogeneous property for $d$-representation finite algebras.

In the setup of higher Auslander-Reiten theory \cite{iyama2007higher, iyama2008auslander}, The notion of being representation finite and hereditary is generalized by the following definition.

\begin{mydef}
	Let $\Lambda$ be a finite dimensional $\K$-algebra of global dimension $d$. We say that $\Lambda$ is $d$-representation finite if there exists a $d$-cluster tilting module $M\in \Lambda\mathrm{-mod}$, i.e.
	\begin{equation*}
		\begin{aligned}
			\mathrm{add}(M) &= \{X\in \Lambda\mathrm{-mod} \mid \mathrm{Ext}^i_\Lambda(M, X) = 0 \mbox{ for any } 0 < i < n \} = \\
			&= \{X\in \Lambda\mathrm{-mod} \mid \mathrm{Ext}^i_\Lambda(X, M) = 0 \mbox{ for any } 0 < i < n \}.
		\end{aligned}
	\end{equation*}
\end{mydef}

Before we introduce the higher analogue of being $l$-homogeneous, we need to introduce the higher Auslander-Reiten translation together with some of its properties, which are stated in the following proposition.

\begin{myprop}\cite[Proposition 1.3b)]{iyama4897cluster}\label{Proposition - higher AR translation, integer l}
	Let $\Lambda$ be a $d$-representation finite algebra. Let
	\begin{equation*}
		\begin{aligned}
			\tau_d &= \mathrm{Tor}^\Lambda_d(D\Lambda, -)\cong D\mathrm{Ext}^d_\Lambda(-, \Lambda): \Lambda\mathrm{-mod}\rightarrow \Lambda\mathrm{-mod}, \\
			\tau_d^{-1} &= D\mathrm{Tor}^\Lambda_d(D-, D\Lambda)\cong \mathrm{Ext}^d_\Lambda(D\Lambda, -): \Lambda\mathrm{-mod}\rightarrow \Lambda\mathrm{-mod}.
		\end{aligned}
	\end{equation*}
	Let $P_1, \dots, P_a$ be the isomorphism classes of indecomposable projective $\Lambda$-modules, and let $I_i$ be the indecomposable injective corresponding to $P_i$.
	\begin{enumerate}
		\item There exists a permutation $\sigma$ and positive integers $l_1, \dots, l_a$ such that $\tau_d^{l_i-1}I_i = P_{\sigma(i)}$ for all $1\le i\le a$.
		\item There exists a unique basic $d$-cluster tilting $\Lambda$-module $M$, which is given as the direct sum of the following mutually non-isomorphic indecomposable $\Lambda$-modules.
		
		\begin{center}
			\begin{tabular}{ccccc}
				$I_1$, & $\tau_d I_1$, & $\cdots$ & $\tau_d^{l_1-2}I_1$, & $\tau_d^{l_1-1}I_1 = P_{\sigma(1)}$ \\
				$I_2$, & $\tau_d I_2$, & $\cdots$ & $\tau_d^{l_2-2}I_2$, & $\tau_d^{l_2-1}I_2 = P_{\sigma(2)}$ \\
				$\cdots$ & $\cdots$ & $\cdots$ & $\cdots$ & $\cdots$ \\
				$I_a$, & $\tau_d I_a$, & $\cdots$ & $\tau_d^{l_a-2}I_a$, & $\tau_d^{l_a-1}I_a = P_{\sigma(a)}$ \\
			\end{tabular}
		\end{center}
		\item We have mutually quasi-inverse equivalences $\tau_d: \mathrm{add}(M/\Lambda) \xrightarrow{\sim} \mathrm{add}(M/D\Lambda)$ and $\tau^{-1}_d: \mathrm{add}(M/D\Lambda)\xrightarrow{\sim} \mathrm{add}(M/\Lambda)$.
	\end{enumerate}
\end{myprop}

\begin{mydef}\cite[Definition 1.2]{herschend2011n}
	Let $\Lambda$ be a $d$-representation finite algebra. We say that $\Lambda$ is $l$-homogeneous if $l_i = l$ for all $i$.
\end{mydef}

In the paper \cite{iyama2013stable} Iyama and Oppermann introduced the $(d+1)$-preprojective algebra of an algebra of global dimension $d$, which is the higher analogue of Definition~\ref{Definition - preprojective algebra BGL perspective}.

\begin{mydef}\cite[Definition 2.11]{iyama2013stable}
	Let $\Lambda$ be a $d$-representation finite algebra. We denote the $(d+1)$-preprojective algebra, or the higher preprojective algebra, of $\Lambda$ by $\Pi(\Lambda)$. It is defined as
	\begin{equation*}
		\Pi(\Lambda) = \bigoplus_{i\ge 0}\Pi_i = \bigoplus_{i\ge 0}\mathrm{Hom}_\Lambda(\Lambda, \tau_d^{-i}\Lambda),
	\end{equation*}
	with multiplication
	\begin{equation*}
		\begin{aligned}
			\Pi_i\times \Pi_j&\to \Pi_{i+j}, \\
			(u, v)&\mapsto uv = (\tau^{-i}_d(v)\circ u: \Lambda\to \tau_d^{-(i+j)}\Lambda).
		\end{aligned}
	\end{equation*}
\end{mydef}

As for the preprojective algebra in Definition~\ref{Definition - preprojective algebra BGL perspective}, the $(d+1)$-preprojective algebra is naturally $\Z$-graded by setting elements in $\Pi_i$ to have degree $i$. This grading will be referred to as the $\star$-grading for $\Pi(\Lambda)$.

\begin{mydef}
	Let $\Lambda$ be a $\K$-algebra. We define $\nu_d = \nu\circ [-d]: \mathcal{D}^b(\Lambda\mathrm{-mod})\to \mathcal{D}^b(\Lambda\mathrm{-mod})$, where $\nu$ the is the Nakayama functor.
\end{mydef}

\begin{myrem}
	For a $d$-representation finite algebra $\Lambda$, we can also define the $(d+1)$-preprojective algebra using $\nu_d^{-1}$ instead of $\tau_d^{-1}$, i.e.
	\begin{equation}\label{Remark - definition of derived preprojective algebra}
		\Pi(\Lambda) = \bigoplus_{i\ge 0}\mathrm{Hom}_{\mathcal{D}^b(\Lambda\mathrm{-mod})}(\Lambda, \nu_d^{-i}\Lambda).
	\end{equation}
\end{myrem}

We have the following important result by Grant and Iyama.

\begin{mythm}\label{Theorem - Grant and Iyama almost koszul}\cite[Theorem 4.21(b)]{grant2020higher}
	Let $\Lambda$ be a Koszul $d$-hereditary algebra, and $\Pi$ its $(d+1)$-preprojective algebra with the $(d+1)$-total grading given in \cite[Definition 4.15]{grant2020higher}. If $\Lambda$ is $d$-representation finite, then $\Pi$ is almost Koszul.
\end{mythm}

\section{Tensor Product of algebras}\label{Section - $n$-Representation Finite Species}
In this section we study tensor products of $d$-representation finite algebras. The motivation comes from the paper \cite{herschend2011n} that showed that if $\Lambda_i$ is $l$-homogeneous and $d_i$-representation finite for $i\in \{1, 2\}$, then $\Lambda_1\otimes_\K \Lambda_2$ is again $l$-homogeneous and $(d_1 + d_2)$-representation finite when $\K$ is a perfect field. For this reason we assume that $\K$ is perfect from this point. The goal of this section is to prove that the higher preprojective algebra $\Pi(\Lambda_1\otimes_\K\Lambda_2)$ is an almost Koszul algebra given certain conditions. This is formulated in Theorem~\ref{Theorem - product of species is almost koszul}. This shows existence of almost Koszul complexes and in Section~\ref{Section - Koszul Complex for Higher Species} we give a more complete description of the almost Koszul complexes given certain assumptions on the algebras $\Lambda_1$ and $\Lambda_2$.

For completeness we will first prove the folklore statement that says that the tensor product of two Koszul algebras is a Koszul algebra. This is proven for connected $\K$-algebras in \cite[Chapter 3.1, Corollary 1.2]{polishchuk2005quadratic}.

\begin{mylemma}\label{Lemma - A,B koszul implies AxB koszul}
	Let $\Lambda_1$ and $\Lambda_2$ be two Koszul algebras. Then $\Lambda_1\otimes_\K \Lambda_2$ is Koszul.
\end{mylemma}

\begin{proof}
	Let $P^{\Lambda_1}_\bullet$ and $P^{\Lambda_2}_\bullet$ be the Koszul complex for $(\Lambda_1)_0$ and $(\Lambda_2)_0$ respectively. Consider the following double complex
	\begin{equation*}
		\begin{tikzcd}
			& \vdots \arrow[d] & \vdots \arrow[d] \\
			\cdots \arrow[r] & P_1^{\Lambda_1}\otimes_\K P_1^{\Lambda_2} \arrow[r, "1\otimes p^{\Lambda_2}_1"] \arrow[d, "p^{\Lambda_1}_1\otimes 1"] & P_1^{\Lambda_1}\otimes_\K P^{\Lambda_2}_0 \arrow[r] \arrow[d, "p^{\Lambda_1}_1\otimes 1"] & 0 \\
			\cdots \arrow[r] & P_0^{\Lambda_1}\otimes_\K P_1^{\Lambda_2} \arrow[r, "1\otimes p^{\Lambda_2}_1"] \arrow[d] & P_0^{\Lambda_1}\otimes_\K P_0^{\Lambda_2} \arrow[r] \arrow[d] & 0 \\
			& 0 & 0
		\end{tikzcd}
	\end{equation*}
	Denote this complex by $P^{\Lambda_1}_\bullet\otimes_\K P^{\Lambda_2}_\bullet$. Note that all the squares in this complex commute, and therefore we can take the total complex. At degree $i$ we have
	\begin{equation*}
		\mathrm{Tot}(P^{\Lambda_1}_\bullet\otimes_\K P^{\Lambda_2}_\bullet)_i = \bigoplus_{m + n = i}P_m^{\Lambda_1}\otimes_\K P_n^{\Lambda_2}
	\end{equation*}
	and the morphism, denoted by $d_i$, from $\mathrm{Tot}(P^{\Lambda_1}_\bullet\otimes_\K P^{\Lambda_2}_\bullet)_i$ is
	\begin{equation*}
		d_i = \begin{bmatrix}
			p^{\Lambda_1}_i\otimes 1 & -1\otimes p^{\Lambda_2}_{1} \\
			& p^{\Lambda_1}_{i-1}\otimes 1 & 1\otimes p^{\Lambda_2}_{2} \\
			& & \ddots & \ddots \\
			& & & p^{\Lambda_1}_2\otimes 1 & (-1)^{i}1\otimes p^{\Lambda_2}_i
		\end{bmatrix}
	\end{equation*}
	Since all squares in $P^{\Lambda_1}_\bullet\otimes_\K P^{\Lambda_2}_\bullet$ commute, $\mathrm{Tot}(P^{\Lambda_1}_\bullet\otimes_\K P^{\Lambda_2}_\bullet)$ is exact except at degree $0$. Moreover, $\mathrm{H}_0(\mathrm{Tot}(P^{\Lambda_1}_\bullet\otimes_\K P^{\Lambda_2}_\bullet)) = (\Lambda_1)_0\otimes_\K (\Lambda_2)_0 = (\Lambda_1 \otimes_\K \Lambda_2)_0$. The morphisms in $\mathrm{Tot}(P^{\Lambda_1}_\bullet\otimes_\K P^{\Lambda_2}_\bullet)$ have degree $1$, and thus $\mathrm{Tot}(P^{\Lambda_1}_\bullet\otimes_\K P^{\Lambda_2}_\bullet)$ is the Koszul complex of $(\Lambda_1\otimes_\K \Lambda_2)_0$.
\end{proof}

Herschend and Iyama gave a description of the Calabi-Yau properties of tensor products of fractionally Calabi-Yau algebras which is applicable to our case.

\begin{myprop}\cite[Proposition 1.4]{herschend2011n}
	Let $\mathbb{K}$ be a perfect field. If $\Lambda_i$ is $\frac{m_i}{l_i}$-CY (respectively, twisted $\frac{m_i}{l_i}$-CY) for each $i\in \{1, 2\}$, then $\Lambda_1\otimes_\K \Lambda_2$ is $\frac{m}{l}$-CY (respectively, twisted $\frac{m}{l}$-CY) for the least common multiple $l$ of $l_1, l_2$ and
	\begin{equation*}
	m=l\left(\frac{m_1}{l_1} + \frac{m_2}{l_2}\right).
	\end{equation*}
\end{myprop}

\begin{mycor}\cite[Corollary 1.5]{herschend2011n}\label{Proposition - A x B is n_1 + n_2 RF}
	Let $\mathbb{K}$ be a perfect field and $l$ a positive integer. If $\Lambda_i$ is $l$-homogeneous $d_i$-representation finite for each $i\in \{1, 2\}$, then $\Lambda_1\otimes_\K \Lambda_2$ is an $l$-homogeneous $(d_1 + d_2)$-representation-finite algebra with a $(d_1 + d_2)$-cluster tilting module $\bigoplus_{i=0}^{l-1}(\tau_{d_1}^{-i}\Lambda_1\otimes_\K \tau_{d_2}^{-i}\Lambda_2)$.
\end{mycor}

Combining Lemma~\ref{Lemma - A,B koszul implies AxB koszul}, Corollary~\ref{Proposition - A x B is n_1 + n_2 RF} and Theorem~\ref{Theorem - Grant and Iyama almost koszul} yields the following result.

\begin{mythm}\label{Theorem - product of species is almost koszul}
	Let $\Lambda_i$ be an $l$-homogeneous and $d_i$-representation finite Koszul algebra. Then the $(d_1 + d_2 + 1)$-preprojective algebra $\Pi(\Lambda_1 \otimes_\K \Lambda_2)$ is an almost Koszul algebra.
\end{mythm}

\section{Properties of the Segre Product}\label{Section - Diagonal Tensor Product}
In section~\ref{Section - Koszul Complex for Higher Species} we describe the almost Koszul complexes in $\Pi(\Lambda_1\otimes_\K \Lambda_2)$ using the almost Koszul complexes in $\Pi(\Lambda_1)$ and $\Pi(\Lambda_2)$, where $\Lambda_1$ and $\Lambda_2$ are $d_{\Lambda_1}$- and $d_{\Lambda_2}$-representation finite Koszul algebras. This is done by using the Segre product, so we devote this section to define the Segre product for graded algebras.

\begin{mydef}\cite[Chapter 3.2, Definition 1]{polishchuk2005quadratic}
	Given two graded $\K$-algebras $\Lambda_1$ and $\Lambda_2$. 
	
	\begin{enumerate}
		\item As a vector space, we can decompose $\Lambda_i = \bigoplus_{k\in \Z}\Lambda_{i, k}$, where $\Lambda_{i, k}$ is the space of elements of degree $k$ in $\Lambda_i$, for $i\in \{1, 2\}$. We define the Segre product of $\Lambda_1$ and $\Lambda_2$ as
		\begin{equation*}
			\Lambda_1\totimes_\K \Lambda_2 = \bigoplus_{k\in \Z}\Lambda_{1, k}\otimes_\K \Lambda_{2, k}.
		\end{equation*}
		\item Let $X\in \Lambda_1\mathrm{-mod}^\Z$ and $Y\in \Lambda_2\mathrm{-mod}^\Z$. We define the Segre product of $X$ and $Y$ as
		\begin{equation*}
			X\totimes_\K Y = \bigoplus_{k\in \Z} X_k\otimes_\K Y_k\in \Lambda_1\totimes_\K \Lambda_2\mathrm{-mod}^\Z,
		\end{equation*}
		where the subscript $k$ denotes the graded part at $k$.
		\item Let $f_i: X_i\to Y_i$ where $X_i, Y_i\in \Lambda_i\mathrm{-mod}^\Z$ for each $i\in \{1,2\}$. We define the Segre product of $f_1$ and $f_2$ as
		\begin{equation*}
			f_1\totimes f_2 = \bigoplus_{k\in \Z} f_{1,k}\otimes f_{2,k}: \bigoplus_{k\in \Z} X_{1,k}\otimes_\K X_{2,k}\to \bigoplus_{k\in \Z}Y_{1,k}\otimes_\K Y_{2,k} 
		\end{equation*}
		where the subscript $k$ denotes the graded part at $k$.
	\end{enumerate}
\end{mydef}

The first application of the Segre product will be used to describe the preprojective algebra of $\Lambda_1\otimes_\K \Lambda_2$.

\begin{mylemma}\label{Lemma - Pi(AxB) = Pi(A)xPi(B)}
	Let $\Lambda_i$ be an $l$-homogeneous $d_{\Lambda_i}$-representation finite algebra for each $i\in \{1, 2\}$. Then
	\begin{equation*}
		\Pi(\Lambda_1\otimes_\K \Lambda_2) \cong \Pi(\Lambda_1)\totimes_\K \Pi(\Lambda_2).
	\end{equation*}
	Moreover, this isomorphism is compatible with the $\star$-grading.
\end{mylemma}

\begin{proof}
	We will use the definition for the preprojective algebra given in Remark~\ref{Remark - definition of derived preprojective algebra}. From \cite[Lemma 2.11]{herschend2014n} we have
	\begin{equation*}
		\nu_{d_{\Lambda_1} + d_{\Lambda_2}}(X\otimes_\K Y) = \nu_{d_{\Lambda_1}}(X)\otimes \nu_{d_{\Lambda_2}}(Y),
	\end{equation*}
	where $X\in \mathcal{D}^b(\Lambda_1\mathrm{-mod})$ and $Y\in \mathcal{D}^b(\Lambda_2\mathrm{-mod})$. Therefore
	\begin{align*}
		\Pi(\Lambda_1\otimes_\K \Lambda_2) &= \bigoplus_{i\ge 0}\mathrm{Hom}_{\mathcal{D}^b(\Lambda_1\otimes_\K \Lambda_2\mathrm{-mod})}(\Lambda_1\otimes_\K \Lambda_2, \nu_{d_{\Lambda_1} + d_{\Lambda_2}}^{-i}(\Lambda_1\otimes_\K \Lambda_2))\cong \\
		&\cong \bigoplus_{i\ge 0}\mathrm{Hom}_{\mathcal{D}^b(\Lambda_1\otimes_\K \Lambda_2\mathrm{-mod})}(\Lambda_1\otimes_\K \Lambda_2, \nu_{d_{\Lambda_1}}^{-i}(\Lambda_1)\otimes_\K \nu_{d_{\Lambda_2}}^{-i}(\Lambda_2)) \cong \\
		&\cong \bigoplus_{i\ge 0} \mathrm{Hom}_{\mathcal{D}^b(\Lambda_1\mathrm{-mod})}(\Lambda_1, \nu_{d_{\Lambda_1}}^{-i}(\Lambda_1))\otimes_\K \mathrm{Hom}_{\mathcal{D}^b(\Lambda_2\mathrm{-mod})}(\Lambda_2, \nu_{d_{\Lambda_2}}^{-i}(\Lambda_2)) = \\
		&= \Pi(\Lambda_1)\totimes_\K \Pi(\Lambda_2). \qedhere
	\end{align*}
\end{proof}

We also prove the Künneth formula for the Segre product with the use of the following lemma.

\begin{mylemma}\label{lemma - diagonal tensor product is exact}
	The Segre product $-\totimes_\K-$ is bi-exact.
\end{mylemma}

\begin{proof}
	Let
	\begin{equation}\label{prooflemma - exact sequence of graded modules}
		0\rightarrow X\rightarrow Y\rightarrow Z\rightarrow 0
	\end{equation}
	be an exact sequence of graded $\Lambda_2$-modules. Let $M$ be a graded $\Lambda_1$-module. We have to show that
	\begin{equation*}
		0\rightarrow M\totimes_\K X\rightarrow M\totimes_\K Y\rightarrow M\totimes_\K Z\rightarrow 0
	\end{equation*}
	is exact. It is enough to show that the sequence is exact at all degrees, i.e.
	\begin{equation*}
		0\rightarrow (M\totimes_\K X)_i\rightarrow (M\totimes_\K Y)_i\rightarrow (M\totimes_\K Z)_i\rightarrow 0
	\end{equation*}
	but then note from the definition of the Segre product that this is equivalent to taking the tensor product over $\K$, since $(M\totimes_\K -)_i = M_i\otimes_\K (-)_i$, which is exact, thus $M\totimes_\K -$ is exact. A similar argument can show that $-\totimes_\K -$ is exact in the first argument.
\end{proof}

Now we want to extend this functor to complexes. As with the usual tensor product, we extend the definition to complexes by taking the total complex of a certain double complex.

\begin{mydef}
	Let $\Lambda_1$ and $\Lambda_2$ be two graded $\K$-algebras. Let $X\in \mathcal{C}(\Lambda_1)$ and $Y\in \mathcal{C}(\Lambda_2)$ be two complexes, then consider the double complex
	\begin{equation}\label{eq - total complex of totimes}
		\begin{tikzcd}[column sep =40pt, row sep=40pt]
		 & \vdots \arrow[d] & \vdots \arrow[d] \\
		 \cdots \arrow[r] & X_i\totimes_\K Y_i \arrow[r, "d_X\totimes 1_{Y_i}"] \arrow[d, "1_{X_i}\totimes d_Y"] & X_{i-1}\totimes_\K Y_i \arrow[r] \arrow[d, "1_{X_{i-1}}\totimes d_Y"] & \cdots \\
		 \cdots \arrow[r] & X_i\totimes_\K Y_{i-1} \arrow[r, "d_X\totimes 1_{Y_{i-1}}"] \arrow[d] & X_{i-1}\totimes_\K Y_{i-1} \arrow[r] \arrow[d] & \cdots \\
		 & \vdots & \vdots
		\end{tikzcd}
	\end{equation}
	We define
	\begin{equation*}
		\mathrm{Tot}(-\totimes_\K -): \mathcal{C}(\Lambda_1)\times \mathcal{C}(\Lambda_2)\rightarrow \mathcal{C}(\Lambda_1)\totimes_\K \mathcal(\Lambda_2)
	\end{equation*}
	such that $\mathrm{Tot}(X\totimes_\K Y)$ is the total complex of (\ref{eq - total complex of totimes}).
\end{mydef}

\begin{mythm}\label{Theorem - Kunneth relation for diagonal tensor product}\cite[Theorem VI.3.1]{cartan1999homological}
	Let $\Lambda_1$ and $\Lambda_2$ be two graded $\K$-algebras. Let $X\in \mathcal{C}(\Lambda_1)$ and $Y\in \mathcal{C}(\Lambda_2)$ be two complexes. Then
	\begin{equation*}
		H(\mathrm{Tot}(X\totimes_\K Y))\cong \mathrm{Tot}(H(X)\totimes_\K H(Y)).
	\end{equation*}
	Explicitly
	\begin{equation*}
		H_n(\mathrm{Tot}(X\totimes_\K Y))\cong \bigoplus_{i + j = n}H_i(X)\totimes_\K H_j(Y).
	\end{equation*}
\end{mythm}

\begin{proof}
	Setting $T=-\totimes_\K -$ in \cite[Theorem IV.8.1]{cartan1999homological} together with Lemma~\ref{lemma - diagonal tensor product is exact} yields the result.
\end{proof}

\section{Almost Koszul Complex for Higher Species}\label{Section - Koszul Complex for Higher Species}
If we have an $l$-homogeneous $d_i$-representation finite Koszul algebra $\Lambda_i$ for each $i\in \{1, 2\}$, then we have seen earlier that $\Lambda_1 \otimes_\K \Lambda_2$ is a $(d_1 + d_2)$-representation finite $l$-homogeneous Koszul algebra and by Theorem~\ref{Theorem - product of species is almost koszul}, the higher preprojective algebra $\Pi(\Lambda_1\otimes_\K \Lambda_2)$ is almost Koszul. In this section we describe the almost Koszul complexes in $\Pi(\Lambda_1\otimes_\K \Lambda_2)$ using the almost Koszul complexes in $\Pi(\Lambda_1)$ and $\Pi(\Lambda_2)$. This is formulated in Theorem~\ref{Theorem - Koszul complex for tensor product of species}. Before doing so we need to introduce some preliminary notions.

\begin{mydef}
	Let $\mathcal{A}$ be an abelian category. Let $f: Q_\bullet \rightarrow R_\bullet$ be a morphism between two complexes in $\mathcal{C}(\mathcal{A})$. Then we define the mapping cone of $f$ as
	\begin{equation*}
	C(f)_i = Q_{i-1}\oplus R_i 
	\end{equation*}
	with the differential
	\begin{equation*}
	d^{C(f)}_i = \begin{bmatrix}
	-d^Q_{i-1} & 0 \\
	f_{i-1} & d^R_i
	\end{bmatrix}.
	\end{equation*}
\end{mydef}

Let $\mathcal{C}^b(\mathcal{A}, m)$ be the full subcategory of $\mathcal{C}^b(\mathcal{A})$ consisting of complexes of length $m$, i.e. if $X_\bullet\in \mathcal{C}^b(\mathcal{A}, m)$ then
\begin{equation*}
	X_\bullet = 0\rightarrow X_m\rightarrow X_{m-1}\rightarrow \cdots \rightarrow X_0\rightarrow 0.
\end{equation*}

\begin{mydef}
	Let $\mathcal{A}$ be an abelian category and let $f: Q_\bullet \rightarrow R_\bullet$ be a morphism in $\mathcal{C}^b(\mathcal{A}, m)$. We say that $f$ is an almost quasi-isomorphism if $H_i(f)$ is an isomorphism when $0<i<m$, and $H_0(f)$ and $H_m(f)$ are a monomorphism and an epimorphism respectively.
\end{mydef}

\begin{mylemma}\label{Lemma - a.q-iso implies homology at two places}
	Let $\mathcal{A}$ be an abelian category, and let $f:Q_\bullet \rightarrow R_\bullet$ be a morphism in $\mathcal{C}^b(\mathcal{A}, m)$. Then $H_i(C(f)) = 0$ for $0< i\le m$ if and only if $f$ is an almost quasi-isomorphism. 
\end{mylemma}

\begin{proof}
	We can assume $\mathcal{A}$ is a module category of some ring due to the full imbedding theorem \cite[Theorem 4.4]{Mitchell1964}. First we assume that $f$ is an almost quasi-isomorphism. It is enough to show that for all $i$, $H_i(C(f)) = 0$ if $H_{i-1}(f)$ is a monomorphism and $H_i(f)$ is an epimorphism. Let $(x, y)\in \ker (d^{C(f)}_i)\subset Q_{i-1}\oplus R_i$. This means that
	\begin{equation*}
	\begin{aligned}
	d^Q_{i-1}(x) &= 0 \\
	f_{i-1}(x) + d_i^R(y) &= 0
	\end{aligned}
	\end{equation*}
	The second equation says that $f_{i-1}(x) \in \im(d_i^R)$. Using the fact that $H_{i-1}(f)$ is a monomorphism yields that $x\in \im (d_i^Q)$. Let $s\in Q_i$ be such that $-d^Q_i(s) = x$. Then $f_i(s) - y\in \ker(d_i^R)$ as is easily shown by the following computation
	\begin{equation*}
	d^R_i(f_i(s) - y) = f_{i-1}(d^Q_i(s)) - d^R_i(y) = -f_{i-1}(x) - d^R_i(y) = 0.
	\end{equation*}
	Since $H_i(f)$ is an epimorphism there exists a $z\in \ker (d_i^Q)$ such that
	\begin{equation*}
	f_i(z) = f_i(s)-y + d^R_{i+1}(t)
	\end{equation*}
	for some $t\in R_{i+1}$. By construction we have that $d^{C(f)}_{i+1}(s-z, t) = (x, y)$, and thus $H_i(C(f))=0$.
	
	Now for the other direction, assume that $H_i(C(f)) = 0$ if $0\le i<m$. It is enough to show that $H_{i-1}(f)$ and $H_i(f)$ are a monomorphism and an epimorphism, respectively. Assume for contradiction that $H_{i-1}(f)$ is not a monomorphism. This means that we can choose an $x\in \ker (d_{i-1}^Q)$ such that $x\not\in \im (d_i^Q)$ and $H_{i-1}(f)(x)=0$. Then there exists $y\in R_i$ such that $(x, y)\in \ker (d^{C(f)}_i)\subset Q_{i-1}\oplus R_i$. By construction, $(x, y)\not\in \im (d^{C(f)}_{i+1})$, contradicting the assumption $H_i(C(f))=0$. We have now shown that $H_{i-1}(f)$ is a monomorphism for all $0<i\le m$. We assume for contradiction that $H_i(f)$ is not an epimorphism. Then there is a $y\in \ker(d_i^R)\backslash\im (d_{i+1}^R)$ such that $y\not\in f_i(\ker(d_{i+1}^Q))$. By the construction of $y$ we see that $y\not\in \im(d_{i+1}^R) + \im(f_i)$, therefore $(0, y)\not\in \im(d_{i+1}^{C(f)})$ and $(0, y)\in \ker(d_i^{C(f)})$, contradiction the fact that $H_i(C(f))=0$. Thus $f$ is an almost quasi-isomorphism.
\end{proof}

In \cite{grant2020higher} $\Pi(\Lambda)$ is described as $\Pi(\Lambda) = T_\Lambda(E)$, where $E = \mathrm{Ext}^d_\Lambda(D\Lambda, \Lambda)$ which coincides with our definition as $\mathrm{Ext}^d_\Lambda(D\Lambda, \Lambda) \cong \tau_d^{-1}\Lambda \cong \mathrm{Hom}_{\mathcal{D}^b(\Lambda\mathrm{-mod})}(\Lambda, \nu_d^{-1}(\Lambda))$. In the case when $\Lambda$ is Koszul we can compute $E$ by a graded projective resolution and so $E$ is graded and generated in degree $-d$ \cite[Proposition 4.16]{grant2020higher}. This gives a natural $\Z^2$-grading on $\Pi(\Lambda) = T_\Lambda(E)$ where
\begin{equation*}
	\Pi(\Lambda)_{jk} = (E^{\otimes k})_j.
\end{equation*}
In order to get a grading where $\mathrm{rad}(\Pi(\Lambda))$ is the positive degree part and $\Pi(\Lambda) / \mathrm{rad}(\Pi(\Lambda))$ is the degree $0$ part. Grant and Iyama introduce the $(d+1)$-total grading by
\begin{equation*}
	\Pi(\Lambda) = \bigoplus_{(d+1)k + j = l}\Pi(\Lambda)_{jk}
\end{equation*}
so that $\Pi(\Lambda)_0 = \Lambda_0$ and $\Pi(\Lambda)_1 = \Lambda_1 \oplus E_{-d}$ and so on. From now on, when $\Lambda$ is Koszul we always consider $\Pi(\Lambda)$ as $\Z^2$-graded
\begin{equation*}
	\Pi(\Lambda) = \bigoplus_{l,k\ge 0}\Pi(\Lambda)_{lk},
\end{equation*}
where $l$ refers to the $(d+1)$-total grading and $k$ refers to the $\star$-grading.

Note that if $\Lambda = T(S)$ where $S$ is some representation finite species, then, using Definition \ref{Definition - preprojective algebra of a species}, the grading on $\Pi(\Lambda)$ is such that the elements in $\overline{M}_\alpha$ and $\overline{M}_{\alpha^*}$ are of degree $(1, 0)$ and $(1 ,1)$ respectively for all $\alpha\in Q_1$.

Now let $\Lambda = \Lambda_1\otimes_\K \Lambda_2$ and $d = d_1 + d_2$. The Segre product $\Pi(\Lambda_1)\totimes_\K \Pi(\Lambda_2)$ with respect to the $\star$-grading has in addition a grading coming from the $(d_1+1)$- respectively $(d_2+1)$-total grading on $\Pi(\Lambda_1)$ and $\Pi(\Lambda_2)$ respectively. However, it does not correspond to the $(d+1)$-total grading on $\Pi(\Lambda_1\otimes_\K \Lambda_2)$. To fix this we regard $\Pi(\Lambda_1\totimes_\K \Lambda_2)$ as a $\Z^2$-graded by
\begin{equation*}
	(\Pi(\Lambda_1)\totimes_\K \Pi(\Lambda_2))_{lk} = \bigoplus_{l_1 + l_2 - k = l}\Pi(\Lambda_1)_{l_1k}\otimes_\K \Pi(\Lambda_2)_{l_2k},
\end{equation*}
where $l_i$ refers to the $(d_i + 1)$-total grading on $\Pi(\Lambda_i)$ and $k$ refers to the $\star$-grading.

\begin{myprop}\label{Proposition - d+1 total grading for tensor product}
	Let $\Lambda_i$ be a $d_i$-representation finite $l$-homogeneous Koszul algebra for each $i\in \{1, 2\}$. Then
	\begin{equation*}
		\Pi(\Lambda_1\otimes_\K \Lambda_2)\cong \Pi(\Lambda_1)\totimes_\K\Pi(\Lambda_2)
	\end{equation*}
	as $\Z^2$-graded algebras.
\end{myprop}

\begin{proof}
	The proof is similar to Lemma~\ref{Lemma - Pi(AxB) = Pi(A)xPi(B)}. In fact letting $E_i = \mathrm{Ext}^{d_i}_{\Lambda_i}(D\Lambda_i, \Lambda_i)\cong \mathrm{Hom}_{\mathcal{D}^b(\Lambda_i\mathrm{-mod})}(\Lambda_i, \nu_{d_i}^{-1}\Lambda_i)$ and $E = \mathrm{Ext}^{d}_{\Lambda}(D\Lambda, \Lambda)\cong \mathrm{Hom}_{\mathcal{D}^b(\Lambda\mathrm{-mod})}(\Lambda_, \nu_{d}^{-1}\Lambda)$ the isomorphism
	\begin{equation*}
		\mathrm{Hom}_{\mathcal{D}^b(\Lambda_1\mathrm{-mod})}(\Lambda_1, \nu_{d_1}^{-1}\Lambda_1)\otimes_\K \mathrm{Hom}_{\mathcal{D}^b(\Lambda_2\mathrm{-mod})}(\Lambda_2, \nu_{d_2}^{-1}\Lambda_2)\cong \mathrm{Hom}_{\mathcal{D}^b(\Lambda\mathrm{-mod})}(\Lambda, \nu_{d}^{-1}\Lambda)
	\end{equation*}
	yields an isomorphism $E_1\otimes_\K E_2 \cong E$ of $\Lambda$-$\Lambda$-bimodules. Thus we have an isomorphism
	\begin{equation*}
		T_\Lambda(E)\cong T_\Lambda(E_1\otimes_\K E_2) = T_{\Lambda_1}(E_1)\totimes_\K T_{\Lambda_2}(E_2)
	\end{equation*}
	which is compatible with the natural $\Z^2$-gradings on $T_\Lambda(E)$ and $T_{\Lambda_1}(E_1)\totimes_\K T_{\Lambda_2}(E_2)$. The above isomorphism induces
	\begin{equation*}
		\begin{aligned}
			\Pi(\Lambda)_{lk} &= \bigoplus_{(d+1)k + j = l} T_\Lambda(E)_{jk} \cong \\
			&\cong \bigoplus_{(d+1)k + j = l} \bigoplus_{j_1 + j_2 = j}T_{\Lambda_1}(E_1)_{j_1k}\otimes_\K T_{\Lambda_2}(E_2)_{j_2k} = \\
			&= \bigoplus_{(d+1)k + j_1 + j_2 = l}T_{\Lambda_1}(E_1)_{j_1k}\otimes_\K T_{\Lambda_2}(E_2)_{j_2k}.
		\end{aligned}
	\end{equation*}
	But $(d+1)k + j_1 + j_2 = l$ if and only if $(d_1+1)k + j_1 + (d_2 + 1)k + j_2 - k = l$, so for $l_i = (d_i+1)k + j_i$ we have
	\begin{equation*}
		\Pi(\Lambda)_{jk} = \bigoplus_{l_1 + l_2 - k = l}\Pi(\Lambda)_{l_1k}\otimes_\K \Pi(\Lambda)_{l_2k}
	\end{equation*}
	as desired.
\end{proof}

When $\Lambda$ is an $d$-representation finite algebra Pasquali studied $d$-almost split sequences and proved the following result.

\begin{mythm}\cite[Theorem 2.4]{pasquali2017tensor}\label{Theorem - Pasquali mapping cone}
	Let $\Lambda$ be a $d$-representation finite $\K$-algebra, and let $C_{d+1}\in \mathrm{add}(\tau_d^{-i}A)$ be an indecomposable non-injective module for some integer $i\in \Z$, and let
	\begin{equation*}
		C_\bullet = 0\to C_{d+1}\xrightarrow{f_{d+1}} C_{d}\to \cdots \to C_1\xrightarrow{f_1} C_0\to 0
	\end{equation*}
	be the corresponding $d$-almost split sequence. Then there are complexes $Q_\bullet\in\mathcal{C}(\mathrm{add}(\tau_d^{-i}A))$, $R_\bullet \in \mathcal{C}(\mathrm{add}(\tau_d^{-(i+1)}A))$ and a morphism $\varphi: Q_\bullet\to R_\bullet$ such that $C_\bullet = C(\varphi)$. 
\end{mythm}

\begin{myrem}
	Note that the morphism $\varphi$ in Theorem~\ref{Theorem - Pasquali mapping cone} is of $\star$-degree $1$ which directly follows from the definition of $\star$-degree. In fact, given a morphism $\varphi$ Pasquali wrote down criterion when $C(\varphi)$ is a $d$-almost split sequence which is formulated in \cite[Lemma 2.7]{pasquali2017tensor}. For the purpose in this article we only need existence of $\varphi$ for a given $d$-almost split sequence.
\end{myrem}

As in \cite{brenner2002periodic}, we consider almost split sequences in order to construct almost Koszul complexes. Grant and Iyama generalized this idea in \cite{grant2020higher} to $d$-representation finite algebras. They introduced the functor
\begin{equation*}
	H^\Z = \bigoplus_{i, j\in \Z}\mathrm{Hom}_{\mathcal{D}^b(\Lambda\mathrm{-mod}^\Z)}(-, \nu^{-i}_d\Lambda(j)): \mathcal{U}^\Z\rightarrow \mathrm{mod}^{\Z^2}\mathrm{-}\Pi,
\end{equation*}
where
\begin{equation*}
	\mathcal{U}^\Z = \mathrm{add}\{\nu_d^{-i}\Lambda(j) \mid i, j\in \Z \}\subset \mathcal{D}^b(\Lambda\mathrm{-mod}^\Z).
\end{equation*}

Note that $H^\Z(\Lambda) = \Pi(\Lambda)\in \mathrm{mod}^{\Z^2}\mathrm{-}\Pi(\Lambda)$ so the essential image of $H^\Z$ is exactly the projective objects in $\mathrm{mod}^{\Z^2}\mathrm{-}\Pi(\Lambda)$. However, the $\Z^2$-grading on $H^\Z(\Lambda)$ does not correspond to the $\Z^2$-grading in Remark~\ref{Proposition - d+1 total grading for tensor product} since the action of $(i, j)$ on $\mathcal{U}^\Z$ is given by $\nu_d^{-i}(j)$. In order to relate these two gradings, we introduce the functor
\begin{equation*}
	\begin{aligned}
		F': \mathrm{mod}^{\Z^2}\mathrm{-}\Pi(\Lambda) &\rightarrow \mathrm{mod}^{\Z^2}\mathrm{-}\Pi(\Lambda) \\
		\bigoplus_{i,j\in \Z} X_{i, j}&\mapsto \bigoplus_{s, t\in \Z}X'_{s, t}
	\end{aligned}
\end{equation*}
where $X'_{s, t} = X_{-t, (d+1)t - s}$. Then the first part of the grading will correspond to the $(d+1)$-total grading and the second part will correspond to the $\star$-grading. In the following theorem we are only interested in the $\star$-grading, and thus we introduce another functor
\begin{equation*}
	\begin{aligned}
		F: \mathrm{mod}^{\Z^2}\mathrm{-}\Pi(\Lambda) &\rightarrow \mathrm{mod}^\Z\mathrm{-}\Pi(\Lambda) \\
		\bigoplus_{i,j\in \Z} X_{i, j}&\mapsto \bigoplus_{t\in \Z}\left(\bigoplus_{s\in \Z}X'_{s, t}\right)
	\end{aligned}.
\end{equation*}
In other words, $F$ forgets the $(d+1)$-total grading.

\begin{mythm}\label{Theorem - Koszul complex is given by a mapping cone}
	Let $\Lambda$ be an acyclic $d$-representation finite Koszul algebra and let $S$ be a simple $\Pi(\Lambda)\mathrm{-mod}^\Z$, then there exist complexes $R_\bullet, Q_\bullet\in \mathcal{C}^b(\Pi(\Lambda)\mathrm{-mod}^\Z, d)$ and an almost quasi-isomorphism $\varphi: Q_\bullet \rightarrow R_\bullet$ such that $\deg^\star (\varphi) = 1$ and $C(\varphi)$ is the almost Koszul complex for $S$.
\end{mythm}

\begin{proof}
	We follow the proof of \cite[Theorem 4.21]{grant2020higher}. They showed that for each simple $S\in \Pi\mathrm{-mod}^\Z$ there is a $d$-almost split sequence
	\begin{equation}\label{eq - d-almost split sequence for S}
		0\rightarrow C_{d+1}\xrightarrow{f_{d+1}}C_d\rightarrow \cdots \rightarrow C_1\xrightarrow{f_1}C_0\rightarrow 0,
	\end{equation}
	in $\mathcal{U}^\Z$ such that
	\begin{equation}\label{eq - Koszul complex for S}
		F\circ H(C_{d+1})\xrightarrow{F\circ H(f_{d+1})}F\circ H(C_n)\rightarrow \cdots \rightarrow F\circ H(C_1)\xrightarrow{F\circ H(f_1)}F\circ H(C_0)\rightarrow 0.
	\end{equation}
	is the almost Koszul complex for $S$. In particular, it only has non-zero homology in position $0$ and $d+1$.
	
	Since the functors $F$ and $H$ are additive, it is enough to show that (\ref{eq - d-almost split sequence for S}) can be written as a mapping cone of some morphism which is homogeneous of $\star$-degree $1$. Using Theorem \ref{Theorem - Pasquali mapping cone} we have that every $d$-almost split sequence in $\mathcal{U}^\Z$ with terms in $\Lambda\mathrm{-mod}^\Z$ can be written as a mapping cone of some morphism $\varphi: Q_\bullet \to R_\bullet$. Moreover, $Q_i\in \mathrm{add}(\nu_d^{-i_0}\Lambda)$ and $R_i\in \mathrm{add}(\nu_d^{-i_0+1}\Lambda)$ and $\deg^\star(\varphi) = 1$. We know that $[1]$ is an equivalence and therefore every $d$-almost split sequence with terms in $(\Lambda\mathrm{-mod}^\Z)[k]$ can also be written as a mapping cone of some morphism which is homogeneous of $\star$-degree $1$ for every $k\in \Z$. Now, since $\nu_d$ is an equivalence, it is left to show that for every $d$-almost sequence $C_\bullet$ in $\mathcal{U}^\Z$ ending in a projective module $P\in \Lambda\mathrm{-mod}^\Z$, i.e. $C_{d+1} = P$, there exists an integer $i\in \Z$ such that $\nu_d^i (C_\bullet)$ is a $d$-almost split sequence with terms in $(\Lambda\mathrm{-mod}^\Z)[k]$ for some integer $k\in \Z$. When $P$ is not an injective module it is clear that $C_\bullet$ is a $d$-almost split sequence in $\Lambda\mathrm{-mod}^\Z$. Now suppose that $P = P_i = \Lambda e_i$ is a graded projective injective $\Lambda$-module. Then $\nu_d^{l_i}P_i = P_{\sigma(i)}[-d]$. If we can show that there exists $j\in \Z$ such that $P_{\sigma^j(i)}$ is not an injective module we are done. Assume that such an integer $j\in \Z$ does not exist. In other words, $P_{\sigma^j(i)}$ is a graded projective-injective $\Lambda$-module for all $j\in \Z$. Since $\Lambda$ is a finite dimensional $\K$-algebra, we know that there exists an integer $n$ such that $\sigma^n(i) = i$, and because we have non-zero morphisms
	\begin{equation*}
		P_{\sigma(i)}\twoheadrightarrow \mathrm{Top}(P_{\sigma(i)}) = \mathrm{Soc}(I_{\sigma(i)}) \rightarrowtail I_{\sigma(i)} = P_i
	\end{equation*}
	which is a contradiction since $\Lambda$ is acyclic. Hence (\ref{eq - d-almost split sequence for S}) can be written as a mapping cone of some morphism $\varphi$ with $\deg^\star(\varphi) = 1$.
	
	Finally observe $F\circ H(C(\varphi)) = C(F\circ H(\varphi))$, which holds because $F$ and $H$ are additive. Now from the definition of $F$ and $H$ we see that $\deg^\star(F\circ H(\varphi)) = 1$. The fact that $F\circ H(\varphi)$ is an almost quasi-isomorphism comes from the fact that (\ref{eq - Koszul complex for S}) only has non-zero homology at position $0$ and $d+1$ and Lemma~\ref{Lemma - a.q-iso implies homology at two places}.
\end{proof}

It is still unclear if we need to assume $\Lambda$ is acyclic since there are no known examples of a $d$-representation finite algebra $\Lambda$ such that $\Lambda$ is not an acyclic algebra. Therefore we make the following conjecture.

\begin{myconjecture}
	Theorem~\ref{Theorem - Koszul complex is given by a mapping cone} still holds if we drop the acyclic assumption on $\Lambda$.
\end{myconjecture}

\begin{mythm}\label{Theorem - Koszul complex for tensor product of species}
	Let $\Lambda_i$ be an acyclic $d_i$-representation finite $l$-homogeneous Koszul algebra for each $i\in \{1, 2\}$. Let $S^{\Lambda_i}$ be a simple $\Pi(\Lambda_i)$-module and let $\varphi^{\Lambda_i}: Q_\bullet^{\Lambda_i}\rightarrow R_\bullet^{\Lambda_i}$ be an almost quasi-isomorphism as in Theorem~\ref{Theorem - Koszul complex is given by a mapping cone} such that $C(\varphi^{\Lambda_i})$ is the almost Koszul complex for $S^{\Lambda_i}$. The complex
	\begin{equation*}
		C(\mathrm{Tot}(\varphi^{\Lambda_1}\totimes \varphi^{\Lambda_2}):\mathrm{Tot}(Q^{\Lambda_1}_\bullet\totimes_\K Q^{\Lambda_2}_\bullet) \rightarrow \mathrm{Tot}(R^{\Lambda_1}_\bullet \totimes_\K R^{\Lambda_2}_\bullet))
	\end{equation*}
	is the almost Koszul complex for $S^{\Lambda_1}\otimes_\K S^{\Lambda_2}\in \Pi(\Lambda_1\otimes_\K \Lambda_2)\mathrm{-mod}^\Z$.
\end{mythm}

\begin{proof}
	Existence of $\varphi^{\Lambda_i}: Q_\bullet^{\Lambda_i}\rightarrow R^{\Lambda_i}_\bullet$ is due to Theorem~\ref{Theorem - Koszul complex is given by a mapping cone}. 
	
	Lemma~\ref{Lemma - Pi(AxB) = Pi(A)xPi(B)} implies that
	\begin{equation*}
		C(\mathrm{Tot}(\varphi^{\Lambda_1}\totimes \varphi^{\Lambda_2}))\in \mathcal{C}(\Pi(\Lambda_1\otimes_\K \Lambda_2)\mathrm{-proj}),
	\end{equation*}
	since $C(\mathrm{Tot}(\varphi^{\Lambda_1}\totimes \varphi^{\Lambda_2}))_i\in \mathrm{add}(\Pi(\Lambda_1)\totimes_\K \Pi({\Lambda_2}))$ for all $i\in \Z$. We show that $\mathrm{Tot}(\varphi^{\Lambda_1}\totimes \varphi^{\Lambda_2})$ is an almost quasi-isomorphism. First note that $\mathrm{Tot}(\varphi^{\Lambda_1}\totimes \varphi^{\Lambda_2})_0 = \varphi^{\Lambda_1}_0\totimes \varphi^{\Lambda_2}_0$ and $\mathrm{Tot}(\varphi^{\Lambda_1}\totimes \varphi^{\Lambda_2})_{d_{\Lambda_1} + d_{\Lambda_2} + 1} = \varphi^{\Lambda_1}_{d_{\Lambda_1}}\totimes \varphi^{\Lambda_2}_{d_{\Lambda_2}}$ and so on the level of homology they induce a monomorphism and an epimorphism respectively. Now let $0<i<d_{\Lambda_1} + d_{\Lambda_2} + 1$. We need to show that
	\begin{equation*}
		\mathrm{Tot}(\varphi^{\Lambda_1}\totimes \varphi^{\Lambda_2})_i = \begin{bmatrix}
		\varphi^{\Lambda_1}_0\totimes \varphi^{\Lambda_2}_i \\
		& \varphi^{\Lambda_1}_1\totimes \varphi^{\Lambda_2}_{i-1} \\
		& & \ddots \\
		& & & \varphi^{\Lambda_1}_i\totimes \varphi^{\Lambda_2}_0
		\end{bmatrix}
	\end{equation*}
	induces an isomorphism on the level of homology. Theorem~\ref{Theorem - Kunneth relation for diagonal tensor product} tells us that
	\begin{equation*}
		H_i(\mathrm{Tot}(\varphi^{\Lambda_1}\totimes \varphi^{\Lambda_2})): \bigoplus_{j+k=i}H_j(Q_\bullet^{\Lambda_1})\totimes_\K H_k(Q_\bullet^{\Lambda_2})\rightarrow \bigoplus_{j+k=i}H_j(R_\bullet^{\Lambda_1})\totimes_\K H_k(R_\bullet^{\Lambda_2}),
	\end{equation*}
	and since $\varphi^{\Lambda_1}$ and $\varphi^{\Lambda_2}$ are both almost quasi-isomorphisms, it is enough to show that all $\varphi^{\Lambda_1}_0\totimes \varphi^{\Lambda_2}_i$, $\varphi^{\Lambda_1}_i\totimes \varphi^{\Lambda_2}_0$, $\varphi^{\Lambda_1}_{d_{\Lambda_1}}\totimes \varphi^{\Lambda_2}_i$ and $\varphi^{\Lambda_1}_i\totimes \varphi^{\Lambda_2}_{d_{\Lambda_2}}$ induces an isomorphism on the level of homology. This is due to the fact that $\varphi_0^{\Lambda_1}$, $\varphi^{\Lambda_2}_0$, $\varphi_{d_{\Lambda_1}}^{\Lambda_1}$ and $\varphi_{d_{\Lambda_2}}^{\Lambda_2}$ are the only morphisms that do not induce isomorphisms on the level of homology. Let us start with $\varphi^{\Lambda_1}_0\totimes \varphi^{\Lambda_2}_i$. First note that elements in $H_i(R_\bullet^{\Lambda_2})$, for $i\not=0$, have $\star$-degree at least $1$ since
	\begin{equation*}
		H_i(\varphi^{\Lambda_2}): H_i(Q_\bullet^{\Lambda_2})\twoheadrightarrow H_i(R_\bullet^{\Lambda_2})
	\end{equation*}
	and the fact that $\deg^\star(\varphi^{\Lambda_2})=1$. Therefore the elements in $H_0(R_\bullet^{\Lambda_1})\totimes_\K H_i(R_\bullet^{\Lambda_2})$ have $\star$-degree at least $1$. Also note that $H_0(\varphi^{\Lambda_1})$ surjects onto the $\star$-degree at least $1$ part of $H_0(R_\bullet^{\Lambda_1})$ because
	\begin{equation*}
		H_0(C(\varphi^{\Lambda_1})) = \coker (H_0(\varphi^{\Lambda_1})) = S^{\Lambda_1}
	\end{equation*}
	is a simple concentrated at $\star$-degree $0$. Thus using that $H_0(\varphi^{\Lambda_1})$ is a monomorphism we get that $\varphi^{\Lambda_1}_0\totimes \varphi^{\Lambda_2}_i$ induces an isomorphism on the level of homology. A similar argument shows that $\varphi^{\Lambda_1}_i\totimes \varphi^{\Lambda_2}_0$ induces an isomorphism on the level of homology. 
	
	For the other two cases, $\varphi^{\Lambda_1}_{d_{\Lambda_1}}\totimes \varphi^{\Lambda_2}_i$ and $\varphi^{\Lambda_1}_i\totimes \varphi^{\Lambda_2}_{d_{\Lambda_2}}$ we need that $\Lambda_i$ is $l$-homogeneous. The $l$-homogeneous property of $\Lambda_i$ ensures that the socle of each projective module is in $\star$-degree $l-1$, and therefore we can use the dual version of the argument as follows. We only show that $\varphi^{\Lambda_1}_{d_{\Lambda_1}}\totimes \varphi^{\Lambda_2}_i$ induces an isomorphism on the level of homology since the argument for $\varphi^{\Lambda_1}_i\totimes \varphi^{\Lambda_2}_{d_{\Lambda_2}}$ is similar. The elements in $H_i(Q_\bullet^{\Lambda_2})$, for $i\not=d_{\Lambda_2}$, have $\star$-degree at most $l - 1$ since
	\begin{equation*}
		H_i(\varphi^{\Lambda_2}): H_i(Q_\bullet^{\Lambda_2})\rightarrowtail H_i(R_\bullet^{\Lambda_2})
	\end{equation*}
	and the fact that $\deg^\star(\varphi^{\Lambda_2}) = 1$. Thus the elements in $H_{d_{\Lambda_1}}(R_\bullet^{\Lambda_1})\totimes_\K H_i(R_\bullet^{\Lambda_2})$ have $\star$-degree at most $l-1$. Now using that
	\begin{equation*}
		H_{d_{\Lambda_1}}(C(\varphi^{\Lambda_1})) = \ker (H_{d_{\Lambda_1}}(\varphi^{\Lambda_1}))
	\end{equation*}
	is concentrated in $\star$-degree $l$, the morphism $\varphi^{\Lambda_1}_{d_{\Lambda_1}}\totimes \varphi^{\Lambda_2}_i$ will induce an isomorphism on the level of homology. Thus $\mathrm{Tot}(\varphi^{\Lambda_1}\totimes \varphi^{\Lambda_2})$ is an almost quasi-isomorphism.
	
	To conclude that this in fact is the almost Koszul complex for $S^{\Lambda_1}\otimes_\K S^{\Lambda_2}$ we need to show that
	\begin{equation*}
		H_0(C(\mathrm{Tot}(\varphi^{\Lambda_1}\totimes \varphi^{\Lambda_2}))) = S^{\Lambda_1}\otimes_\K S^{\Lambda_2},
	\end{equation*}
	and compute the homology at $d_{\Lambda_1} + d_{\Lambda_2} + 1$. This is shown by again using the fact that
	\begin{equation*}
		H_0(\varphi^{\Lambda_i}): H_0(Q_\bullet^{\Lambda_i})\rightarrow H_0(R_\bullet^{\Lambda_i})
	\end{equation*}
	surjects onto the $\star$-degree $\ge 1$ part of $H_0(R_\bullet^{\Lambda_i})$ and the fact that
	\begin{equation*}
		H_0(C(\mathrm{Tot}(\varphi^{\Lambda_1}\totimes \varphi^{\Lambda_2}))) = \coker (H_0(\varphi^{\Lambda_1})\totimes H_0(\varphi^{\Lambda_2})).
	\end{equation*}
	
	Note that we did not have to use the fact that both $\Lambda_1$ and $\Lambda_2$ are $l$-homogeneous to compute the homology at $0$, but to compute the homology at $d_{\Lambda_1} + d_{\Lambda_2} + 1$, we need to use $l$-homogeneous property to ensure that elements in $H_{d_{\Lambda_1}}(C(\varphi^{\Lambda_1}))$ and $H_{d_{\Lambda_2}}(C(\varphi^{\Lambda_2}))$ have $\star$-degree $l$. We want to show that
	\begin{equation}\label{eq - kernel of tot almost quasi-iso}
		\ker (H_{d_{\Lambda_1} + d_{\Lambda_2} + 1}(C(\mathrm{Tot}(\varphi^{\Lambda_1}\totimes \varphi^{\Lambda_2})))) = H_{d_{\Lambda_1}}(C(\varphi^{\Lambda_1})) \totimes_\K H_{d_{\Lambda_2}}(C(\varphi^{\Lambda_2})).
	\end{equation}
	
	First note that
	\begin{equation*}
		\overline{H_{d_{\Lambda_1}}(\varphi^{\Lambda_1})}: H_{d_{\Lambda_1}}(Q_\bullet^{\Lambda_1})/H_{d_{\Lambda_1}+1}(C(\varphi^{\Lambda_1}))\xrightarrow{\sim} H_{d_{\Lambda_1}}(R_\bullet^{\Lambda_1})
	\end{equation*}
	is an isomorphism, and since both $\Lambda_1$ and $\Lambda_2$ are $l$-homogeneous the morphism
	\begin{equation*}
		\begin{aligned}
			\overline{H_{d_{\Lambda_1} + d_{\Lambda_2}}(\mathrm{Tot}(\varphi^{\Lambda_1}\totimes \varphi^{\Lambda_2}))}:\qquad & \\
			(H_{d_{\Lambda_1}}(Q_\bullet^{\Lambda_1})\totimes_\K H_{d_{\Lambda_2}}(Q_\bullet^{\Lambda_2}))/&(H_{d_{\Lambda_1} + 1}(C(\varphi^{\Lambda_1}))\totimes_\K H_{d_{\Lambda_2} + 1}(C(\varphi^{\Lambda_2})))\xrightarrow{\sim} \\
			&\hspace{2cm} \xrightarrow{\sim}H_{d_{\Lambda_1}}(R_\bullet^{\Lambda_1})\totimes_\K H_{d_{\Lambda_2}}(R_\bullet^{\Lambda_2})
		\end{aligned}
	\end{equation*}
	is also an isomorphism. Hence (\ref{eq - kernel of tot almost quasi-iso}) holds.
\end{proof}

\begin{mycor}\label{Corollary - Pi(AxB) is almost Koszul}
	Let $\Lambda_i$ be an acyclic $d_i$-representation finite $l$-homogeneous Koszul algebra for each $i\in \{1, 2\}$. If $\Pi(\Lambda_i)$ is an $(p_i, q_i)$-almost Koszul algebra, then $\Pi(\Lambda_1\otimes_\K \Lambda_2)$ is an $(p_1 + p_2 - l + 2, q_1 + q_2 - 1)$-almost Koszul algebra.
\end{mycor}

\begin{proof}
	Let $\varphi^{\Lambda_i}: Q^{\Lambda_i}_\bullet \rightarrow R^{\Lambda_i}_\bullet$ be a morphism such that $C(\varphi)$ is the almost Koszul complex for $\Pi(\Lambda_i)_0$. By Theorem~\ref{Theorem - Koszul complex for tensor product of species} $C(\mathrm{Tot}(\varphi^{\Lambda_1}\totimes \varphi^{\Lambda_2}))$ is the almost Koszul complex for $\Pi(\Lambda_1\otimes_\K \Lambda_2)_0$. Since $q_i$ describes the length of the almost Koszul complex, we can use the definition of $\mathrm{Tot}(-\totimes_\K -)$ to find the length of the almost Koszul complex $C(\mathrm{Tot}(\varphi^{\Lambda_1}\totimes \varphi^{\Lambda_2}))$ to be $q_1 + q_2 - 1$. 
	
	By Theorem~\ref{Theorem - Koszul complex for tensor product of species} the last projective module of the almost Koszul complex \\ $C(\mathrm{Tot}(\varphi^{\Lambda_1}\totimes \varphi^{\Lambda_2}))$ will be given by $Q_{q_1}^{\Lambda_1}\totimes_\K Q_{q_2}^{\Lambda_2}$, i.e.
	\begin{equation*}
		\begin{aligned}
			0\rightarrow Q_{q_1}^{\Lambda_1}\totimes_\K Q_{q_2}^{\Lambda_2}\xrightarrow{f}R_{q_1}^{\Lambda_1}\totimes_\K R_{q_2}^{\Lambda_2}\oplus Q_{q_1}^{\Lambda_1}\totimes_\K Q_{q_2-1}^{\Lambda_2} \oplus Q_{q_1-1}^{\Lambda_1}\totimes_\K Q_{q_2}^{\Lambda_2} \rightarrow \cdots \rightarrow R^{\Lambda_1}_0\otimes_\K Q^{\Lambda_2}_0\rightarrow 0
		\end{aligned}
	\end{equation*}
	is the almost Koszul complex for $\Pi({\Lambda_1}\otimes_\K \Lambda_2)_0$. Here
	\begin{equation*}
		f =  \begin{bmatrix}
			\varphi^{\Lambda_1}_{q_1}\totimes \varphi^{\Lambda_2}_{q_2} \\
			1_{R_{q_1}}\totimes d_{q_2}^{C(\varphi^{\Lambda_2})} \\
			d_{q_1}^{C(\varphi^{\Lambda_1})}\totimes 1_{Q_{q_2}}
		\end{bmatrix}.
	\end{equation*}
	Thus we can compute the kernel as
	\begin{equation*}
		\ker f= \ker (\varphi^{\Lambda_1}_{q_{\Lambda_1}}\totimes \varphi^{\Lambda_2}_{q_{\Lambda_2}})\cap \ker (1_{R_{q_{\Lambda_1}}}\totimes d_{q_{\Lambda_2}}^{C(\varphi^{\Lambda_2})})\cap \ker (d_{q_{\Lambda_1}}^{C(\varphi^{\Lambda_1})}\totimes 1_{Q_{q_{\Lambda_2}}})
	\end{equation*}
	and hence $\ker f = H_{q_1}(C(\varphi^{\Lambda_1}))\totimes_\K H_{q_2}(C(\varphi^{\Lambda_2}))$. The algebras $\Lambda_1$ and $\Lambda_2$ being $l$-homogeneous ensures that $H_{q_1}(C(\varphi^{\Lambda_1}))$ and $H_{q_2}(C(\varphi^{\Lambda_2}))$ are concentrated in $\star$-degree $l$, which implies that
	\begin{equation*}
		\ker f = H_{q_1}(C(\varphi^{\Lambda_1}))\totimes_\K H_{q_2}(C(\varphi^{\Lambda_2})) = H_{q_1}(C(\varphi^{\Lambda_1}))\otimes_\K H_{q_2}(C(\varphi^{\Lambda_2}))
	\end{equation*}
	
	It is left to show that $\deg \ker f = p_1 + p_2 + q_1 + q_2 - l + 1$, but this follows from Proposition~\ref{Proposition - d+1 total grading for tensor product}.
\end{proof}

\begin{myprop}\label{Proposition - Nakayama automorphism of tensor product of algebras}
	Let $\Lambda_i$ be a $d_i$-representation finite $l$-homogeneous Koszul algebra. Let $\gamma_i$ be the Nakayama automorphism for $\Pi(\Lambda_i)$. Then $\gamma_1\totimes \gamma_2$ is the Nakayama automorphism for $\Pi(\Lambda_1\otimes_\K \Lambda_2)$.
\end{myprop}

\begin{proof}
	By Lemma~\ref{Lemma - Pi(AxB) = Pi(A)xPi(B)} we have to show that
	\begin{equation}\label{eq - nakayama automorphism for Pi(S^1 x S^2)}
		\Pi(\Lambda_1)\totimes_\K \Pi(\Lambda_2)\cong D(\Pi(\Lambda_1)\totimes_\K \Pi(\Lambda_2))_{\gamma_1\totimes \gamma_2}.
	\end{equation}
	
	First recall that the Nakayama automorphism is an isomorphism of graded modules, where the grading is defined via the path length. This is seen as an application of \cite[Proposition 2.4]{herschend2011n} together with the fact that inner automorphisms does not change the grading. Here we defined the grading on $D\Pi(\Lambda_i)$ to be the induced grading from $\Pi(\Lambda_i)$. We want to ensure that the Nakayama automorphism is a morphism of $\star$-graded modules, in particular, we want that $\deg^\star(\gamma_i) = 0$. Since $\Lambda_1$ and $\Lambda_2$ are $l$-homogeneous we can apply \cite[Theorem 2.3]{herschend2011n} to show $\deg^\star(\gamma_i)=0$. Thus $\deg^\star(\gamma_1\otimes \gamma_2) = 0$, which implies that $\gamma_1\totimes \gamma_2$ will be an automorphism on $\Pi(\Lambda_1)\totimes_\K \Pi(\Lambda_2)$. At degree $i$ we have
	\begin{equation*}
		\begin{aligned}
		(\Pi(\Lambda_1)\totimes_\K \Pi(\Lambda_2))_i&\cong \Pi(\Lambda_1)_i\otimes_\K \Pi(\Lambda_2)_i \cong (D\Pi(\Lambda_1)_{\gamma_1})_i\otimes_\K (D\Pi(\Lambda_2)_{\gamma_2})_i \cong \\
		&\cong (D\Pi(\Lambda_1)_i\otimes_\K D\Pi(\Lambda_2)_i)_{\gamma_1\otimes \gamma_2}.
		\end{aligned}
	\end{equation*}
	Since it holds for every $i\in \Z$, we can conclude that (\ref{eq - nakayama automorphism for Pi(S^1 x S^2)}) holds.
\end{proof}

\section{Computations for Tensor Products of Species with Relations}\label{Section - Examples}
We devote this section to examples we get when we apply the theory we developed in this paper. We can generate a lot of examples if we start with representation finite species. It is important that these species are $l$-homogeneous due to Corollary~\ref{Proposition - A x B is n_1 + n_2 RF}, and thanks to Corollary~\ref{Corollary - l homogeneous species} we have complete set of $l$-homogeneous representation finite species. For all representation finite species $S$, we have an explicit description of the almost Koszul complexes for simple $\Pi(S)$-modules and the Nakayama automorphism. By Corollary~\ref{Proposition - A x B is n_1 + n_2 RF}, tensor products between representation finite $l$-homogeneous species $S^1$ and $S^2$ are $2$-representation finite $l$-homogeneous species with relations which we will describe completely in Corollary~\ref{Corollary - tensor product of species is a species explicitly}. Applying Theorem~\ref{Theorem - product of species is almost koszul} we know that $\Pi(T(S^1)\otimes_\K T(S^2))$ is an almost Koszul algebra. Since we established earlier that we already know the almost Koszul complexes for all representation finite species, we can apply Theorem~\ref{Theorem - Koszul complex for tensor product of species} to fully describe the almost Koszul complexes for the simple $\Pi(T(S^1)\otimes_\K T(S^2))$-modules. We also have a description of the Nakayama automorphism of $\Pi(T(S^1)\otimes_\K T(S^2))$ due to Proposition~\ref{Proposition - Nakayama automorphism of tensor product of algebras}. Let us now compute some explicit examples to illustrate this.

\begin{myex}\label{example - S of Dynkin type C_3, computing koszul complex}
	Let $S$ be the species in Example~\ref{example - S species of Dynkin type C_3} over $Q: 1\to 2\to 3$. The Auslander-Reiten quiver $\Gamma_S$ is
	\begin{equation*}
		\begin{tikzcd}[sep={40pt,between origins}]
			& & P_1 \arrow[dr, "2"] & & \tau^{-1}P_1 \arrow[dr, "2"] & & I_1 \\
			& P_2 \arrow[ur, "2"] \arrow[dr] & & \tau^{-1}P_2 \arrow[ur, "2"] \arrow[dr] & & I_2 \arrow[ur, "2"] \\
			P_3 \arrow[ur] & & \tau^{-1}P_3 \arrow[ur] & & I_3 \arrow[ur]
		\end{tikzcd}
	\end{equation*}
	By Proposition~\ref{Proposition - Koszul resolution in Dynkin case} the almost Koszul complex for the simple $\Pi(S)$-module $D_1$ is
	\begin{equation}\label{eq - Almost kosul complex in example 1}
		R_\bullet: 0\rightarrow P_1\rightarrow P_2\oplus P_2 \rightarrow P_1\rightarrow 0,
	\end{equation}
	where $H_0(R_\bullet) = D_1$ and $H_2(R_\bullet) = D_3$. The map $P_1\rightarrow P_2$ in (\ref{eq - Almost kosul complex in example 1}) corresponds to the morphism $P_1\to \tau^{-1}P_2$ in the Auslander-Reiten quiver $\Gamma_S$, and so has $\star$-degree $1$. Therefore we can view (\ref{eq - Almost kosul complex in example 1}) as the mapping cone of the morphism
	\begin{equation*}
		\begin{tikzcd}
			0 \arrow[r] & P_1 \arrow[r] \arrow[d] & 0 \arrow[r] \arrow[d] & 0 \\
			0 \arrow[r] & P_2\oplus P_2 \arrow[r] & P_1 \arrow[r] & 0.
		\end{tikzcd}
	\end{equation*}
\end{myex}

\begin{myex}\label{example - S of Dynkin type D_4, computing koszul complex}
	Let $Q$ be the quiver
	\begin{equation*}
		\begin{tikzcd}[row sep = 0]
			& & 3 \\
			1 \arrow[r] & 2 \arrow[ru] \arrow[rd] \\
			& & 4
		\end{tikzcd},
	\end{equation*}
	and let $S$ be a species over $Q$ such that $T(S) = \R Q$. The Auslander-Reiten quiver $\Gamma_S$ is
	\begin{equation*}
		\begin{tikzcd}[sep={40pt,between origins}]
			P_3 \arrow[dr] & & \tau^{-1}P_3 \arrow[dr] & & I_3 \arrow[dr] \\
			& P_2 \arrow[ur] \arrow[dr] \arrow[r] & P_1 \arrow[r] & \tau^{-1}P_2 \arrow[ur] \arrow[dr] \arrow[r] & \tau^{-1}P_1 \arrow[r] & I_2 \arrow[r] & I_1 \\
			P_4 \arrow[ur] & & \tau^{-1}P_4 \arrow[ur] & & I_4 \arrow[ur]
		\end{tikzcd}
	\end{equation*}
	By Proposition~\ref{Proposition - Koszul resolution in Dynkin case} the almost Koszul complex for the simple $\Pi(S)$-module $D_2$ is
	\begin{equation}\label{eq - Almost Koszul complex for example 2}
		R_\bullet: 0\rightarrow P_2 \rightarrow P_1 \oplus P_3 \oplus P_4 \rightarrow P_2 \rightarrow 0,
	\end{equation}
	where $H_0(R_\bullet) = D_2$ and $H_2(R_\bullet) = D_2$. As in Example \ref{example - S of Dynkin type C_3, computing koszul complex} we can deduce that the morphism $P_1\rightarrow P_2$ and $P_2\rightarrow P_3\oplus P_4$ are the only morphisms of $\star$-degree $1$ by looking at $\Gamma_S$. Hence we can view (\ref{eq - Almost Koszul complex for example 2}) as the mapping cone of
	\begin{equation*}
		\begin{tikzcd}
			0 \arrow[r] & P_2 \arrow[r] \arrow[d] & P_1 \arrow[r] \arrow[d] & 0 \\
			0 \arrow[r] & P_3\oplus P_4 \arrow[r] & P_2 \arrow[r] & 0.
		\end{tikzcd}
	\end{equation*}
\end{myex}

\begin{myex}\label{example - koszul complex for C_3 otimes D_4}
	Let $S^1$ and $S^2$ be the species from Example~\ref{example - S of Dynkin type C_3, computing koszul complex} and Example~\ref{example - S of Dynkin type D_4, computing koszul complex} respectively. Reading from the table in Corollary~\ref{Corollary - l homogeneous species}, or by studying their Auslander-Reiten quivers, we see that both $S^1$ and $S^2$ are $3$-homogeneous. Therefore $T(S^1)\otimes_\K T(S^2)$ is a $3$-homogeneous $2$-representation finite algebra. We can use the almost Koszul complexes in Example~\ref{example - S of Dynkin type C_3, computing koszul complex} and Example~\ref{example - S of Dynkin type D_4, computing koszul complex} together with Theorem~\ref{Theorem - Koszul complex for tensor product of species} to compute the almost Koszul complex for the simple module $D^1_1\otimes_\K D^2_2$ in $\Pi(T(S^1)\otimes_\K T(S^2))\mathrm{-mod}$. Let us define $P_{ij} = P^1_i\totimes_\K P^2_j$, then the almost Koszul complex is given as the mapping cone of
	\begin{equation*}
		\begin{tikzcd}
			0 \arrow[r] & P_{12} \arrow[r] \arrow[d] & P_{11} \arrow[r] \arrow[d] & 0 \arrow[r] \arrow[d] & 0 \\
			0 \arrow[r] & P_{23}\oplus P_{23}\oplus P_{24}\oplus P_{24} \arrow[r] & P_{22}\oplus P_{22}\oplus P_{13}\oplus P_{14} \arrow[r] & P_{12} \arrow[r] & 0.
		\end{tikzcd}
	\end{equation*}
\end{myex}

\begin{myex}
	Let $S^1$ and $S^2$ be the species from Example~\ref{example - S of Dynkin type C_3, computing koszul complex} and Example~\ref{example - S of Dynkin type D_4, computing koszul complex} respectively. Let $S^3$ be the species over the quiver
	\begin{equation*}
		Q^3: 1\rightarrow 2\rightarrow 3\leftarrow 4\leftarrow 5
	\end{equation*}
	such that $T(S^3) = \R Q$. Similarly as in Example \ref{example - S of Dynkin type C_3, computing koszul complex} and Example \ref{example - S of Dynkin type D_4, computing koszul complex}, using Proposition~\ref{Proposition - Koszul resolution in Dynkin case} we have the almost Koszul complex for $D_4^3$ in $\Pi(S^3)\mathrm{-mod}$ described as the mapping cone of
	\begin{equation*}
		\begin{tikzcd}
			0 \arrow[r] & P_4 \arrow[r] \arrow[d] & P_5 \arrow[r] \arrow[d] & 0 \\
			0 \arrow[r] & P_3 \arrow[r] & P_4 \arrow[r] & 0.
		\end{tikzcd}
	\end{equation*}
	By Corollary~\ref{Corollary - l homogeneous species} we have that the species $S^3$ is $3$-homogeneous. Thus, by Corollary~\ref{Proposition - A x B is n_1 + n_2 RF}, the tensor product $\Lambda = T(S^1)\otimes_\K T(S^2)\otimes_\K T(S^3)$ is a $3$-homogeneous $3$-representation finite algebra.
	
	Let us now compute the almost Kosul complex for $D_1^1\otimes_\K D_2^2\otimes_\K D_4^3$ in $\Pi(\Lambda)\mathrm{-mod}$. Applying Theorem~\ref{Theorem - Koszul complex for tensor product of species} we get the following almost Koszul complex for $D_1^1\otimes_\K D_2^2\otimes_\K D_4^3$ in $\Pi(\Lambda)\mathrm{-mod}$ as the mapping cone of
	\begin{equation*}
		\begin{tikzcd}
			0 \arrow[r] & P_{124} \arrow[r] \arrow[d] & P_{114}\oplus P_{125} \arrow[r] \arrow[d] & P_{115} \arrow[r] \arrow[d] & 0 \arrow[d] \arrow[r] & 0 \\
			0 \arrow[r] & P_{233}\oplus P_{233}\oplus P_{243}\oplus P_{243} \arrow[r] & \substack{
				P_{223}\oplus P_{223}\oplus P_{133}\oplus P_{143} \\
				\oplus \\
				P_{234}\oplus P_{234} \oplus
				P_{244}\oplus P_{244}
			} \arrow[r] & \substack{P_{123}\oplus P_{224}\oplus P_{224} \\ \oplus \\ P_{134}\oplus P_{144}} \arrow[r] & P_{124} \arrow[r] & 0
		\end{tikzcd}
	\end{equation*}
	where $P_{ijk} = P_i^1\totimes_\K P_j^2\totimes_\K P_k^3$.
\end{myex}

We have now illustrated with some example on how to compute the almost Koszul complex using Proposition~\ref{Proposition - Koszul resolution in Dynkin case} and Theorem~\ref{Theorem - Koszul complex for tensor product of species}. Let us now go back to Example~\ref{example - koszul complex for C_3 otimes D_4} and compute the preprojective algebra more explicitly and describe its Nakayama automorphism. To describe the Segre products and tensor products between species with relations we need the following proposition.

\begin{myprop}\label{Proposition - diagonal tensor product between species}
	For $k\in \{1, 2\}$, let $S^k$ be a graded species as in Remark~\ref{Remark - Graded species}, with $M^k = M^k_0\oplus M^k_1$. Moreover, let $R^k\subset T(S^k)$ be a homogeneous ideal generated in degree $0$ and $1$. Then
	\begin{equation*}
		T(S^1)/R^1 \totimes_\K T(S^2)/R^2\cong T(S)/R,
	\end{equation*}
	where $T(S) = T(D, M)$, $D=D^1\otimes_\K D^2$ and $M=(M^1_0\otimes_\K D^2)\oplus (D^1\otimes_\K M^2_0) \oplus (M^1_1\otimes_\K M^2_1)$, and
	\begin{equation*}
		R = \langle R^1_0\otimes 1_{D^2}, 1_{D^1}\otimes R^2_0, R^1_1\otimes_\K M^2_1, M^1_1\otimes_\K R^2_1, [\alpha\otimes 1_{D^2}, 1_{D^1}\otimes \alpha']\mid\alpha\in M^1_0, \alpha'\in M^2_0\rangle,
	\end{equation*}
	where
	\begin{equation*}
		[\alpha\otimes 1_{D^2}, 1_{D^1}\otimes \alpha'] = (\alpha\otimes 1_{D^2})\otimes_{D^1\otimes_\K D^2} (1_{D^1}\otimes \alpha') - (\alpha\otimes 1_{D^2})\otimes_{D^1\otimes_\K D^2} (1_{D^1}\otimes \alpha').
	\end{equation*}
\end{myprop}

\begin{myrem}
	In Definition~\ref{Definition - Species} we require that $D^1\otimes_\K D^2$ is a direct sum of division rings, but in general it will only be Morita equivalent to a direct sum of division rings. Therefore we consider the more general definition of a species in Remark~\ref{Remark - General definition of a species}. For example, if $D^1 = D^2 = \mathbb{H}$, then $D^1\otimes_\K D^2$ is isomorphic to the space of $4\times 4$-matrices over $\R$, which in turn is Morita equivalent to $\R$.
\end{myrem}

\begin{proof}
	Since $\mathbb{K}$ is perfect we have that $D^1\otimes_\K D^2$ is semi-simple. First we check that $S$ is a species. It is enough to show that $S$ is dualisable. Let $i\xrightarrow{\alpha}j\in Q_1^1$ and $m\xrightarrow{\beta}n\in Q_1^2$. Since $S^1$ and $S^2$ are dualisable, we have graded isomorphisms 
	\begin{equation*}
		\begin{aligned}
			\phi_\alpha^1:&\mathrm{Hom}_{(D_i^1)^{op}}(M_\alpha^1, D_i^1)\xrightarrow{\sim}\mathrm{Hom}_{D_j^1}(M_\alpha^1, D_j^1) \\ \phi_{\beta}^2:&\mathrm{Hom}_{(D_m^2)^{op}}(M_\beta^2, D_m^2)\xrightarrow{\sim}\mathrm{Hom}_{D_n^2}(M_\beta^2, D_n^2) 
		\end{aligned}
	\end{equation*}
	Now let $\psi_{D_i^k}: \mathrm{Hom}_{{D^k_i}^{op}}(D^k_i, D^k_i)\xrightarrow{\sim}\mathrm{Hom}_{D^k_i}(D^k_i, D^k_i)$ be the isomorphism that sends $L_d\mapsto R_d$, where $L_d$ and $R_d$ denote left and right multiplication by $d\in D^k$ respectively, for all $D_i^k$. Consider the following maps
	\begin{equation}\label{eq - three maps for dualiseable condition}
		\begin{aligned}
			\phi_\alpha^1\otimes \psi_{D_k^2}:&\mathrm{Hom}_{(D_i^1\otimes_\K D_k^2)^{op}}(M_\alpha^1\otimes_\K D_k^2, D_i^1\otimes_\K D_k^2)\rightarrow \mathrm{Hom}_{D_j^1\otimes_\K D_k^2}(M_\alpha^1\otimes_\K D_k^2, D_j^1\otimes_\K D_k^2), \\
			\psi_{D_k^1}\otimes\phi_\beta^2:&\mathrm{Hom}_{(D_k^1\otimes_\K D_m^2)^{op}}(D_k^1\otimes_\K M_\beta^2, D_k^1\otimes_\K D_m^2)\rightarrow \mathrm{Hom}_{D_k^1\otimes_\K D_n^2}(D_k^1\otimes_\K M_\beta^2, D_k^1\otimes_\K D_n^2), \\
			\phi_\alpha^1\otimes\phi_\beta^2:& \mathrm{Hom}_{(D_i^1\otimes_\K D_m^2)^{op}}(M^1_\beta\otimes_\K M_\beta^2, D_i^1\otimes_\K D_m^2)\rightarrow \mathrm{Hom}_{D_j^1\otimes_\K D_n^2}(M_\alpha^1\otimes_\K M_\beta^2, D_j^1\otimes_\K D_n^2),
		\end{aligned}
	\end{equation}
	The morphisms in (\ref{eq - three maps for dualiseable condition}) are graded isomorphisms since all of them are tensor product of two graded isomorphisms. This shows that $S$ is dualisable.
	
	Let $X^i_\bullet$ be the exact sequence 
	\begin{equation*}
		0\rightarrow R^i\xrightarrow{f^i} T(S^i)
	\end{equation*}
	where $X^i_0 = T(S^i)$. Then $T(S^i)/R^i = H_0(X^i_\bullet)$. We shift our focus to the total complex $\mathrm{Tot}(X^1_\bullet\totimes_\K X^2_\bullet)$ which is explicitly
	\begin{equation*}
		0\rightarrow R^1\totimes_\K R^2\rightarrow T(S^1)\totimes_\K R^2\oplus R^1\totimes_\K T(S^2)\xrightarrow{f = \begin{bmatrix}
				1\totimes f^2 & f^1\totimes 1
		\end{bmatrix}} T(S^1)\totimes_\K T(S^2) \rightarrow 0.
	\end{equation*}
	Here we see that $\coker(f) = T(S^1)/R^1 \totimes_\K T(S^2)/R^2$ by using the Künneth formula in Theorem~\ref{Theorem - Kunneth relation for diagonal tensor product}.
	
	By the structure of the tensor product
	\begin{equation*}
		(\alpha_1\otimes 1)(1\otimes \alpha_2) = \alpha_1\otimes \alpha_2 = (1\otimes \alpha_2)(\alpha_1\otimes 1),
	\end{equation*}
	where $\alpha_i\in M^i_0$. Therefore we have a natural epimorphism
	\begin{equation}\label{eq - natural epimorphism onto T(S)/R'}
		\xi: T(S)\twoheadrightarrow T(S^1)\totimes_\K T(S^2),
	\end{equation}
	with kernel $R' = \langle (\alpha_1\otimes 1)(1\otimes \alpha_2) - (1\otimes \alpha_2)(\alpha_1\otimes 1)\mid\alpha_i\in M^i_0\rangle$. To see this we first note that every element in $T(S)/R'$ can be written as a linear combination of elements of the form
	\begin{equation}\label{eq - unique representation}
		(\alpha_0\otimes 1)(1\otimes \beta_0)(\gamma_1\otimes \delta_1)(\alpha_1\otimes 1)(1\otimes \beta_1)\cdots(\gamma_N\otimes \delta_N)(\alpha_N\otimes 1)(1\otimes \beta_N),
	\end{equation}
	where $\alpha_k\in (M^1_0)^{\otimes a_k}$, $\beta_k\in (M^2_0)^{\otimes b_k}$ and $\gamma_k\otimes \delta_k\in M^1_1\totimes_\K M^2_1$ for $k\in \Z_{\ge 0}$ and some $a_k, b_k, N\in \Z_{\ge 0}$. Let $\underline{\alpha}$ and $\underline{\beta}$ be bases of $M^1$ and $M^2$ respectively, chosen such that $\underline{\alpha}$ and $\underline{\beta}$ consists of homogeneous elements. Then we can create a generating set $G\subset T(S)/ R$ consisting of non-zero elements of the form (\ref{eq - unique representation}) where $\alpha_k\in \underline{\alpha}^{\otimes a_k}$ and $\beta_k\in \underline{\beta}^{\otimes b_k}$ are of degree $0$ and $\gamma_k\in \underline{\alpha}\totimes_\K \underline{\beta}$ is of degree $1$. The image of $G$ under
	\begin{equation}\label{eq - T(S)/R' -> diagonal tensor}
		\tilde{\xi}: T(S)/R'\twoheadrightarrow T(S^1)\totimes_\K T(S^2).
	\end{equation}
	is a linearly independent set. Since $|G| = |\tilde{\xi}(G)|$ we have that $G$ is linearly independent and hence a basis of $T(S)/R'$. Therefore (\ref{eq - T(S)/R' -> diagonal tensor}) is injective. Thus (\ref{eq - T(S)/R' -> diagonal tensor}) is bijective and hence an isomorphism.
	
	Using (\ref{eq - T(S)/R' -> diagonal tensor}) we have
	\begin{equation*}
		\begin{tikzcd}
			& T(S)/R' \arrow[d, "\tilde{\xi}"] \\
			T(S^1)\totimes_\K R^2\oplus R^1\totimes_\K T(S^2) \arrow[r, "f"] \arrow[ur, "\tilde{f}"] & T(S^1)\totimes_\K T(S^2)
		\end{tikzcd},
	\end{equation*}
	where $\tilde{f}$ is defined so the diagram commutes. Then $T(S)/R\cong\coker\tilde{f}$ which proves the proposition.
\end{proof}

\begin{mycor}\label{Corollary - tensor product of species is a species explicitly}
	Let $S^i$ be a species with relations $R^i$. Then
	\begin{equation*}
		T(S^1)/R^1 \otimes_\K T(S^2)/R^2\cong T(S)/R,
	\end{equation*}
	where $T(S) = T(D=D^1\otimes_\K D^2, M=(M^1\otimes_\K D^2)\oplus (D^1\otimes_\K M^2))$ and
	\begin{equation*}
		R = \langle R^1\otimes 1_{D^2}, 1_{D^1}\otimes R^2,  [\alpha\otimes 1_{D^2}, 1_{D^1}\otimes \alpha']\mid\alpha\in M^1, \alpha'\in M^2\rangle,
	\end{equation*}
	where
	\begin{equation*}
		[\alpha\otimes 1_{D^2}, 1_{D^1}\otimes \alpha'] = (\alpha\otimes 1_{D^2})\otimes_{D^1\otimes_\K D^2} (1_{D^1}\otimes \alpha') - (\alpha\otimes 1_{D^2})\otimes_{D^1\otimes_\K D^2} (1_{D^1}\otimes \alpha').
	\end{equation*}
\end{mycor}

\begin{proof}
	Define a grading on $T(S^i)/R^i$ such that everything lies in degree $0$ and then apply Proposition~\ref{Proposition - diagonal tensor product between species}.
\end{proof}

\begin{myex}
	Let $S^1$ and $S^2$ be the species from Example~\ref{example - S of Dynkin type C_3, computing koszul complex} and Example~\ref{example - S of Dynkin type D_4, computing koszul complex} respectively. Due to \cite{Dlab_1980} we have a complete description of $\Pi(S^1)$ and $\Pi(S^2)$. The preprojective algebra of $S^1$ is given by $\Pi(S^1) = T(\overline{S}^1)/\langle c_1\rangle$, where
	\begin{equation}\label{eq - double quiver of C_3}
		\begin{tikzcd}
			\overline{S}^1: \C^1_1 \arrow[r, shift left, "\C"] & \R^1_2 \arrow[l, shift left, "\C^*"] \arrow[r, shift left, "\R"] & \R^1_3 \arrow[l, shift left, "\R^*"]
		\end{tikzcd}
	\end{equation}
	and
	\begin{equation*}
		c_1 = \sum_{\alpha\in \overline{Q}_1^1}\sgn(\alpha)c_\alpha.
	\end{equation*}
	
	The preprojective algebra of $S^2$ is given by $\Pi(S^2) = T(\overline{S}^2)/\langle c_2\rangle$, where
	\begin{equation}\label{eq - double quiver of D_4}
		\begin{tikzcd}[row sep = 10]
			& & \R^2_3 \arrow[ld, shift left, "\R^*"] \\
			\overline{S}^2: \R^2_1 \arrow[r, shift left, "\R"] & \R^2_2 \arrow[ru, shift left, "\R"] \arrow[rd, shift left, "\R"] \arrow[l, shift left, "\R^*"] \\
			& & \R^2_4 \arrow[lu, shift left, "\R^*"]
		\end{tikzcd},
	\end{equation}
	and
	\begin{equation*}
		c_2 = \sum_{\alpha\in \overline{Q}_1^2}\sgn(\alpha)c_\alpha.
	\end{equation*}
	The subscripts in (\ref{eq - double quiver of C_3}) and (\ref{eq - double quiver of D_4}) denotes the positions in $S^1$ and $S^2$ respectively.
	
	By Lemma~\ref{Lemma - Pi(AxB) = Pi(A)xPi(B)} we know that $\Pi(T(S^1)\otimes_\K T(S^2)) = \Pi(S^1)\totimes_\K \Pi(S^2)$ and using Proposition~\ref{Proposition - diagonal tensor product between species} we can describe $\Pi(T(S^1)\otimes_\K T(S^2))$ as the species $T(S)/R$, where $S$ given by the diagram
	\begin{equation}\label{eq - quiver of C_3 x D_4}
		\begin{tikzcd}[row sep = 10, column sep = 50]
			& & \C_{13} \arrow[ddr] \\
			\C_{11} \arrow[r] \arrow[ddr] & \C_{12} \arrow[ur] \arrow[r] \arrow[ddr] & \C_{14} \arrow[ddr] \\
			& & & \R_{23} \arrow[ddr] \arrow[llu, dotted] \\
			& \R_{21} \arrow[r] \arrow[ddr] & \R_{22} \arrow[ur] \arrow[r] \arrow[ddr] \arrow[lluu, dotted] & \R_{24} \arrow[ddr] \arrow[lluu, dotted] \\
			& & & & \R_{33} \arrow[llu, dotted] \\
			& & \R_{31} \arrow[r] & \R_{31} \arrow[ur] \arrow[r] \arrow[lluu, dotted] & \R_{34} \arrow[lluu, dotted]
		\end{tikzcd}
	\end{equation}
	Here the dotted lines represents the $\star$-degree $1$ part of $M$ in $S=(D, M)$. The bimodule associated to an arrow $\alpha$ in (\ref{eq - quiver of C_3 x D_4}) is $\C$ if either the source or the target of $\alpha$ is $\C$, otherwise it is $\R$. For each arrow in (\ref{eq - quiver of C_3 x D_4}) we have a relation. More explicitly, the relations are given by
	\begin{equation*}
		R = \langle c_1\otimes_\K M_1^2, M_1^2\otimes_\K c_2, [\alpha\otimes 1_{D^2}, 1_{D^1}\otimes \alpha'] \mid \alpha\in M_0^1, \alpha'\in M_0^2 \rangle.
	\end{equation*}
	
	We can describe the Nakayama automorphism of $\Pi(T(S^1)\otimes_\K T(S^2))$ by using Theorem~\ref{Theorem - description of nakayama automorphism true} together with Proposition~\ref{Proposition - Nakayama automorphism of tensor product of algebras}. Let $\gamma_i$ be the Nakayama automorphism of $\Pi(S^i)$ for each $i\in \{1, 2\}$. Then by Proposition~\ref{Proposition - Nakayama automorphism of tensor product of algebras} $\gamma = \gamma_1\totimes \gamma_2$ is the Nakayama automorphism for $\Pi(S^1\totimes_\K S^2)$. Recall that
	\begin{equation*}
		\gamma_i(y_\alpha^k) = \begin{cases}
		y^k_{\sigma(\alpha)}, & \mbox{if }\alpha\in Q_1 \\
		\sgn(\sigma(\alpha))y^k_{\sigma(\alpha)}, & \mbox{if }\alpha\not\in Q_1
		\end{cases},
	\end{equation*}
	for each $i\in \{1, 2\}$. Since the Nakayama permutation is trivial for both $C_3$ and $D_4$ by Theorem~\ref{Theorem - Nakayama permutation} we get that $\gamma = \id_{\Pi(T(S^1)\otimes_\K T(S^2))}$.
\end{myex}

\newpage
\nocite{*}
\bibliographystyle{alpha}
\bibliography{References.bib}
\end{document}